\documentclass[a4paper,12pt]{report}
\usepackage[left=2.5cm,top=2.5cm,right=2.5cm,bottom=2.5cm]{geometry}
\usepackage{amsmath}
\usepackage{amsfonts}
\usepackage{tikz-cd}
\usepackage{mdframed}
\usepackage{graphicx,color}
\usepackage{xcolor}
\usepackage{placeins}
\usepackage{fancyhdr}
\usepackage[Conny]{fncychap}
\ChTitleVar{\huge\sffamily\bfseries}
\ChNumVar{\Huge\sffamily\bfseries\color{black}}
\ChNameVar{\centering\Huge\sffamily \bfseries}
\usepackage{titlesec}

\usepackage{microtype}
\usepackage{etoolbox}
\usepackage{quiver}
\makeatletter
\newcommand{\backvec}[1]{\reflectbox{$\vec{\reflectbox{\!$#1$}}$}}
\usepackage{float}

\usepackage[refpage]{nomencl}

\makenomenclature

\usepackage{lscape}
\usepackage{pdfpages}
\patchcmd{\chapter}{\if@openright\cleardoublepage\else\clearpage\fi}{}{}{}
\makeatother
\usepackage{soul}
\usepackage{cancel}
\usetikzlibrary{babel}
\usepackage{quiver}
\usepackage{tikz-cd}
\usepackage{graphicx}
\usepackage{tikz}
\usepackage{lipsum}
\usepackage{amsthm}

\usepackage{mathtools}
\usepackage{framed}
\usepackage{hyperref}
\usepackage{enumerate}
\theoremstyle{definition}
\title{Algebraic Topology Without Open Sets: A Net Approach to Homotopy Theory in Limit Spaces}
\author{Rodrigo Santos Monteiro }

\date{\today}
\usepackage{xcolor}
\hypersetup{
	colorlinks=true,
	linkcolor=blue,
	filecolor=magenta,      
	urlcolor=red,
}

\newcommand{\addsymbol}[1]{\hfill \textcolor{gray}{#1}}

\AtEndEnvironment{axiom}{\addsymbol{$\diamondsuit$}}
\AtEndEnvironment{definition}{\addsymbol{$\clubsuit$}}
\AtEndEnvironment{exercise}{\addsymbol{$\triangle$}}
\AtEndEnvironment{example}{\addsymbol{$ \blacktriangleleft$}}
\AtEndEnvironment{remark}{\addsymbol{$\bigstar$}}

\newtheorem{definition}{Definition}[section]
\newtheorem{theorem}{Theorem}[section]
\newtheorem{lemma}{Lemma}[section]

\newtheorem{proposition}{Proposition}[section]
\newtheorem{example}{Example}[section]
\newtheorem{corollary}{Corollary}[section]
\newtheorem{remark}{Remark}[section]

\usepackage{mdframed} 

\newmdenv[
linecolor=blue, 
linewidth=3pt, 
leftline=true, 
rightline=false, 
topline=false, 
bottomline=false,
innertopmargin=0.5\baselineskip, 
innerbottommargin=0.5\baselineskip
]{defstyle}

\newmdenv[
linecolor=red, 
linewidth=3pt, 
leftline=true, 
rightline=false, 
topline=false, 
bottomline=false,
innertopmargin=0.5\baselineskip, 
innerbottommargin=0.5\baselineskip
]{propstyle}

\newmdenv[
linecolor=orange, 
linewidth=3pt, 
leftline=true, 
rightline=false, 
topline=false, 
bottomline=false,
innertopmargin=0.5\baselineskip, 
innerbottommargin=0.5\baselineskip
]{theostyle}

\newmdenv[
linecolor=gray, 
linewidth=3pt, 
leftline=true, 
rightline=false, 
topline=false, 
bottomline=false,
innertopmargin=0.5\baselineskip, 
innerbottommargin=0.5\baselineskip
]{lemstyle}

\newmdenv[
linecolor=green, 
linewidth=3pt, 
leftline=true, 
rightline=false, 
topline=false, 
bottomline=false,
innertopmargin=0.5\baselineskip, 
innerbottommargin=0.5\baselineskip
]{corstyle}
\let\olddefinition\definition
\renewenvironment{definition}{\begin{defstyle}\olddefinition}{\end{defstyle}}

\let\oldproposition\proposition
\renewenvironment{proposition}{\begin{propstyle}\oldproposition}{\end{propstyle}}

\let\oldtheorem\theorem
\renewenvironment{theorem}{\begin{theostyle}\oldtheorem}{\end{theostyle}}

\let\oldlemma\lemma
\renewenvironment{lemma}{\begin{lemstyle}\oldlemma}{\end{lemstyle}}

\let\oldcorollary\corollary
\renewenvironment{corollary}{\begin{corstyle}\oldcorollary}{\end{corstyle}}

\usepackage{helvet}

\usepackage[math]{blindtext}
\newcommand{\Top}{\textsc{Top}}
\newcommand{\Lim}{\textsc{Lim}}
\newcommand{\Groupoid}{\textsc{Groupoid}}

\newcommand{\apair}[1]{\left\langle #1\right\rangle}

\definecolor{mycolor}{rgb}{0.1, 0.6, 0.3}

\begin{document}

\pagestyle{empty}


\begin{titlepage}
	\centering
	
	\vspace*{2cm}
	
	\vspace{4cm}
    \begin{tikzpicture}
	\draw[cyan, line width=2mm] (-8, 1.5) -- (-8, 2) -- (-6.5, 2);
	\draw[red!70!black, line width=2mm] (8, -1.5) -- (8, -2) -- (6.5, -2);
	
	\node[text width=16cm, align=center] at (0, 0) {
		{\Huge \textbf{Algebraic Topology without open sets:}}\\[2em]
		{\Large \textit{A net approach to homotopy theory in limit spaces}}
	};
\end{tikzpicture}

\vspace{2cm}
	
	{\Large Rodrigo Santos Monteiro} \\
	
	\vspace{10cm}
	
	{\large Universidade Estadual de Santa Cruz} \\
	{\large \today}

\end{titlepage}
\newpage

\

\newpage
\begin{center}
	\Large{Algebraic Topology without open sets: A net approach to homotopy theory in limit spaces}

\vspace{4cm}

\large{Rodrigo Santos Monteiro}
\end{center}

\vspace{4cm}

\begin{flushright}
\begin{minipage}{10cm}
	Final Course Work presented to the Universidade Estadual de Santa Cruz for the degree of Bachelor of Mathematics.
	
	\vspace{0.5cm}
	\textbf{Advisor}: Prof. Dr. Renan Maneli Mezabarba.
	
\end{minipage}
\end{flushright}

\vspace{8cm}

\begin{center}
\textbf{Ilhéus - Bahia} \\
\today
\end{center}

\newpage

\thispagestyle{empty}

\begin{center}
	\Large{Algebraic Topology without open sets: A net approach to homotopy theory in limit spaces}

	\vspace{2.2cm}
	
	\large{Rodrigo Santos Monteiro}
\end{center}

\vspace{2.2cm}

\hfill

\begin{flushright}
	
	\begin{minipage}{8.5 cm}

		\begin{small} 
			\setlength{\baselineskip}{\baselineskip}

			{Final Course Work presented to the Universidade Estadual de Santa Cruz for the degree of Bachelor of Mathematics.
			}\\

			\vspace*{1.0 cm}
			
			{\textbf{{\large Examination Board}:}\\
				
				\vspace*{1.0 cm}
				
				\rule{\linewidth}{.1 mm}\\
				{ \centering Prof. Dr. Renan Maneli Mezabarba\\
				\ \ \ \ \ \ \ \ \ \ \ \  \ \ \ \  \ \ \ \ \ \ \ \ \ \ \ \ \ \  UESC
				}
				
				\vspace*{1.0 cm}
				
				\rule{\linewidth}{.1 mm}\\
				{\centering Prof. Dr. Germán Ignacio Gomero Ferrer\\
					\ \ \ \ \ \ \ \ \ \ \ \  \ \ \ \  \ \ \ \ \ \ \ \ \ \ \ \ \ \ 	UESC
				}
				
				\vspace*{1.0 cm}
				
				\rule{\linewidth}{.1 mm}\\
				{\centering Prof. Dr. Weslem Liberato Silva\\
					 	\ \ \ \ \ \ \ \ \ \ \ \  \ \ \ \  \ \  \ \ \ \ \ \ \ \ \ \ \ \   UESC}
			}

		\end{small} 

	\end{minipage}
	
\end{flushright}

\newpage

\
\newpage

\chapter*{Abstract}

\thispagestyle{empty}

Convergence spaces are a generalization of topological spaces. The category of convergence spaces is well-suited for Algebraic Topology, one of the reasons is the existence of exponential objects provided by continuous convergence. In this work, we use a net-theoretic approach to convergence spaces. The goal is to simplify the description of continuous convergence and apply it to problems related to homotopy theory. We present methods to develop the basis of homotopy theory in limit spaces, define the fundamental groupoid, and prove the groupoid version of the Seifert-van Kampen Theorem for limit spaces.

\vspace{0.2cm}

\begin{flushleft}
	\textbf{Keywords:} Convergence spaces, Nets, Homotopy Theory
\end{flushleft}

\newpage

\chapter*{List of Symbols}

\vspace{-2cm}
\begin{description}
	\item $\mathcal{P}(X)$: \ \  family of subsets of $X$, ~\pageref{symbol:PX}.
	\item $\textsc{Fil}^*(X)$: \ \  family of proper filters on a set $X$, ~\pageref{symbol:Fil}.
	\item $\mathcal{N}_x$: \ \ family of neighborhoods of a point $x$, ~\pageref{symbol:Nx}.
	\item $\langle x_n\rangle_n$ or $\langle x_n\rangle_{n\in\mathbb{N}}$:\ \  sequence $\mathbb{N}\to X$ such that $n\mapsto x_n$, ~\pageref{symbol:sequence}.
	\item $\langle x_n\rangle_n^{\uparrow}$: \ \ filter generated by a sequence $\langle x_n\rangle_n$, ~\pageref{symbol:sequenceFilter}.
	\item $\mathcal{B}^{\uparrow}$: \ \ filter generated by a family $\mathcal{B}$, ~\pageref{symbol:familyFilter}.
	\item $\mathcal{F}\to x$: \ \ a filter $\mathcal{F}$ converges to $x$, ~\pageref{symbol:filterConverge}.
	\item $\textsc{Nets}(X)$: \ \ class of nets on a set $X$, ~\pageref{symbol:nets}.

	\item $\varphi\to x$: \ \ a net $\varphi$ converges to $x$, ~\pageref{symbol:netConverge}.
	\item $\varphi^{\uparrow}$: \ \ induced filter by a net $\varphi$, ~\pageref{symbol:netFilter}.
		\item $\langle a,b\rangle$: \ \ ordered pair with coordinates $a$ and $b$, ~\pageref{symbol:pair}.
	\item $\hbox{dom}(\varphi)$: \ \ domain of a net $\varphi$,~\pageref{symbol:domainNet}.
		\item $\lim_\tau$: \ \ induced preconvergence by a topology $\tau$ , ~\pageref{symbol:limTau}.
	\item $\to_{\mathbb{R}}$: \ \ usual convergence on the real line, ~\pageref{symbol:convR}.
	\item $\to_{\mathbb{R}^2}$: \ \  usual convergence on the plane, ~\pageref{symbol:convR2}.
		\item $\mathcal{C}(X,Y)$: \ \ set of continuous functions $X\to Y$, ~\pageref{symbol:continuousFunctions}.
	\item $1_X$: \ \ identity function on $X$, ~\pageref{symbol:identity}.
	\item $\textsc{PrConv}$: \ \ category of preconvergence spaces, ~\pageref{symbol:PrConv}.
	\item $\textsc{Conv}$: \ \ category of convergence spaces~\pageref{symbol:Conv}.
	\item $\textsc{Lim}$: \ \ category of limit spaces, ~\pageref{symbol:Lim}.
	\item $\textsc{Top}$: \ \ category of topological spaces, ~\pageref{symbol:Top}.

	\item $L\geq L'$:  \ \ $L$ is a preconvergence stronger than $L'$, ~\pageref{symbol:preconvergenceOrder}.
	\item $\bigvee_{i\in \mathcal{I}} f_i^{-1}L_i$:  \ \ initial preconvergence induced by functions $f_i$, ~\pageref{symbol:initialPreconv}.
	\item $\prod_{i\in\mathcal{I}} X_i$: \ \ product of a family $\{X_i\}_{i\in\mathcal{I}}$, ~\pageref{symbol:product}.
	\item $\bigwedge_{i\in \mathcal{I}} f_iL_i$:  \ \ final preconvergence induced by functions $f_i$, ~\pageref{symbol:finalPreconv}.
	\item $\coprod_{i\in\mathcal{I}} X_i$: \ \ coproduct of a family $\{X_i\}_{i\in\mathcal{I}}$, ~\pageref{symbol:coproduct}.
	\item $\hbox{inh}_L(S)$: \ \ $L$-inherence of a subset $S$, ~\pageref{symbol:inherence}.
	\item $\mathcal{O}(L)$:  \ \ induced topology by a preconvergence $L$, ~\pageref{symbol:inducedTopology}.
	\item $\sqcup(L)$: \ \ limit modification of a convergence $L$, ~\pageref{symbol:limitMod}.
		\item $\langle x_d\rangle_d$ or $\langle x_d\rangle_{d\in\mathbb{D}}$: \ \  function $\mathbb{D}\to X$ such that $d\mapsto x_d$, ~\pageref{symbol:function}.

	\item $\hbox{adh}_L(S)$:  \ \ $L$-adherence of a subset $S$, ~\pageref{symbol:adherence}.
	\item $a^{\uparrow}$:  \ \ set of all elements greater than $a$ in a directed set, ~\pageref{symbol:greaterSet}.
	\item $\gamma*\gamma'$:  \ \ concatenation of paths $\gamma$ and $\gamma'$, ~\pageref{symbol:concatenation}.
	\item $\Pi(X)$: \ \  fundamental groupoid of a limit space $X$, ~\pageref{symbol:groupoid}.
	\item $\pi_1(X,x_0)$:  \ \ fundamental group of a pointed limit space $\langle X,x_0\rangle$, ~\pageref{symbol:fundGroup}.
\end{description}
\newpage 

\

\newpage

\tableofcontents
\thispagestyle{empty}
\newpage

\newpage

\

\newpage

\chapter*{Introduction}

\vspace{-1cm}
Algebraic Topology is one of the most important branches of mathematics which uses algebraic tools to study topological spaces. The basic goal is to find algebraic invariants that classify topological spaces up to homeomorphism. The most important of these invariants are
homotopy groups, homology groups, and cohomology groups. The concept of homotopy is a formulation of the intuitive idea of a continuous deformation from one geometrical configuration to other in the sense that this concept formalizes the idea of continuous deformation of a continuous function. On the other hand, the concept of
homology formalizes the intuitive idea of a curve bounding an area or a surface bounding a volume. Cohomology is the dual concept of homology. But what does Algebraic Topology without open sets mean? In Analysis, convergence is fundamental, but perhaps not as much for Algebraic Topology. The idea of this work is to discuss homotopy theory using convergence of nets in spaces more general than topological spaces, the convergence spaces. Generally, in the literature, convergence spaces are defined as spaces equipped with some notion of convergence of filters. The idea is to associate a filter with its set of limits. In our context, these are spaces equipped with some notion of convergence of nets. This approach is not new, as it was also proposed by \v{C}ech~\cite{cech}, Schechter~\cite{schechter}, Kelley~\cite{Kelley}, Kat\v{e}tov~\cite{Katetov}, Poppe~\cite{poppe}, and Pearson~\cite{Pearson1988}. Choquet introduced convergence structures, such as pseudotopologies and pretopologies, in~\cite{Choquet}. Several papers and books have addressed the topic, including those by Cech~\cite{cech}, Binz~\cite{binz}, Gähler~\cite{gahler}, Schechter~\cite{schechter}, Beattie and  Butzmann~\cite{ref4}, Preuss~\cite{ref5}, and more recently, Nel~\cite{nel}, Dolecki and Mynard~\cite{ref3}, and Dolecki~\cite{schechter1996handbook}.

Topological spaces are convergence spaces, but the converse is not true, as convergence structures are defined very generally and may not always be induced by a topology. Then the category of convergence spaces expands the category of topological spaces. One major advantage of convergence spaces over topological spaces lies in the space of continuous functions.
Not always exists a suitable topology that renders the space of continuous functions into a well-behaved topological space, in the sense of functioning categorically as a function space should. However, for convergence spaces, such a structure exists: the continuous convergence. Simply put, the category of convergence spaces possesses exponential objects, making it a ``convenient" category for Algebraic Topology. The term ``convenient category of topological spaces", introduced by Brown~\cite{Brown0} and popularized by Steenrod~\cite{steenrod}, refers to any category of topological spaces that is sufficiently well-behaved for Algebraic Topology. Recently, Rieser~\cite{rieserTop,riesernew}, Dossena~\cite{dossena}  and Marroquín~\cite{marroquin} have addressed topics in Algebraic Topology in the context of convergence spaces. Their focus is on pseudotopological spaces and closure spaces. The focus of this work is to show how to develop the basis of homotopy theory for limit spaces. In particular, we construct the fundamental groupoid of a limit space. In choosing algebra to model geometry there is a tendency to take groups, rings, fields, modules, etc. Motivated by Brown~\cite{BrownBook}, in this work we use groupoids. The difference between groupoids and groups is that in groupoids there is a partial multiplication which is defined under geometric conditions: two arrows can be compose if and only if the end point of one is the initial point of the other.

The Seifert-van Kampen Theorem is a powerful tool in Algebraic Topology, specifically in the study of fundamental groups of topological spaces. In its groupoid version, the theorem generalizes the classical result by allowing us to compute not just the fundamental group, but the fundamental groupoid, which encodes more detailed information about paths between multiple base points. The groupoid version of the theorem provides a more flexible framework, especially when dealing with spaces that are not path-connected or when it's beneficial to consider multiple base points. We present the groupoid version of the Seifert-van Kampen Theorem for limit spaces.

This work extends the author's previous paper~\cite{absolutepaper}, and we adopt terminology and notation from a textbook that is currently in draft. We assume the reader has a basic understanding of Set Theory, General and Algebraic Topology, and Category Theory. For sake of completeness, some basic notions of Category Theory can be found in  Leinster~\cite{tom} and Roman~\cite{roman}, while for Set Theory, we recommend Jech~\cite{Jech2002}. Finally,  for General Topology, Engelking~\cite{Engelking} and for Algebraic Topology, Kammeyer~\cite{Kammeyer} and Vick~\cite{vick}.

	This work is organized as follows: Chapter \ref{chap1} present the basic terminology related to filters and nets. In Chapter \ref{chap2} we introduce the convergence spaces, discuss final and initial structures, modifiers, and some topological notions in the context of convergence spaces. Chapter \ref{chap3} develop the basis of homotopy theory in limit spaces and construct the Fundamental Groupoid. In Chapter \ref{chap4}, we present the Seifert–Van Kampen theorem for limit spaces. Finally, in Chapter \ref{chap5}, we discuss the presented results and show the direction in which the research will proceed.

\pagestyle{myheadings}
\setcounter{page}{1}
\addcontentsline{toc}{chapter}{Introduction}
\newpage

\chapter{Basic terminology related to filters and nets}

\titleformat{\section}[block]
{\bfseries\Large\vrule width 2pt\hspace{1em}} 
{\thesection}{0.5em}
{} 
\titlespacing{\section}{0pt}{*4}{*4}

\label{chap1}

In this chapter, we introduce the basic terminology related to filters and nets. Our principal references are Dolecki~\cite{schechter1996handbook}, Dolecki and Mynard~\cite{ref3} and Schechter~\cite{schechter}. Filters and nets are important concepts in General Topology that help study convergence in topological spaces, generalizing the idea of sequences for situations where sequences are not sufficient, such as in spaces that are not sequentially compact or where the topological structure is more complex. Nets were introduced by E. H. Moore and H. L. Smith~\cite{MooreSmith} in 1922 to address convergence problems in topological spaces, while filters were introduced by Henri Cartan in 1937 as a way to study convergence without relying on ordinal number theory. In Section \ref{sec1.2}, we will see that, in a certain way, filters and nets are equivalent. 
\section{Filters}

\label{symbol:PX}
\label{symbol:Fil}
\begin{definition}
	\label{filters}
	Let $X$ be a set and $\mathcal{F}\subseteq\mathcal{P} (X)$ be a nonempty family of subsets of $X$. We say that $\mathcal{F}$ is a \textbf{filter} on $X$ if
	\begin{enumerate}
		\item $A\cap B\in\mathcal{F}$ whenever $A,B\in  \mathcal{F}$
		\item $C\in\mathcal{F}$ whenever $A\subseteq C$ and $A\in\mathcal{F}$
	\end{enumerate}
	A filter $\mathcal{F}$ is said to be \textbf{proper} if $\emptyset\notin\mathcal{F}$. We denote the family of proper filter on $X$ by $\textsc{Fil}^*(X)$.
\end{definition}

\begin{remark}
	In this work, we will only deal with proper filters. Initially, we could allow filters to be non-proper, but in that case, we would have to consider the empty net to obtain the equivalence between filters and nets that we will see in Section \ref{sec1.2}.
\end{remark}
More generally, a filter  is a special subset of a partially ordered set, describing ``large" or ``eventual" elements. Filters appear in both Order and Lattice Theory, as well as in Topology, where they emerged. Additionally, filters give an idea of approximation in some sense.

\begin{figure}[H]
	\centering
	\includegraphics[width=8cm]{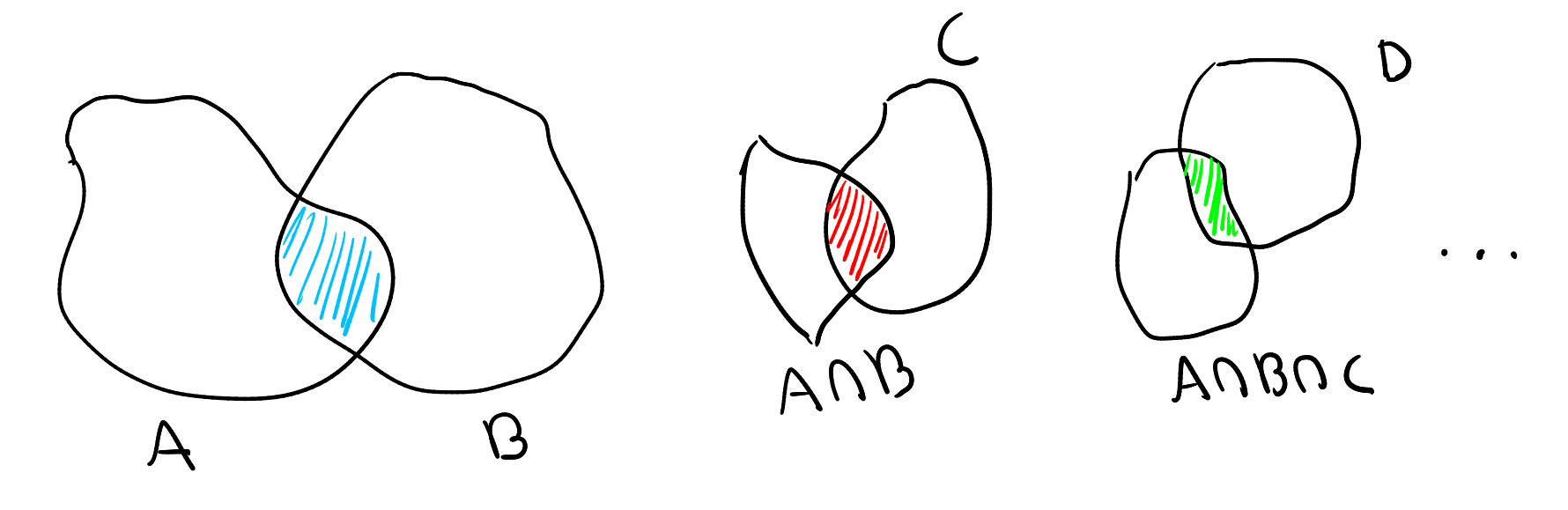}
	\label{fig1}
	\caption{Suppose that $A,B,C$ and $D$ are elements of a filter. Condition 1 in Definition \ref{filters} allows us to find increasingly smaller sets the filter, which gives us an idea of approximation or convergence.}
\end{figure}

\begin{example}
	\label{symbol:Nx}
	Let $ X$ be a topological space. For each $x\in X$, the family $\mathcal{N}_x$ of neighborhoods of $x$, is a filter. Indeed, if $A,B\in\mathcal{N}_x$ there are open sets $U,V\subseteq X$ such  that $x\in U\cap V$, $U\subseteq A$ and $V\subseteq B$. Notice that $U\cap V\subseteq A\cap B$ and, since $U\cap V$ is open, it follows that $A\cap B\in\mathcal{N}_x$. Let $A\in\mathcal{N}_x$ be a neighborhood of $x$ such that $A\subseteq C$, there is an  open set $U\subseteq X$ such that $x\in U$ and $U\subseteq A$. Notice that $U\subseteq A\subseteq C$ and, therefore, $C\in\mathcal{N}_x$. \label{ex1.1}
\end{example}

\begin{example}
	\label{symbol:sequence} \label{symbol:sequenceFilter}
	For a sequence $\langle x_n\rangle_n$ in $X$, we associate the filter $\langle x_n\rangle_n^{\uparrow}$ such that $A\in \langle x_n\rangle _n^{\uparrow}$ if and only if contains a subset of the form $\{x_n:n\geq n_0\}$ for some $n_0\in\mathbb{N}$, which we refer to a tail set of the sequence. Let us show that $\langle x_n\rangle_n^{\uparrow}$ is indeed a filter. This family is nonempty because contains the tails of the sequence. Given $A,B\in\langle x_n\rangle_n^{\uparrow}$, there are $n_0,n_1\in\mathbb{N}$ such that $\{x_n:n\geq n_0\}\subseteq A$ and $\{x_n:n\geq n_1\}\subseteq B$. Considering $n_2=\hbox{max}\{n_0,n_1\}$ it happens that $\{x_n:n\geq n_2\}\subseteq A\cap B$. This means that $A\cap B\in\langle x_n\rangle_n^{\uparrow}$.
	 Now, if $A\in\langle x_n\rangle_ n^{\uparrow}$ and $A\subseteq B$, there is $n_0\in\mathbb{N}$ such that $\{x_n:n\geq n_0\}\subseteq A$. Then $\{x_n:n\geq n_0\}\subseteq B$. It follows that $B\in\langle x_n\rangle_n^{\uparrow}$.
	\label{ex1.2}
\end{example}

In Example \ref{ex1.2}, the tails of the sequence are the witnesses to of membership in the filter, and in Example \ref{ex1.1}, the open sets containing the point fulfill the same role. This motivates the definition of a filter basis.

\begin{definition}
	Let $\mathcal{F}$ be a filter and $\mathcal{B}\subseteq\mathcal{F}$ be a subfamily. We say that $\mathcal{B}$ is a \textbf{basis} of $\mathcal{F}$ if for every $B\in\mathcal{F}$ there is $A\in\mathcal{B}$ such that $A\subseteq B$. In this case we say that $\mathcal{B}$ generates $\mathcal{F}$ and we write $\mathcal{F}=\mathcal{B}^{\uparrow}$. \label{symbol:familyFilter}
\end{definition}

\begin{example}
	Let $X$ be a topological space. We denote by $\mathcal{T}_x$ the family of all open sets containing $x\in X$. Notice that $\mathcal{T}_x$ is a basis of $\mathcal{N}_x$. Indeed, recall that  $V\in\mathcal{N}_x$ if and only if there is $A\in\mathcal{T}_x$ such that $A\subseteq V$. 
\end{example}

\begin{proposition}
	Let $\mathcal{B}$ be a nonempty family of subsets of a set $X$. The family $$\mathcal{B}^{\uparrow}=\{A\subseteq X:\exists B\in\mathcal{B}\hbox{ such that } B\subseteq A\}$$ is a filter if and only if for every $B,C\in\mathcal{B}$ there is $D\in\mathcal{B}$ such that $D\subseteq B\cap C$.
\end{proposition}
\begin{proof}
	The first implication holds because $\mathcal{B}\subseteq \mathcal{B}^{\uparrow}$ and $\mathcal{B}^{\uparrow}$ is a filter. For the converse, if $F,G\in\mathcal{B}^{\uparrow}$ there are $A,B\in\mathcal{B}$ such that $A\subseteq F$ and $B\subseteq G$. There is $C\in\mathcal{B}$ such that $C\subseteq A\cap B$. Since $A\cap B\subseteq F\cap G$, it follows that $C\subseteq F\cap G$. This means that $F\cap G\in \mathcal{B}^{\uparrow}$. This family is clearly closed upwards. Then $\mathcal{B}^{\uparrow}$ is a filter.
\end{proof}

The previous proposition allows us to generate a filter from a family of subsets, and it is straightforward  to verify that the filter generated corresponds to the intersection of all filters containing this family.



\begin{definition}
	A proper filter $\mathfrak{u}$ is an\textbf{ ultrafilter} if it is a maximal filter, that is, if $\mathfrak{u}=\mathcal{F}$ whenever $\mathcal{F}$ is a proper filter such that $\mathfrak{u}\subseteq\mathcal{F}$.
\end{definition}

Ultrafilters play a significant role in various branches, particularly in Set Theory, Topology, and Model Theory. In Topology, they are closely related to compactness and other fundamental topological properties.

\begin{proposition}
	For a proper filter $\mathcal{F}\in\textsc{Fil}^*(X)$ the following are equivalent
	\begin{enumerate}
		\item $\mathcal{F}$ is an ultrafilter.
		\item For every $A\subseteq X$, either $A\in\mathcal{F}$ or $X\setminus A\in\mathcal{F}$.
		\item For every $A,B\subseteq X$ it happens that $A\in\mathcal{F}$ or $B\in\mathcal{F}$ whenever $A\cup B\in\mathcal{F}$.
	\end{enumerate}
\end{proposition}
\begin{proof}

	$(1\implies 2)$ Let $A\subseteq X$ be a subset of $X$ such that $X\setminus A\notin \mathcal{F}$. Note that $\mathcal{G}=(\mathcal{F}\cup\{A\})^{\uparrow}$ is a filter and $\mathcal{F}\subseteq \mathcal{G}$. Since $\mathcal{F}$ is an ultrafilter, it happens that $\mathcal{F}=\mathcal{G}$. This means that $A\in\mathcal{F}$.
	
	$(2\implies 3)$ Let $A$ and $B$ be subsets of $X$ such that $A\cup B\in\mathcal{F}$ and $A,B\notin\mathcal{F}$. By hypothesis, it happens that  $X\setminus A\in\mathcal{F}$ and $X\setminus B\in\mathcal{F}$. Notice that $$(X\setminus A)\cap (X\setminus B)=X\setminus (A\cup B)\in\mathcal{F}$$This is a contradiction, because in this case $\emptyset=(A\cup B)\cap X\setminus(A\cup B)\in\mathcal{F}$ and $\mathcal{F}$ is a proper filter. It follows that $A\in\mathcal{F}$ or $B\in\mathcal{F}$.
	
	$(3\implies 1)$ Let $\mathcal{G}\in\textsc{Fil}^*(X)$ be a filter such that $\mathcal{F}\subseteq \mathcal{G}$. Suppose that $\mathcal{G}\neq\mathcal{F}$. There is $A\in\mathcal{G}$ such that $A\notin\mathcal{F}$. By hypothesis $X\setminus A\in\mathcal{F}\subseteq \mathcal{G}$. Then $A\in\mathcal{G}$ and $ X\setminus A\in\mathcal{G}$, in this case $\emptyset=A\cap X\setminus A\in\mathcal{G}$, a contradiction. This proves that $\mathcal{G}=\mathcal{F}$. Then $\mathcal{F}$ is an ultrafilter.
\end{proof}

\begin{example}
	For a point $x\in X$, the family $$\mathfrak{u}_x=\{A\subseteq X:x\in A\}$$
	is an ultrafilter, called the principal ultrafilter generated by $x$. Notice that this ultrafilter is generated by a constant sequence.
\end{example}

\begin{definition}
	Let $X$ be a topological space. We say that a filter $\mathcal{F}\in\textsc{Fil}^*(X)$ converges to a point $x\in X$ if $\mathcal{N}_x\subseteq\mathcal{F}$. In this case we write $\mathcal{F}\to x$. \label{symbol:filterConverge}
\end{definition}

\begin{proposition}
	\label{compact}
	A topological space is compact if and only if every ultrafilter on it converges.
\end{proposition}
\begin{proof}
	See~\cite{Engelking}.
\end{proof}

Recall that a family of subsets has the finite intersection property if every finite intersection of sets in the family is nonempty.  The following proposition shows the relationship between this property and ultrafilters.

\begin{lemma}
	\label{ultrafilter}
	Every family that has the finite intersection property is contained in an ultrafilter.
\end{lemma}
\begin{proof}
	See~\cite{schechter}.
\end{proof}

Given a filter $\mathcal{F}$ on $X$ and a function $f:X\to Y$, it is convenient to have a way to use this function to ``push" the filter $\mathcal{F}$ to the set $Y$. Naturally, we use the images of the elements of the filter $\mathcal{F}$ to generate a filter on $Y$. The filter image of $\mathcal{F}$ under $f$ is the filter $f(\mathcal{F})=f\mathcal{F}^{\uparrow}$,
where $f\mathcal{F}=\{f[A]:A\in\mathcal{F}\}$.
On the other hand, considering a filter $\mathcal{G}$ in $Y$ such that $f^{-1}[G]\neq\emptyset$ for all $G\in\mathcal{G}$, the idea now is to ``pull" the filter $\mathcal{G}$ to the set $X$ using the function $f$.  The preimage filter of $\mathcal{G}$ under $f$ is the filter $f^{-1}(\mathcal{G})=f^{-1}\mathcal{G}^{\uparrow}$,
where $f^{-1}\mathcal{G}=\{f^{-1}[A]:A\in\mathcal{G}\}$.
But what is the interest in pushing and pulling filters? We can characterize continuity in topological spaces through filter convergence, as follows.

\begin{proposition}
	\label{filter}
Let $X$ and $Y$ be topological spaces.	A function $f:X\to Y$ is continuous if and only if $f(\mathcal{F})\to f(x)$ whenever $\mathcal{F}$ is a filter in $X$ such that $\mathcal{F}\to x$.
\end{proposition}

\begin{proof}
	Let $f:X\to Y$ be a continuous function and $\mathcal{F}$ be a filter in $X$ such that $\mathcal{F}\to x$. It happens that $\mathcal{N}_x\subseteq\mathcal{F}$. Let $V\in\mathcal{N}_{f(x)}$ be a neighborhood of $f(x)$, there is an open set $U\subseteq V\subseteq Y$ such that $f(x)\in U$. Since $f$ is continuous, $f^{-1}[U]$ is an open set such that $x\in f^{-1}[U]$ and $f^{-1}[U]\subseteq f^{-1}[V]\subseteq X$. It follows that $f^{-1}[V]\in\mathcal{N}_x\subseteq \mathcal{F}$. Then $f[f^{-1}[V]]\subseteq V\in f(\mathcal{F})$, that is, $\mathcal{N}_{f(x)}\subseteq f(\mathcal{F})$. This means that $f(\mathcal{F})\to f(x)$. Conversely, since $\mathcal{N}_x\to x$, it happens that $f(\mathcal{N}_x)\to f(x)$. Then $\mathcal{N}_{f(x)}\subseteq f(\mathcal{N}_x)$. But this inclusion means that $f$ is continuous at $x\in X$. Since 
	$x$ is arbitrary, it follows that 
	$f$ is continuous.
\end{proof}
\section{Nets}
\label{sec1.2}

\begin{definition}
	Let $\langle \mathbb{D},\leq\rangle$ be a preorder, i.e., $\leq$ is reflexive and transitive. We say that $\mathbb{D}$ is a \textbf{directed set} if for every $a,b\in\mathbb{D}$ there is $c\in\mathbb{D}$ such that $a,b\leq c$. A \textbf{net} in $X$ is a function $\mathbb{D}\to X$. We denote by $\textsc{Nets}(X)$ the class of nets in $X$. \label{symbol:nets}
\end{definition}

\begin{remark}
	Note that $\textsc{Nets}(X)$ is not a set if $X$ is nonempty, but rather a proper class. Indeed, since any set is well ordered, a function $Y\to X$ can be seen as a net for any set $Y$. Then $\textsc{Nets}(X)$ is too large to be a set. For details, see~\cite{schechter}. 
\end{remark}

\begin{example}
	A total order $\langle \mathbb{T},\leq\rangle$ is a directed set. Then any function $\varphi:\mathbb{T}\to X$ is a net. In particular, as $\mathbb{N}$ is totally ordered, every sequence is a net.  Essentially, the difference between a net and a sequence is that the domain of a net is not always the set of natural numbers, but psychologically, we treat nets as sequences to gain intuition.
\end{example}

 Recall that in a topological space \( X \), a net \( \varphi: \mathbb{D} \to X \) converges to a point \( x \in X \) if for every neighborhood \( V \subseteq X \) of $x$ there is an index \( d' \in \mathbb{D} \) such that \( \varphi_d \in V \) for all \( d \geq d' \). We denote by $\varphi\to x$. In simpler terms, this means that every neighborhood of the point \( x \) contains a tail of the net, that is, a subset of the form $\{\varphi_d:d\geq d'\}$ for some $d'\in\mathbb{D}$. \label{symbol:netConverge}

\begin{figure}[H]
	\centering
	\includegraphics[width=5cm]{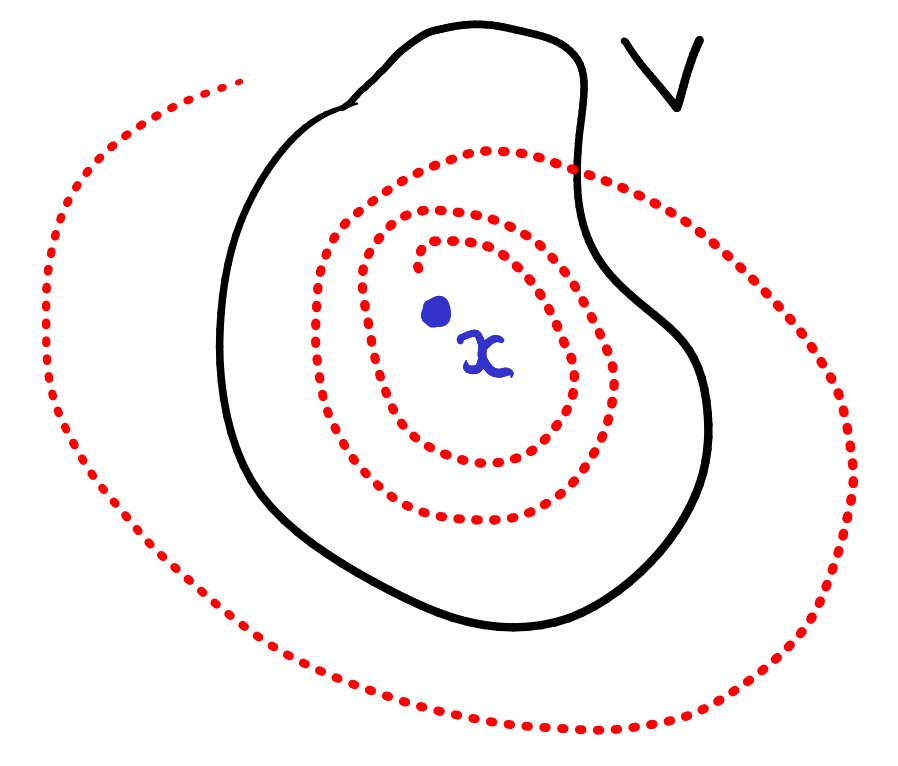}
	\caption{There is a moment from which all terms of the net are within the neighborhood. }
\end{figure}

\begin{example}
	In a Real Analysis course, we are introduced to the Riemann integral, which aims to calculate the area under a curve using Riemann sums. In this example, we generalize this type of integral and see how it can be viewed as the limit of a net. Given real numbers \( a, b \in \mathbb{R} \) such that \( a < b \). Denote by \(\mathbb{P}[a,b]\) the set of finite sequences \( P = \langle p_0, p_1, \ldots, p_n \rangle \) where \( p_0 = a \), \( p_n = b \), and \( p_i \leq p_{i+1} \) for \( 0 \leq i \leq n-1 \), with \( n \in \mathbb{N}^* \). We refer to \( P \) as an untagged partition of the interval \([a, b]\), and its norm is defined as 
	
	\[
	\|P\| = \max \{ p_i - p_{i-1} : 1 \leq i \leq n \}.
	\]
	For untagged partitions $P,Q\in\mathbb{P}[a,b]$ we say that $P$ is better than $Q$, and denote by $Q\leq P$, if $||P||\leq ||Q||$. Intuitively, one partition is better than another if the chosen points are closer together.
	
	\begin{figure}[H]
		\centering
		\includegraphics[width=9cm]{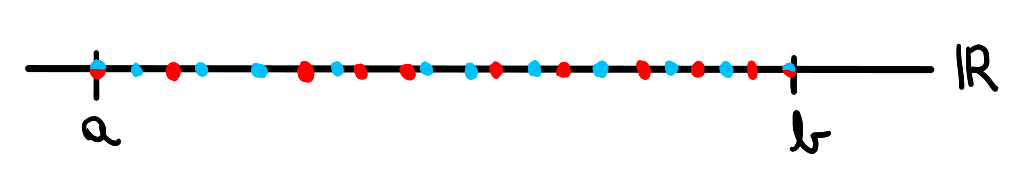}
		\caption{The partition in blue is better than the in red}
	\end{figure}
	A tag \( T \) of an untagged partition \( P = \langle p_0, p_1, \ldots, p_n \rangle \) is a finite sequence \( \langle t_1, \ldots, t_n \rangle \) such that \( p_i \leq t_i \leq p_{i+1} \) for \( 1 \leq i \leq n \). A tagged partition is a pair \( \langle P, T \rangle \), where \( P \) is an untagged partition of \([a, b]\) and \( T \) is a tag associated with the partition \( P \). We denote by \( \mathbb{P}^*[a, b] \) the set of tagged partitions of the interval \([a, b]\). The preorder in $\mathbb{P}^*[a,b]$ is such that $\langle P,T\rangle\leq \langle Q,T'\rangle$ if and only if $||Q||\leq ||P||$.
	
	\begin{figure}[H]
		\centering
		\includegraphics[width=9cm]{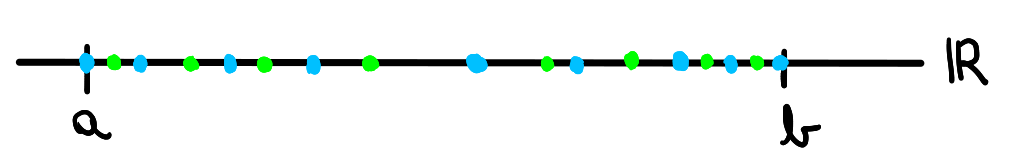}
		\caption{A tag in green of the partition in blue}
	\end{figure}
	Let \( X \) be a real vector space, and \( f: [a, b] \to X \) and \( \alpha: [a, b] \to \mathbb{R} \) be functions. For a tagged partition \( \langle P, T \rangle \in \mathbb{P}^*[a, b] \), we associated the Riemann-Stieltjes sum of \( f \) and \( \alpha \) over \( \langle P, T \rangle \)
	
	\[
	\sum_{\langle P, T \rangle} f, \alpha = \sum_{i=1}^n f(t_i) \cdot (\alpha(p_i) - \alpha(p_{i-1})).
	\]
	This defines a net
	\[
	\begin{array}{rcl}
		\int:\mathbb{P}^*[a, b] & \to & X \\ 
		\langle P, T \rangle & \mapsto & \sum_{\langle P, T \rangle} f, \alpha 
	\end{array}
	\]
	If $X$ have some notion of convergence or topology, when it exists, the limit of this net is called the Riemann-Stieltjes integral of \( f \) with respect to \( \alpha \) over the interval \([a, b]\). In  the case that $X=\mathbb{R}$ and $\alpha $ is the identity, this reduces to the Riemann integral.
\end{example}

\begin{remark}
The preorder on the directed set plays a crucial role in the convergence of a net. Consider the net \( \varphi: \mathbb{N}\setminus \{0\} \to \mathbb{R} \) defined by \( \varphi_n = (-1)^n \). Under the standard order on \( \mathbb{N}\setminus \{0\} \), this net does not converge. However, if we consider a different order on \( \mathbb{N}\setminus\{0\} \) such that \( n \leq m \) if and only if \( n \) divides \( m \), for any \( n \geq 2 \), it happens that \( 2 | n \). This means that \( \varphi_n = 1 \). Then, in this case, the net $\varphi$ converges to \( 1 \).
\end{remark}

The objective of the next three propositions is to demonstrate how the convergence of nets fully characterizes the open and closed sets of a topological space, as well as the continuity of functions. This understanding will enable us to generalize the concept of a topological space in Chapter \ref{chap2}.

\begin{proposition}
	\label{prop1.2.1}
	Let $ X $ be a topological space and $A\subseteq X$. The following are equivalent
	\begin{enumerate}
		\item $x\in\overline{A}$
		\item There is a net $\eta\in\textsc{Nets}(A)$ such that $\eta \to x$.
	\end{enumerate}
\end{proposition}
\begin{proof}
	
	($1\implies 2$) If $x\in \overline{A}$ it is an adherent point, for each neighborhood $V\in\mathcal{N}_x$ it happens that $V\cap A\neq\emptyset$. There is $x_V\in V\cap A$. The axiom of choice allows us to construct the net	$$\begin{array}{rcl}
		\Psi: \mathcal{N}_x & \to     & X \\ 
		V& \mapsto & 	x_V\\
	\end{array}$$
	where $\mathcal{N}_x$ is directed by reverse inclusion. Let $M$ be a neighborhood of $x$, if $A$ is also a neighborhood of $x$ such that $A\subseteq M$, we have $x_A\in M$. This proves that $\Psi\to x$ and $\Psi$ is a net in $A$ by construction.
	
	($2\implies 1$) Let $N$ be a neighborhood of $x$. Since there is a net  $\eta:\mathbb{D}\to A$ such that $\eta_d\to x$, there is $d'\in\mathbb{D}$ such that $\{\eta_d:d\geq d'\}\subseteq N$. It follows that $\eta_{d'}\in N\cap A$. This proves that $x\in \overline{A}$.
\end{proof}

\begin{proposition}
	Let $X$ be a topological space. A subset $A\subseteq X$ is open if and only if for every net $\varphi\in\textsc{Nets}(X)$ such that $\varphi\to x\in A$ there is $d'\in\hbox{dom}(\varphi)$ such that $\{\varphi_d:d\geq d'\}\subseteq A$.
	\label{open}
\end{proposition}
\begin{proof}
	The first implication follows from the definition of convergence of nets. We prove the converse by contrapositive. Suppose that $A\subseteq X$ is not open. There is $x\in A$ such that $x\notin \hbox{int}(A)$. This means that for every open set $V\subseteq X$ such that $x\in V$ we have $V\not\subseteq A$. It follows that there is $x_V\in V\cap (X\setminus A)$. Consider the net 

$$\begin{array}{rcl}
	\Psi: \mathcal{N}_x & \to     & X \\ 
	V& \mapsto & 	x_V\\
\end{array}$$ Let us show that $\Psi\to x$. If $A,M\in\mathcal{N}_x$ are neighborhoods of $x$ such that $A\subseteq M$, there is $U\in\mathcal{T}_x$ such that $U\subseteq A\subseteq M$.  It follows that $x_A\in V$. This proves that $\Psi$ converges to $x$. Note that $A$ does not contain tails of this net by construction.
\end{proof}

\begin{proposition}
	Let $ X$ and $Y$ be topological spaces. A function $f:X\to Y$ is continuous if and only if $f\circ \varphi\to f(x)$ whenever $\varphi\in\textsc{Nets}(X)$ is a net such that $\varphi\to x$.
	\label{continuity}
\end{proposition}
\begin{proof}
		Suppose that $f$ is continuous. Let $\varphi\in\textsc{Nets}(X)$ be a net such that $\varphi\to x$, for every open set $V\subseteq Y$ such that $f(x)\in V$, since $f$ is continuous, it happens that $f^{-1}[V]$ is open and $x\in f^{-1}[V]$. Since $\varphi\to x$, there is $d'\in\hbox{dom}(\varphi)$ such that $\{\varphi_d:d\geq d'\}\subseteq f^{-1}[V]$. Then $$f(\{\varphi_d:d\geq d'\})\subseteq\{(f\circ\varphi_d:d\geq d')\}\subseteq V$$ This proves that $f\circ \varphi\to f(x)$. Conversely, if $V\subseteq X$ is an open set, let us show that $f^{-1}[V]$ is open. Given a net $\varphi\in\textsc{Nets}(X)$ such that $\varphi\to x$ and $x\in f^{-1}[V]$, it follows by hypotheses that  $f\circ \varphi\to f(x)$.  There is $d'\in\hbox{dom}(f\circ\varphi)$ such that $\{(f\circ \varphi)_d:d\geq d'\}\subseteq V$. It follows that $\{\varphi_d:d\geq d'\}\subseteq f^{-1}[V]$ and, therefore, $f^{-1}[V]$ is open by Proposition \ref{open}. This proves that $f$ is continuous.
	
\end{proof}

In a natural way, we induce a filter \( \varphi^{\uparrow} \) from a net \( \varphi \), where \( \varphi^{\uparrow} \) consists of the sets that contain a tail set of the net \( \varphi \), i.e., a subset $\{\varphi_d:d\geq d'\}$ for some $d'\in\hbox{dom}(\varphi)$. This correspondence defines a class function \label{symbol:netFilter}

$$\begin{array}{rcl}
	(\bullet)^{\uparrow}: \textsc{Nets}(X)& \to     & \textsc{Fil}^*(X) \\ 
	\varphi& \mapsto & 	\varphi^{\uparrow}\\
\end{array}$$  

A question arises: Is this class function onto, that is, is every proper filter induced by a net? The reader may reflect for a moment, but the next proposition gives us an answer.

\begin{proposition}
	\label{prop124}
	The class function $(\bullet)^{\uparrow}: \textsc{Nets}(X)\to\textsc{Fil}^*(X) $ is onto.
\end{proposition}

\begin{proof}
	For each proper filter $\mathcal{F}\in\textsc{Fil}^*(X)$, consider the set $$\mathbb{D}_{\mathcal{F}}=\{\langle x,F\rangle \in X\times\mathcal{F}:x\in F\}$$
	Declare that $\langle x,F\rangle \leq \langle y,G\rangle $ if and only if $G\subseteq F$. Note that $\langle \mathbb{D}_{\mathcal{F}},\leq \rangle$ is a directed set. The net	$$\begin{array}{rcl}
		\Gamma: \mathbb{D}_{\mathcal{F}}& \to     & X \\ 
		\langle x,F\rangle & \mapsto & 	x\\
	\end{array}$$ is such that $\Gamma^{\uparrow} = \mathcal{F}$. Indeed, if \( F \in \mathcal{F} \), since \( F \neq \emptyset \), there is \( x \in F \). Then \( \langle x,F\rangle \in \mathbb{D}_{\mathcal{F}} \). Notice that $$ \{\Gamma\langle y,G\rangle  = y : \langle y,G\rangle  \geq \langle x,F\rangle \} \subseteq F$$ It follows that \( F \in \Gamma^{\uparrow} \). In other words, $\mathcal{F}\subseteq\Gamma^{\uparrow}$. On other hand, if \( F \in \Gamma^{\uparrow} \), there is \( \langle y,G\rangle  \in \mathbb{D}_{\mathcal{F}} \) such that \( \{\Gamma\langle x,H\rangle  = x : \langle x,H\rangle  \geq \langle y,G\rangle \} \subseteq F \). Then $$G \subseteq \{\Gamma\langle x,H\rangle  = x : \langle x,H\rangle \geq \langle y,G\rangle \}\subseteq F$$ It follows that \( F \in \mathcal{F} \). In other words $\Gamma^{\uparrow}\subseteq\mathcal{F}$. This proves that \( \Gamma^{\uparrow} = \mathcal{F} \).
\end{proof}

 This means that we have flexibility to work with filter or nets. Consequently, it will be common from now on to utilize either nets or filters in our arguments, with the choice depending on which approach provides greater clarity and intuition. Recall that, basically, a subsequence is a new sequence derived from a sequence by selecting some, but not all, elements from this sequence. In the context of nets, we have the concept of subnet.

\begin{definition}
	Let $\varphi$ and $\psi$  be nets in $X$. We say that $\psi$ is a \textbf{subnet }of $\varphi$ if $\varphi^{\uparrow}\subseteq\psi^{\uparrow}$, that is, for all $a\in\hbox{dom}(\varphi)$ there is $b'\in\hbox{dom}(\psi)$ such that $\{\psi_b:b\geq b'\}\subseteq \{\varphi_d:d\geq a \}$.
\end{definition} There are several non-equivalent definitions of subnets, but the most common one, by Stephen Willard, defines a subnet as a net that can be mapped ``cofinally" into another net through a monotonic function. The definition of subnet used in this text was given by Aarnes and Andenaes. For more details see~\cite{schechter}.

\begin{proposition}

	Let $X$ be a topological space and $\varphi\in\textsc{Nets}(X)$ be a net such that $\varphi\to x$. If $\psi\in\textsc{Nets}(X)$ is a subnet of $\varphi$, then  $\psi\to x$.
	
	\label{subnet}
\end{proposition}

\begin{proof}
	If $\varphi\to x$ , it happens that $\mathcal{N}_x\subseteq \varphi^{\uparrow}$. Since $\varphi^{\uparrow}\subseteq \psi^{\uparrow}$, we conclude that $\mathcal{N}_x\subseteq \psi^{\uparrow}$. This means that $\psi\to x$.
\end{proof}

\begin{figure}[H]
	\centering
	\includegraphics[width=6cm]{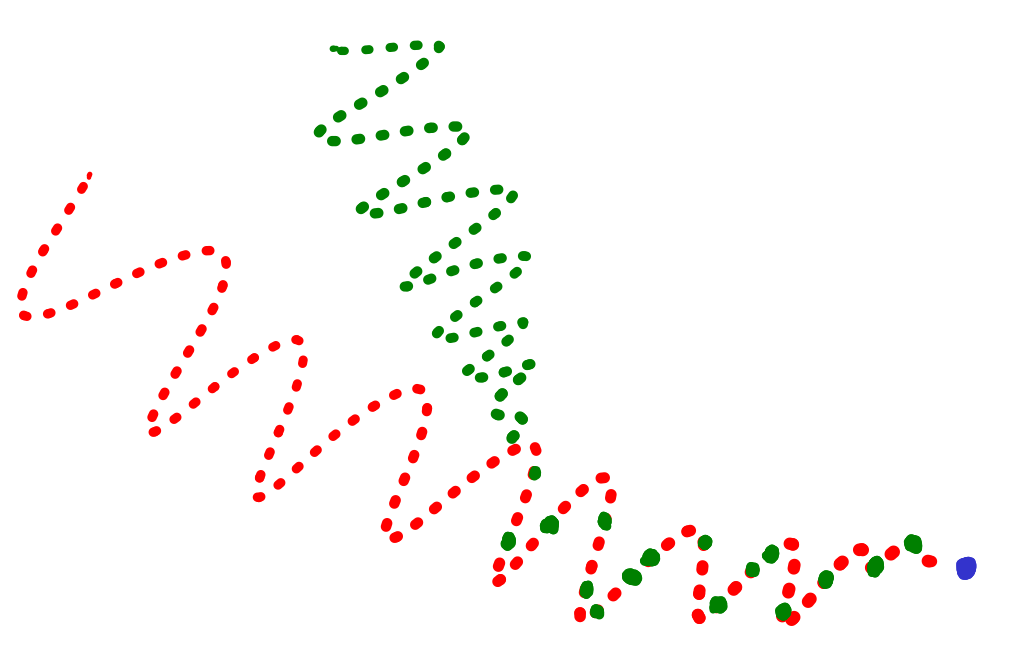}
	
	\caption{A subnet(green) of a net(red) converging to the same point}
\end{figure}

Given two sequences $\langle x_n\rangle_n$ and $\langle y_n\rangle_n$, we can create a sequence $\langle z_n\rangle_n$ such that $z_{2n}=x_n$ and $z_{2n+1}=y_n$ for all natural number $n\in\mathbb{N}$. Notice that this was created just by choosing terms from the two sequences we initially considered. This motivates us to introduce the concept of mixing nets.

\begin{definition}
Let $\varphi,\psi \in\textsc{Nets}(X)$ be nets in $X$ with the same domain. A  net $\rho$ in $X$ is a \textbf{mixing} of $\varphi$ and $\psi$ if it has the same domain and $\rho_d\in\{\varphi_d,\psi_d\}$ for every $d\in\hbox{dom}(\varphi)$.
	\label{mixing}\label{symbol:domainNet}
\end{definition}
\begin{remark}
	We could define mixing of nets with different domains, but if we consider the nets 
	
	$$	\begin{array}{rcl}
		\tilde{\varphi}: \hbox{dom}(\varphi)\times \hbox{dom}(\psi) & \to     & X\\ 
		\langle x,y\rangle  & \mapsto & \varphi_x \\
	\end{array}$$
		and 
		$$	\begin{array}{rcl}
		\tilde{\psi}: \hbox{dom}(\varphi)\times \hbox{dom}(\psi) & \to     & X\\ 
		\langle x,y\rangle  & \mapsto & \psi_y \\
	\end{array}$$
	we have $\tilde{\varphi}^{\uparrow}=\varphi^{\uparrow}$ and $\tilde{\psi}^{\uparrow}=\psi^{\uparrow}$, where the cartesian product of directed sets $\mathbb{A}$ and $\mathbb{B}$ is such that $\langle a,b\rangle\leq\langle a',b'\rangle$ if and only if $a\leq a'$ and $b\leq b'$. In other words, they are ``equivalent" nets, in the sense of inducing the same filter, and without loss of generality, we can consider that $\varphi$ and $\psi$ have the same domain.
	\label{symbol:pair}
\end{remark}

\begin{remark}

At its core, the mixing of nets is related to the intersection of filters. Let $\varphi,\psi$ and $\rho$ be nets, where $\rho$ is a mixing of $\varphi$ and $\psi$, if $A\in\varphi^{\uparrow}\cap\psi^{\uparrow}$ there are $d_0,d_1\in\hbox{dom}(\varphi)$ such that $$\{\varphi_d:d\geq d_0\},\{\psi_d:d\geq d_1\}\subseteq A$$For  $d_2\in\hbox{dom}(\varphi)$ such that $d_2\geq d_0,d_1$, it happens that $$\{\rho_d:d\geq d_2\}\subseteq \{\varphi_d:d\geq d_2\}\cup\{\psi_d:d\geq d_2\}\subseteq A$$
It follows that $A\in\rho^{\uparrow}$. This proves that $\varphi^{\uparrow}\cap\psi^{\uparrow}\subseteq\rho^{\uparrow}$. Then a mixing is a subnet of every net inducing the intersection of filters.
\end{remark}

\begin{proposition}
	Let $\psi$ and $\varphi$ are nets in $X$ such that $\varphi,\psi\to x$. A mixing of $\psi$ and $\varphi$ also converges to $x$.
\end{proposition}

\begin{proof}
If $\psi,\varphi\to x$ and $\rho$ is a mixing of $\varphi$ and $\psi$, it happens that $\mathcal{N}_x\subseteq \varphi^{\uparrow}$ and $\mathcal{N}_x\subseteq\psi^{\uparrow}$. Since $\varphi^{\uparrow}\cap\psi^{\uparrow}\subseteq\rho^{\uparrow}$, it follows that $\mathcal{N}_x\subseteq\rho^{\uparrow}$. This proves that $\rho\to x$.
\end{proof}

	\begin{figure}[H]
		\centering
	\includegraphics[width=13cm]{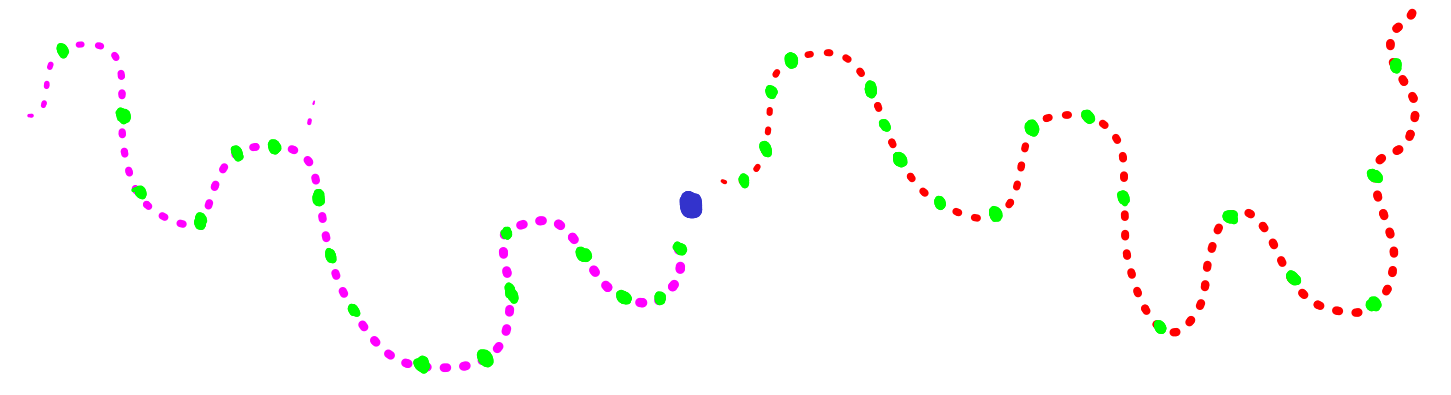}
	\label{key2}
	\caption{Two nets (pink and red) converging to the same point (blue) and their mixing (green) also converging to this point. There is a stability in some sense, the choice did not mess up the convergence.}
\end{figure}

\newpage

\

\newpage

\chapter{A crash course in convergence spaces}
\label{chap2}

In this chapter, we introduce the convergence spaces, which will be our main subject of study. The motivation behind this comes from the discussions made in the previous chapter on how convergence of nets and filters fully characterized the topology of a space. Our principal references are~\cite{ref4},~\cite{schechter1996handbook},~\cite{ref3} and~\cite{schechter}, that use filters to describe convergence spaces. Although filters are a good tool for convergence theory, for those who are just starting to study this area, the notation is complicated. It is necessary to perform many operations with families of sets, and in this process, much geometric intuition is lost. The goal is to provide an alternative approach to convergence theory using nets.

\section{Preconvergences, convergences and limit spaces}

\begin{definition}
	A function $$L:\textsc{Fil}^*(X)\to \mathcal{P}(X)$$ is a \textbf{preconvergence} on a set $X$ and the pair $\langle X, L\rangle $ is said to be a \textbf{preconvergence space}. We write $\mathcal{F}\to_{L}x$ whenever $\mathcal{F}\in\textsc{Fil}^*(X)$ and $x\in L(\mathcal{F})$ and say that the filter $\mathcal{F}$ $L$-converges to $x$.
\end{definition}

\begin{remark}
	This is the same definition presented by Dolecki and Mynard in~\cite{schechter1996handbook}. Schecter~\cite{schechter} adapted this definition for nets, but we will delve deeper into the discussion by addressing adherence, inherence, etc., with nets.
\end{remark}

Intuitively, a preconvergence associates to each proper filter its set of limit points. Many authors define a preconvergence as a relation $L\subseteq\textsc{Fil}^*(X)\times X$, as in~\cite{ref4}. This is possible because there is a bijection between the set $\hbox{Rel}(\textsc{Fil}^*(X)\times X)$ of relations on $\textsc{Fil}^*(X)\times X$ and the set $\hbox{Fun}(\textsc{Fil}^*(X),\mathcal{P}(X))$ of functions with domain $\textsc{Fil}^*(X)$ and codomain $\mathcal{P}(X)$. But in this work, a preconvergence $L$ on a set $X$ is a class function $$L:\textsc{Nets}(X)\to\mathcal{P}(X)$$ such that $L(\varphi)=L(\psi)$ whenever $\varphi^{\uparrow}=\psi^{\uparrow}$.  We write $\varphi\to_L x$ instead $x\in L(\varphi)$. How can we do this? In Section \ref{sec1.2} we define a class function that associates a net to its induced filter

$$\begin{array}{rcl}
(\bullet)^{\uparrow}: \textsc{Nets}(X)& \to     & \textsc{Fil}^*(X) \\ 
\varphi & \mapsto & 	\varphi^{\uparrow}\\
\end{array} $$ If a preconvergence is a function $L:\textsc{Fil}^*(X)\to\mathcal{P}(X)$, we have the commutative diagram
\[\begin{tikzcd}[ampersand replacement=\&,cramped]
	{\textsc{Nets}(X)} \& {\textsc{Fil}^*(X)} \\
	\& {\mathcal{P}(X)}
	\arrow["{(\bullet)^{\uparrow}}", from=1-1, to=1-2]
	\arrow["{L\circ(\bullet)^{\uparrow}}"', from=1-1, to=2-2]
	\arrow["L", from=1-2, to=2-2]
\end{tikzcd}\]
This means that there is class function $L\circ (\bullet)^{\uparrow}:\textsc{Nets}(X)\to\mathcal{P}(X)$ such that $L(\varphi)=L(\psi)$ whenever $\varphi^{\uparrow}=\psi^{\uparrow}$. On the other hand, if a preconvergence is a class function $L:\textsc{Nets}(X)\to\mathcal{P}(X)$, recall that Proposition \ref{prop124} tells us that there is a right inverse $\Gamma(\bullet):\textsc{Fil}^*(X)\to\textsc{Nets}(X)$ of $(\bullet)^{\uparrow}$. Then there is a function $\textsc{Fil}^*(X)\to\mathcal{P}(X)$ by the commutativity of the diagram
\[\begin{tikzcd}[ampersand replacement=\&,cramped]
	{\textsc{Fil}^*(X)} \& {\textsc{Nets}(X)} \\
	\& {\mathcal{P}(X)}
	\arrow["{\Gamma(\bullet)}", from=1-1, to=1-2]
	\arrow["{L\circ \Gamma(\bullet)}"', from=1-1, to=2-2]
	\arrow["L", from=1-2, to=2-2]
\end{tikzcd}\]

\begin{example}
	The class function $\hat{\emptyset}:\textsc{Nets}(X)\to\mathcal{P}(X)$ such that  $\hat{\emptyset}(\varphi)=\emptyset$  for every net $\varphi\in\textsc{Nets}(X)$ is called the empty preconvergence on $X$. On the other hand, the class function $c_X:\textsc{Nets}(X)\to \mathcal{P}(X)$ such that $c_X(\varphi)=X$ for every $\varphi\in\textsc{Nets}(X)$ is called the chaotic preconvergence. Clearly, there is no theoretical interest behind these convergences; after all, what's the point with a net that does not converge to any point or that converges to every point? The important thing here is to see how the concept of convergence is becoming very general, to the point where these pathological convergences are allowed. For the real line and plane, we denote their corresponding preconvergences by $\to_{\mathbb{R}}$ and $\to_{\mathbb{R}^2}$. \label{symbol:limTau}\label{symbol:convR}\label{symbol:convR2}
\end{example}

\begin{example}
	A topological space $\langle X,\tau\rangle$ has a natural preconvergence induced by definition of convergence of nets. We denote this preconvergence by $\lim_\tau$. Unless otherwise mentioned, whenever we refer to topological spaces as preconvergence spaces, it is this preconvergence that we are considering.
\end{example}

\begin{definition}
A preconvergence $L$ on $X$ is
\begin{enumerate}
		\item \textbf{centered}, if every constant net converges to its image,
	\item\textbf{isotone}, if $\mathcal{\psi}\to_L x$ whenever $\psi$ is a subnet of $\varphi$ and $\varphi\to_L x $,
	\item \textbf{stable}, if $\rho\to_L x$ whenever $\varphi,\psi\to_L x$ and $\rho$ is a mixing of $\varphi$ and $\psi$.
\end{enumerate}
If $L$ is centered and isotone, we say that $L$ is a convergence and the pair $\langle X, L\rangle$ is a \textbf{convergence space}. A convergence  $L$ that is stable is said to be a limit convergence and $\langle X,L\rangle$ is a \textbf{limit space}.
\end{definition} In the context of filters, centerness refers to the convergence of principal ultrafilters, while isotonicity ensures that $L$ is an increasing function with respect to the inclusion relation and keeps Proposition \ref{subnet} true. Stability relates to the convergence of finite intersections of filters. These are desirable properties that we expect a preconvergence to possess. Even if a preconvergence does not possess some of these properties, it is possible to modify it so that they are recovered, as we will see in Section \ref{modifiers}.
\begin{example}
	Follows from Propositions \ref{mixing} and \ref{subnet} that topological spaces are limit spaces. 
\end{example}

\begin{example}[Adapted from~\cite{marroquin}]
	\label{circle}
 Fix a real number $r>0$. In the circle $\mathbb{S}^1$ consider a preconvergence $L_r$ such that a net $\varphi\in\textsc{Nets}(\mathbb{S}^1) $ $L_r$-converges to $x\in \mathbb{S}^1$ if and only if there are $x_0,\cdots, x_n\in\mathbb{S}^1$ such that $ \{x_0,\cdots,x_n\}^{\uparrow}\subseteq \varphi^{\uparrow}$ and $||x_i-x||\leq r$ for every $0\leq i\leq n$, where $||\cdot||$ denotes the usual euclidean distance. Let us show that $\langle \mathbb{S}^1,L_r\rangle$ is a limit space. Clearly this preconvergence is centered and isotone. For stability, if $\varphi,\psi\to_{L_r} x$ there are $x_0,\cdots,x_n,y_0,\cdots, y_m\in\mathbb{S}^1$ such that $||x_i-x||,||y_j,x||\leq r$ for $0\leq i\leq n$ and $0\leq j \leq m$, $ \{x_0,\cdots,x_n\}^{\uparrow}\subseteq \varphi^{\uparrow}$ and $ \{y_0,\cdots,y_m\}^{\uparrow}\subseteq \psi^{\uparrow}$. Notice that $$\{x_0,x_1,\cdots,x_n,y_0,y_1,\cdots,y_m\}^{\uparrow}\subseteq \psi^{\uparrow}\cap \varphi^{\uparrow}$$ Since $\varphi^{\uparrow}\cap\psi^{\uparrow}\subseteq\rho^{\uparrow}$ for a mixing $\rho$ of $\varphi$ and $\psi$, it follows that $\rho\to_{L_r} x$. But what is the intuition about this limit convergence? The idea is that, from a certain point onward, the net has only a finite number of points as its image that are close to limit point.
	\begin{figure}[H]
		\centering
		\includegraphics[width=6cm]{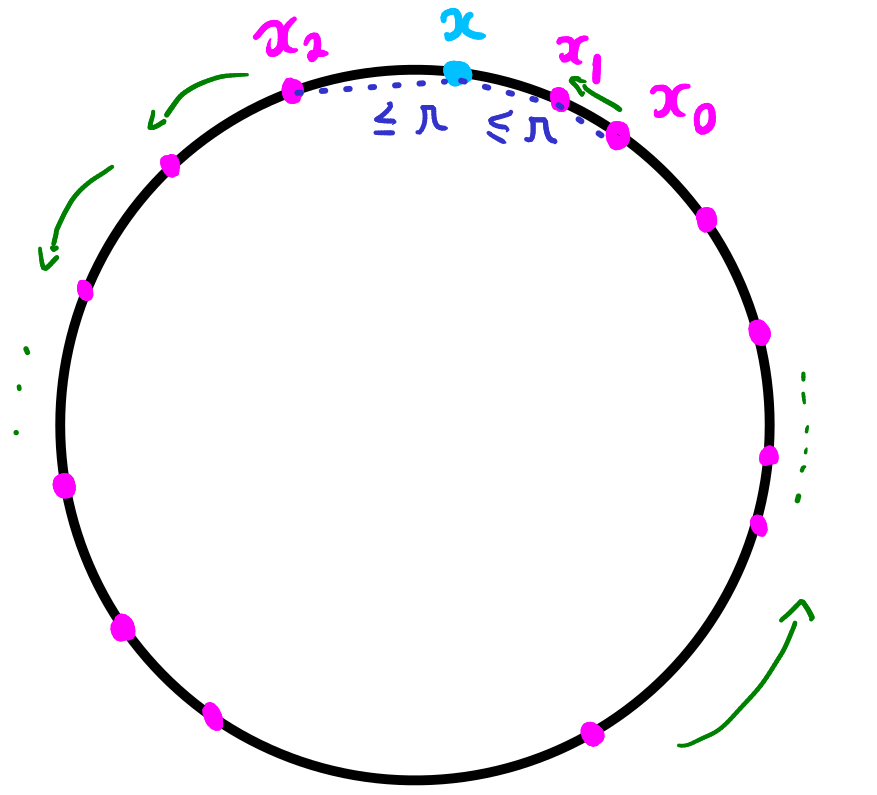}
		\label{key3}
		\caption{Consider a net (in pink) that starts at $x_0$
			and moves in a counterclockwise direction, and after the first cycle, it only takes the values 
		$x_0,x_1$ or $x_2$. In this case, the convergence is witnessed by the points $x_0,x_1$ and $x_2$.}
	\end{figure}
\end{example}

\begin{example}[Adapted from~\cite{ref5}]
	\label{215}
	In the unit interval $[0,1]$ we define a preconvergence $L$ such that $\varphi\to_L x$ if and only if 	$x\neq\frac{1}{2}$ and  $\varphi\to_{\mathbb{R}} x $ or $x=\frac{1}{2}$ and there is a net  $\psi\in\textsc{Nets}([0,\frac{1}{2}])\cup\textsc{Nets}([\frac{1}{2},1])$ such that $\psi \to_{\mathbb{R}}\frac{1}{2}$ and $\psi ^{\uparrow}\subseteq \varphi^{\uparrow} $ or $\varphi^{\uparrow}=\mathfrak{u}_{\frac{1}{2}} $. We show that this preconvergence is centered and isotone, but it is not stable. 
	
	\begin{enumerate}
		\item (Centerness) Given a constant net $\langle x\rangle_d\in\textsc{Nets}([0,1])$. If $x\neq \frac{1}{2}$ clearly converges usually on the real line. Now, if $x=\frac{1}{2}$ we have $\langle x\rangle_d^{\uparrow}=\mathfrak{u}_{\frac{1}{2}}$. This means that $\langle x\rangle_d\to_L x$.
		
		\item (Isotonicity) If $\varphi\to_{L} x$ and $\psi$ is subnet of $\varphi$, in case that $x\neq \frac{1}{2}$ and $\varphi\to_{\mathbb{R}} x$, it happens that $\psi\to_{\mathbb{R}}x$. On the other hand, if $x=\frac{1}{2}$ and there is a net $\sigma\in\textsc{Nets}([0,\frac{1}{2}])\cup\textsc{Nets}([\frac{1}{2}])$ such that $\sigma\to_{\mathbb{R}}\frac{1}{2}$ and $\sigma^{\uparrow}\subseteq \varphi^{\uparrow}$ or $\varphi^{\uparrow}=\mathfrak{u}_{\frac{1}{2}}$, then $\sigma^{\uparrow}\subseteq \psi^{\uparrow}$ or $\psi^{\uparrow}=\mathfrak{u}_{\frac{1}{2}}$ and $\psi\to_{\mathbb{R}}x$. It follows that $\psi\to_{L} x$.
		
		\item Let $\langle s_n\rangle$ be a sequence in $[0,\frac{1}{2})$ and $\langle t_n\rangle$ be a sequence in $(\frac{1}{2},1]$ both converging to $\frac{1}{2}$. The sequence $\langle z_n\rangle$ such that $z_{2n}=s_n$ and $z_{2n+1}=t_n$ is a mixing of $s$ and $t$ and there is no net in $\textsc{Nets}([0,\frac{1}{2}])$ or $\textsc{Nets}([\frac{1}{2},1])$ having $\langle z_n\rangle_n$ as subnet because his tails have points in $[0,\frac{1}{2})$ and $(\frac{1}{2},1]$. This show that this preconvergence is not stable.
	\end{enumerate}
	 In particular, this preconvergence is not topological, in the sense of being induced by a topology. This is not just an example of a convergence lacking stability; in Section \ref{sec3.1}, we will use this example to examine  the hypotheses necessary for the existence of a Pasting Lemma in preconvergence spaces.
\end{example}
\begin{example}
	\label{lollipop}
		Let $X\subseteq\mathbb{R}^2$ be a lollipop, that is, a union of a line $L$ and a circle $C$ in the plane such that $L\cap C=\{p\}$.
	
	\begin{figure}[H]
		\centering
		\includegraphics[width=6cm]{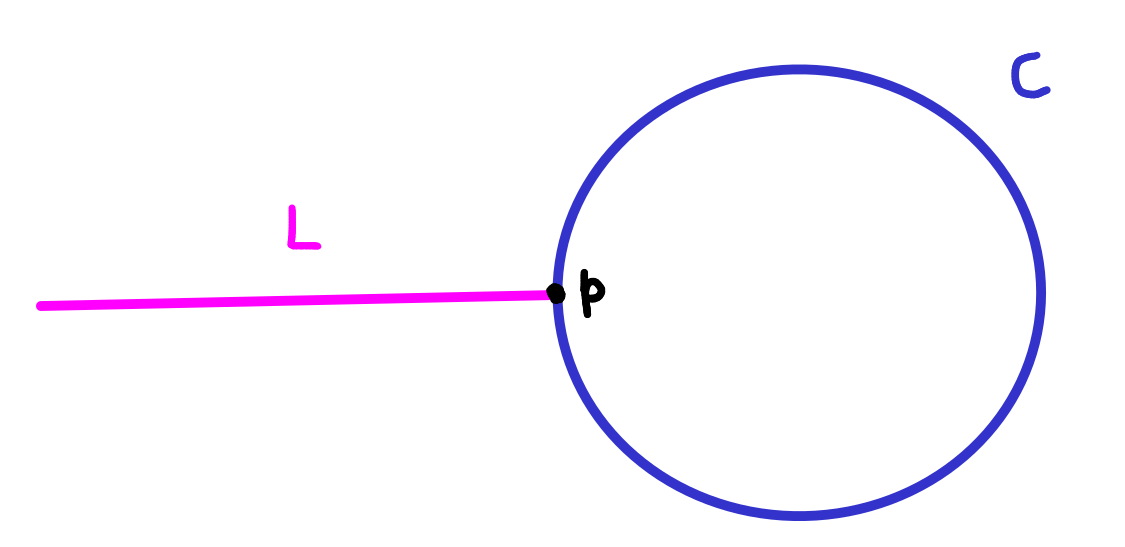}
		\caption{The lollipop $X$}
	\end{figure}

	Let $S=L\setminus\{p\}$ and $D\subseteq S$ be a dense countable set with respect to the usual topology of $\mathbb{R}^2$. For a net $\varphi\in\textsc{Nets}(X)$ and a point $x\in X$, we define a preconvergence $\lambda$ such as:
	\begin{enumerate}
		\item If $x\neq p$, then $\varphi\to_\lambda x$ if and only if $\varphi\to_{\mathbb{R}^2} x$,
		\item If $x=p$, then $\varphi\to_\lambda x$ if and only if
		\begin{enumerate}[i)]
			\item $C\in\varphi^\uparrow$ and $\varphi\to_{\mathbb{R}^2} p$ or,
			\item $\varphi^\uparrow\# D$, in the sense of the definition \ref{defmesh}, and $\varphi\to_{\mathbb{R}^2} p$ or,
			\item $\varphi$ is a subnet of a mixing of finitely many nets satisfying conditions (i) or (ii).
			
		\end{enumerate}
		
	\end{enumerate}
	Notice that $\langle X,\lambda \rangle $ is a limit space. Indeed,

	\begin{enumerate}
		\item (Centerness) 
		Let $\langle x\rangle_d$ be a constant net in $X$. If $x\neq p$, clearly $\langle x\rangle_d\to_{\mathbb{R}^2} x$. That is, $\langle x\rangle_d\to_\lambda x$. On the other hand, if $x=p$, notice that $C\in\langle x\rangle_d^{\uparrow}=\mathfrak{u}_x$ and $\langle x\rangle_d\to_{\mathbb{R}^2}$. This means that $\langle x\rangle_d\to_\lambda p$.
		\item (Isotonicity) If $\varphi\to_\lambda x$ and $\psi$ is subnet of $\varphi$. In case that $x\neq p$ it is clear that $\psi\to_\lambda x$. On the other hand, if $x=p$ we would have $C\in\varphi^{\uparrow}\subseteq\psi^{\uparrow}$ and $\psi\to_{\mathbb{R}^2} p$, or $D\#\varphi^{\uparrow}$ and $\varphi\to_{\mathbb{R}^2} x$, then $\psi$ is subnet of a mixing of finitely many nets satisfying condition (ii), or $\psi$ is subnet of a mixing of finitely many nets satifsfayng the  conditions (i) and (ii). It follows that $\psi\to_{\lambda} x$.
		\item (Stability) Let $\varphi,\psi\in\textsc{Nets}(X)$ be nets such that $\varphi,\psi\to_{\lambda} x$ and $\rho$ be  a mixing of $\varphi$ and $\psi$. If $x\neq p$, clearly $\rho\to_{\lambda} x$. Otherwise, if $x=p$, in any case $\rho$ is subnet of a mixing of finitely many nets satisfying conditions (i) or (ii).
	\end{enumerate}
\end{example}

 Recall that Proposition \ref{continuity} states that continuity of a function preserves the convergence of nets. Therefore, the definition of continuity in preconvergence spaces is as follows.

\begin{definition}
	Let $\langle X,L\rangle $ and $\langle Y,L'\rangle$ be preconvergence spaces. A function $f:X\to Y$ is \textbf{continuous} if $f\circ \varphi \to_{L'} f(x)$ whenever $\varphi\in\textsc{Nets}(X)$ is a net such that $\varphi\to_L x$. We denote by $\mathcal{C}(X,Y)$ the set of all continuous from $X$ to $Y$. \label{symbol:continuousFunctions}
\end{definition}

\begin{remark}
	In filter language, like in Proposition \ref{filter}, a function $f:\langle X,L\rangle\to\langle Y,L'\rangle$ is continuous if $f(\mathcal{F})\to_{L'} f(x)$ whenever $\mathcal{F}\in\hbox{Fil}^*(X)$ is a filter such that $\mathcal{F}\to_L x$.
\end{remark}

\begin{proposition}
	Let $\langle X,\tau\rangle$ and $\langle Y,\tau'\rangle$ be topological spaces. If $f:\langle X,\tau\rangle \to\langle Y,\tau'\rangle$ is continuous as a function between topological spaces, then $f:\langle X,\lim_\tau\rangle \to\langle Y,\lim_{\tau'}\rangle $ is continuous as function between preconvergences spaces.
\end{proposition}

\begin{proof}
	Follows from Proposition \ref{continuity}.
\end{proof}

\begin{proposition}
	If $f: X\to Y$ and $g:Y\to Z$ are continuous, then $g\circ f:X\to Z$ is continuous. Furthermore, the identity function $1_X:X\to X$ is continuous. \label{symbol:identity}
\end{proposition}

\begin{proof}
If $\varphi\in\textsc{Nets}(X)$ is a net such that $\varphi\to_L x$, since $f$ is continuous, it happens that $f\circ \varphi\to f(x)$. But $g$ is also continuous, so $g\circ(f\circ \varphi)=(g\circ f)\circ\varphi\to g(f(x))=g\circ f(x)$. This proves that $g\circ f$ is continuous. It is easy to check that the identity function is continuous.
\end{proof}

\label{symbol:PrConv}\label{symbol:Conv}\label{symbol:Lim}
\label{symbol:Top}
We denote the categories of preconvergence spaces, convergence spaces, and limit spaces by $\textsc{PrConv}$, $\textsc{Conv}$, and $\textsc{Lim}$, respectively, where the arrows are the continuous functions. The previous examples show that the following strict categorical inclusions holds: $\textsc{Lim}\subsetneq \textsc{Conv}\subsetneq\textsc{PrConv}$. In Example \ref{sphere}, we will also see that $\textsc{Top}\subsetneq \textsc{Lim}$. Some of these categories are preferable to others precisely due to the properties that preconvergence possesses. In Section \ref{modifiers}, we show how to make a topological or limit modification of a preconvergence space in a functorial way. But what do we gain by extending the category of topological spaces? The most significant advantage is the existence of exponential objects, as we will see in Section \ref{sec2.4}. There are additional categories of preconvergences spaces, such as pretopological spaces, pseudotopological spaces and Kent spaces. A convergence space $\langle X,L\rangle$ is:

\begin{itemize} 
	
	\item pseudotopological, if $L(\mathcal{F})=\bigcap_{\mathfrak{u}\in\beta(\mathcal{F})}L(\mathfrak{u}) $, where $\beta(\mathcal{F})$ is the set of ultrafilters on $X$ containing $\mathcal{F}$. Ultrafilters can be easily replaced by ultranets, which are precisely those nets
	whose induced filters are ultrafilters. 
	
	\item a Kent space, if for every $x\in X$
	and every net $\varphi\in\textsc{Nets}(X)$, we have $\rho\to_L x$ whenever $\rho$ is a
	mixing of $\varphi$ and the constant net $\langle x\rangle_d $ and $\varphi\to_L x$.
	
	\item pretopological, if for each $x\in X$ there is a filter $\mathcal{V}_x$ converging to $x$ such that $\mathcal{V}_x\subseteq \mathcal{F}$ whenever $\mathcal{F}\in\textsc{Fil}^*(X)$ is a filter such that $\mathcal{F}\to_ Lx$. For more details see~\cite{schechter1996handbook} and~\cite{dossena}.
	
\end{itemize}

Recall that we can compare two topologies on a set by inclusion. We say that a topology $\tau'$ on a set $X$ is finer than another topology $\tau$ on the same set if $\tau \subseteq \tau'$, that is, every open set in $\tau$ is also open in $\tau'$. Similarly, we can also compare preconvergences on a set. The idea now is that the convergence of a net in one preconvergence implies the convergence of the same net in another preconvergence. This is interesting because it will allow us to discuss final and initial structures, where we see that these categories have good properties, such as the existence of products, coproducts, quotients, and subspaces.

\begin{definition}
Let $L$ and $L'$ be preconvergences on $X$. We say that $L'$ is \textbf{finer} or \textbf{stronger} than $L$, or that $L$ is \textbf{coarser} or \textbf{weaker} than $L'$, if $\varphi\to_L x$ whenever $\varphi\to_{L'} x$, that is, $L'(\varphi)\subseteq L(\varphi)$ for every net $\varphi\in\textsc{Nets}(X)$. In this case we write $L\leq L'$. \label{symbol:preconvergenceOrder}
\end{definition}

\begin{remark}
This order on preconvergences reflects the order of topologies. If $\tau$ and $\tau'$ are topologies on $X$ such that $\tau \subseteq \tau'$ and $\varphi \in \textsc{Nets}(X)$ is a net such that $\varphi \to_{\tau'} x$, then $\mathcal{N}_{x,\tau'} \subseteq \varphi^{\uparrow}$. Since $\mathcal{N}_{x,\tau} \subseteq \mathcal{N}_{x,\tau'}$, it follows that $\varphi \to_{\tau} x$.  This means that $\lim_{\tau'}$ is finer than $\lim_{\tau}$. Conversely, if $\lim_{\tau'}\geq \lim_{\tau}$ and $U\subseteq X$ is an $\tau$-open set, for a net $\varphi\in\textsc{Nets}(X)$ such that $\varphi\to_{\tau'}x\in U$ we have $\varphi\to_{\tau} x$. Since $U$ is $\tau$-open, its follows that $U\in\varphi^{\uparrow}$. By Proposition \ref{open}, $U$ is an $\tau'$-open set. This proves that $\tau\subseteq \tau'$. 
\end{remark}

\begin{example}[Adapted from~\cite{ref3}]
	\label{seq}
	For a net $\varphi\in\textsc{Nets}(\mathbb{R})$ we declare $\varphi\to_{\hbox{Seq}} x$ if there is sequence $s:\mathbb{N}\to\mathbb{R}$ such that $\varphi$ is subnet of $s$ and $s\to_{\mathbb{R}}x$. In this case, we say that $\varphi$ converges sequentially to $x$. Notice that this preconvergence is stronger than the usual convergence on the real line
	 We show that $\langle \mathbb{R},\hbox{Seq} \rangle$ is a limit space.
	\begin{enumerate}

	\item (Centerness) A constant net and a constant sequence induces the same principal ultrafilter. Then constant nets converge sequentially.
	
	\item (Isotonicity) If $\varphi\in\textsc{Nets}(\mathbb{R})$ is a net such that $\varphi\to_{\hbox{Seq}} x$ and $\psi$ is subnet of $\varphi$, there is a sequence $s:\mathbb{N}\to\mathbb{R}$ such that $s\to_{\mathbb{R}} x$ and $s^{\uparrow}\subseteq \varphi^{\uparrow}$. It follows that $s^{\uparrow}\subseteq \psi^{\uparrow}$, then $\psi\to_{\hbox{Seq}} x$.
	
	\item (Stability)  Let $\rho\in\textsc{Nets}(\mathbb{R})$ be a mixing of two nets $\varphi$ and $\psi$ such that $\varphi,\psi\to_{\hbox{Seq}}x$. There are sequences $p,q:\mathbb{N}\to\mathbb{R}$ such that $s_1,s_2\to x$ and $p^{\uparrow}\subseteq \varphi^{\uparrow}$ and $q^{\uparrow}\subseteq \psi^{\uparrow}$. So, $p^{\uparrow}\cap q^{\uparrow}\subseteq \varphi^{\uparrow}\cap\psi^{\uparrow}\subseteq \rho^{\uparrow}$. The sequence $s:\mathbb{N}\to \mathbb{R}$ such that $s_{2n}=p_n$ and $s_{2n+1}=q_n$ converges to $x$ and $s^{\uparrow}=p^{\uparrow}\cap q^{\uparrow}$. It follows that $\rho\to_{\hbox{Seq}} x$.
\end{enumerate}
	 The difference between this convergence and the usual one is that we now have many more filters to witness convergence. In the usual case, we only have the neighborhood filter, whereas in this context, the filters of convergent sequences can be used. We will back to this example in Section \ref{sec3.3}.
\end{example}

The goal now is to study the structure of preconvergences with the order defined earlier. Recall that a complete lattice is a partially ordered set in which every subset has both a supremum  and an infimum. For a family $\mathcal{S}$ of preconvergences on a set $X$, we define the preconvergences $\bigvee \mathcal{S}$ and $\bigwedge \mathcal{S}$ in the following way

$$\begin{array}{rcl}
	\bigvee\mathcal{S} : \textsc{NETS}(X)& \to     & \mathcal{P}(X) \\ 
	\varphi & \mapsto & \begin{cases}
		\bigcap _{L\in\mathcal{S}}L(\varphi) \ \hbox{  if } \mathcal{S}\neq \emptyset \\
		X \ \hbox{  otherwise }
	\end{cases}\\
\end{array} \ \ \hbox{ and } \ \ \begin{array}{rcl}
	\bigwedge\mathcal{S} : \textsc{Nets}(X)& \to     & \mathcal{P}(X) \\ 
	\varphi & \mapsto & 	\bigcup_{L\in\mathcal{S}} L(\varphi)\\
\end{array}$$

 Notice that $\bigvee\mathcal{S}=\hbox{sup}(\mathcal{S})$. Indeed, clearly $\bigvee\mathcal{S}\geq L$ for all $L\in\mathcal{S}$ and if $\lambda$ is a preconvergence on $X$ such that $\lambda\geq L$ for all $L\in\mathcal{S}$, it is imediate that $\varphi\to_{\bigvee\mathcal{S}} x$ whenever $\varphi\to_{\lambda} x$. This establishes that $\lambda\geq \bigvee\mathcal{S}$. Then we conclude that $ \bigvee\mathcal{S}=\hbox{sup}(\mathcal{S})$. Similarly, we can prove that $\bigwedge\mathcal{S}=\hbox{inf}(\mathcal{S})$.  In  other words, the set of all preconvergences on a set $X$ is a complete lattice.

\label{sec211}

\begin{definition}
	Let $X$  be a set, $\langle Y_i,L_i\rangle$ be a preconvergence space and $f_i:X\to Y_i$ be a function for each $i\in\mathcal{I}$. The \textbf{initial preconvergence }on $X$ is the coarser preconvergence such that each  $f_i$ is continuous. We denote this preconvergence by $\bigvee_{i\in\mathcal{I}}f_i^{-1}L$. If $|\mathcal{I}|=1$ we denote this preconvergence by $f^{-1}L$. \label{symbol:initialPreconv}
\end{definition}

\begin{proposition}
Let $\langle Y,L\rangle$ be a preconvergence space and $f:X\to Y$ be a function, Then $\varphi \in\textsc{Nets}(X)$ is a net such that $\varphi\to_{f^{-1}L} x$ if and only $f\circ\varphi\to f(x)$.
\end{proposition}

\begin{proof}
	
	Consider the preconvergence $\lambda$ in $X$ such that a net $\varphi\in\textsc{Nets}(X)$ $\lambda$-converges to $x$ if and only if $f\circ \varphi\to_{L} f(x)$. Clearly $\lambda$ makes $f$ continuous. Let $\lambda'$  be a preconvergence in $X$ such that $f:\langle X,\lambda\rangle\to\langle Y,L\rangle $ is continuous, given a net $\varphi\in\textsc{Nets}(X)$ such that $\varphi\to_{\lambda'} x$, since $f$ is continuous, we have $f\circ\varphi \to_L f(x)$. It follows that $\varphi\to_{\lambda} x$. This proves that $f^{-1}L=\lambda$.
\end{proof}

\begin{proposition}
	Let $\langle Y_i, L_i\rangle$ be a preconvergence space and $f_i:X\to Y_i$ be a function for each $i\in\mathcal{I}$. A net $\varphi\in\textsc{Nets}(X)$ is such that $\varphi\to _{\bigvee_{i\in\mathcal{I}}f^{-1}L} x$ if and only if $f_i\circ \varphi\to_{L_i} f(x)$ for all $i\in\mathcal{I}$.
\end{proposition}

\begin{proof}
	
	Notice that $$\varphi\to _{\bigvee_{i\in\mathcal{I}}f^{-1}L} x\Longleftrightarrow \varphi\to_{f_{i}^{-1}L} x \hbox{ for all }i\in\mathcal{I}\Longleftrightarrow f_i\circ\varphi \to_{L_i} f_i(x) \hbox{ for all } i\in\mathcal{I}$$

\end{proof}

\begin{example}
	Let $S\subseteq X$ be a subset of a preconvergence space $\langle X,L\rangle $. The subpreconvergence is the initial preconvergence such that the inclusion $i:S\to X$ is continuous, we denote this preconvergence by $L|_S$. Note that a net $\varphi\in\textsc{Nets}(S)$ $L|_S$-converges to $ x\in S$ if and only if $i\circ\varphi\to_L x$, i.e, $\varphi\to_L x$. This tells us that in subspace preconvergence we are not changing the net, whereas in the context of filters it would be necessary to use a different filter.
\end{example}

\begin{example}
	For each $i\in\mathcal{I}$ consider a preconvergence space $\langle X_i,L_i\rangle$. The product preconvergence on $\prod_{i\in\mathcal{I}} X_i$ is the initial preconvergence such that each projection $$\pi_j:\prod_{i\in\mathcal{I}} X_i\to X_j$$ is continuous. Explicitly, a net $\varphi\in\textsc{Nets}\big(\prod_{i\in\mathcal{I}}X_i\big)$ converges to a $\mathcal{I}$-tuple  $x=\langle x_i \rangle_{i\in\mathcal{I}}$ if and only if $\pi_j\circ\varphi\to_{L_j}x_j$ for all $j\in \mathcal{I}$. In summary, a net converges in the cartesian product if it converges at each coordinate. In the case of two preconvergence spaces $X$ and $Y$, a net $\langle x_a,y_a\rangle_{a\in\mathbb{A}} $ in $X\times Y$ converges to $\langle x,y\rangle $ if and only if $x_a\to x$ in $X$ and $y_a\to y$ in $Y$. \label{symbol:product}
\end{example}
\begin{proposition}
	Let $\langle Y_i,L_i\rangle $ be a preconvergence space and $f_i:X\to Y_i$ be a function for each $i\in\mathcal{I}$. A function $g:\langle Z,L'\rangle \to \langle X,\bigvee f_i^{-1}L_i\rangle $ is continuous if and only if $f_i\circ g$ is continuous for all $i\in\mathcal{I}$.
\end{proposition}

\begin{proof}
	Since the composition of continuous function is continuous, the first implication holds. For the converse, given a net $\varphi\in\textsc{Nets}(Z)$ such that $\varphi\to_{L'} x$. Since $f_i\circ g$ is continuous, we have $f_i\circ(g\circ\varphi)\to_{L_i} f_i(g(x))$ for all $i\in\mathcal{I}$. It follows that $g\circ\varphi\to_{\bigvee f_i^{-1}L_i} g(x)$. This means $g$ is continuous. 
\end{proof}

\begin{definition}
	Let  $Y$ be a set, $\langle X,L_i\rangle$ be a preconvergence and $f_i:X_i\to Y$ be a function for each $i\in\mathcal{I}$. The  \textbf{final preconvergence} on $Y$ is the finest preconvergence such that each $f_i$ is continuous. We denote this preconvergence by $\bigwedge_{i\in\mathcal{I}}f_iL_i$. In case that $|\mathcal{I}|=1$, we denote this preconvergence by $fL$. \label{symbol:finalPreconv}
\end{definition}

\begin{proposition}
	Let $\langle X,L\rangle$ be a preconvergence space and $f:X\to Y$ be a function. A net $\varphi\in\textsc{Nets}(Y)$ is such that $\varphi\to_{fL} y$ if and only if there is $x\in X$ such that $f(x)=y$ and a net $\psi\in\textsc{Nets}(X)$ such that $\psi \to_L x$ and $f(\psi^{\uparrow})=\varphi^{\uparrow}$.
\end{proposition}

\begin{proof}
Consider the preconvergence $\lambda$ in $Y$ such that a net $\varphi\in\textsc{Nets}(Y)$ $\lambda$-converges to $y\in Y$ if and only if there is a net $\psi\in\textsc{Nets}(X)$ such that $\psi \to_L x$, $f(x)=y$ and $f(\psi^{\uparrow})=\varphi^{\uparrow}$. Clearly $\lambda$ makes $f$ continuous. If $\lambda'$ is a preconvergence in $Y$ wich makes $f$ continuous and $\varphi\in\textsc{Nets}(Y)$ is a net such that $\varphi\to_{\lambda} y$. There is a net $\psi\in\textsc{Nets}(X)$ such that $\psi\to_L x$, $f(x)=y$ and $f(\psi^{\uparrow})=\varphi^{\uparrow}$. Since $f:\langle X,L\rangle\to\langle Y,L'\rangle$ is continuous, we have $f\circ\psi\to_{\lambda'}f(x)$. It follows that $\varphi\to_{\lambda'}$. Then $\lambda\geq \lambda'$. This proves that $fL=\lambda$.
\end{proof}

\begin{proposition}
	Let be $\langle X_i, L_i\rangle$ a preconvergence space and $f_i:X_i\to Y$ be a function for each $i\in\mathcal{I}$. A net $\varphi\in\textsc{Nets}(Y)$ is such that $\varphi \to_{\bigwedge_{i\in\mathcal{I}}f_iL_i} y$ if and only if there is a net $\psi\in\textsc{Nets}(X_j)$ such that $\psi \to_{L_j} x$, $f(\varphi^{\uparrow})= \psi^{\uparrow}$ and $f_j(x)=y$ for some $j\in\mathcal{I}$.
\end{proposition}

\begin{proof}
	Follows from definition of infimum of preconvergences.
\end{proof}

\begin{example}
	Given an equivalence relation $\sim$ on a preconvergence space $X$. The quotient convergence on $X/\sim$ is the final preconvergence such that the canonical projection $$\begin{array}{rcl}
		\pi: X & \to     & X/\sim \\ 
		x & \mapsto & 	[x] \\
	\end{array}$$  is continuous. 
\end{example}

\begin{example}
	Let  $X_i$ be a preconvergence space for each $i\in\mathcal{I}$. Recall that the coproduct, or disjoint union, of this family is the set $\coprod_{i\in\mathcal{I}}X_i=\{\langle x,i\rangle :x\in X_i\}$. The coproduct preconvergence on $\coprod_{i\in\mathcal{I}}X_i$ is the final preconvergence such that each inclusion $$\begin{array}{rcl}
		i_j: X_j & \to     & \coprod_{i\in\mathcal{I}}X_i \\ 
		x & \mapsto & 	\langle x,j\rangle \\
	\end{array}$$ is continuous. \label{symbol:coproduct}
\end{example}

\begin{proposition}
	Let $\langle X_i,L_i\rangle $ be a preconvergence space and $f_i:X_i\to Y$ be a function for each $i\in\mathcal{I}$. A function $g: \langle Y,\bigwedge f_iL_i\rangle \to \langle Z,L'\rangle $ is continuous if and only if $g\circ f_i$ is continuous for all $i\in\mathcal{I}$.
\end{proposition}

\begin{proof}
	Since composition of continuous function is continuous, the first implication holds. For the converse, if $\varphi\in\textsc{Nets}(Y)$ is a net such that $\varphi\to_{\bigwedge f_iL_i} y$, there is a net $\psi\in\textsc{Nets}(X_j)$ such that $\psi \to_{L_j} x$, $f_j(\varphi^{\uparrow})= \psi^{\uparrow}$ and $f_j(x)=y$ for some $j\in\mathcal{I}$. Since $g\circ f_j$ is continuous, we have $(g\circ f_j)\circ \psi\to_{L'} g(y)$. Note that $g((f_j\circ\psi)^{\uparrow})=(g\circ\varphi)^{\uparrow} $. It follows that $g\circ \varphi\to_{L'} g(y)$. This proves that $g$ is continuous.
\end{proof}

\begin{lemma}
	Let $\mathcal{S}$ be a nonempty family of preconvergences on a set $X$. If $\apair{X,L}$ is a limit space for every $L\in\mathcal{S}$, then $\apair{X,\bigvee\mathcal{S}}$ is also a limit space.
\end{lemma}

\begin{proof}
	\ 
	\begin{enumerate}

		\item (Centerness) If $\varphi=\apair{x}_d$ is a constant net, it happens that $\varphi\to_L x$ for every $L\in\mathcal{S}$. This means that $\varphi\to_{\bigvee\mathcal{S}} x$.
		
		\item (Isotonicity) If  $\psi$ is a subnet of $\varphi$ such that $\varphi\to_{\bigvee\mathcal{S}} x$, for every $L\in\mathcal{S}$ we have $\varphi\to_{L}x$. Then $\psi\to_L x$ for every $L\in\mathcal{S}$. It follows that $\psi\to_{\bigvee \mathcal{S}} x$. 
		
		\item (Stability) If $\rho$ is a mixing of $\varphi$ and $\psi$ such that $\psi,\varphi\to_{\bigvee\mathcal{S}}x$. For every $L\in\mathcal{S}$ we have $\psi,\varphi\to_L x$. It follows that $\rho\to_L x$ for every $L\in\mathcal{S}$. This means that $\rho\to_{\bigvee\mathcal{S}}x$.
		
	\end{enumerate}
\end{proof}
\begin{remark}
	The same can be showed for topological spaces and others, see~\cite{Preusspaper}.
\end{remark}

\begin{remark}

Final structures do not preserve the type of preconvergence space. For example, the quotient of a limit space may not be a limit space. Given $x,y\in[0,1] $ we declare that $x\sim y$ if $x,y\in (0,1)$ and $x=y$ or $x=0$ and $y=1$ or $x=1$ and $y=0$. Notice that $\sim$ is an equivalence relation on $[0,1]$ and $[0,1]/\sim=(0,1)\cup \{\bullet\}$, where $\bullet$ is the equivalence class of $0$ and $1$. Consider the function

$$\begin{array}{rcl}
	f:[0,1] & \to     & (0,1) \cup\{\bullet\}\\ 
	x & \mapsto & 	\begin{cases} 
		x & \text{if } x \neq 0, 1 \\
		\bullet  & \text{otherwise}
	\end{cases} \\
\end{array}$$ 
We consider the final preconvergence induced by $f$ in the quotient. Notice that this preconvergence is not stable. Indeed, let \( \langle x_n\rangle _n \) and \( \langle y_n\rangle _n \) be sequences in \( (0, 1) \) such that \( x_n \to 0 \) and \( y_n \to 1 \). Since $f$ is continuous, it follows that \( f(x_n) \to f(0) = \bullet  \) and \( f(y_n) \to f(1) = \bullet \). The mixing \( \langle z_n\rangle _n \) of $\langle x_n\rangle_n$ and $\langle y_n\rangle_n$ defined by

\[
z_n =
\begin{cases} 
	f(x_n)=x_n  & \text{if } n \text{ is even} \\
	f(y_n)=y_n & \text{if } n \text{ is odd}
\end{cases}
\]
does not converges to \( \bullet \) in $(0,1)\cup\{\bullet\}$. If \( \varphi \in \textsc{Nets}([0, 1]) \) is a net converging to \( 0 \), there is \( a' \in \text{dom}(\varphi) \) such that \( \varphi_a < \frac{1}{2} \) whenever \( a \geq a' \). Notice that \( f \circ \varphi[a'^{ \uparrow} ]\) contains no tail set of \( \langle z_n\rangle _n \), since every tail of $\langle z_n \rangle_n$ contains points greater than $\frac{1}{2}$. This means that \( f(\varphi^\uparrow) \not\subseteq \langle z_n\rangle_n^\uparrow \). Then $\langle z_n\rangle_n\not\to \bullet$. A similar argument holds in the case that $\varphi$ converges to $1$.
\end{remark}

\section{Topological and limit modification}

\label{modifiers}

Notice that Proposition \ref{open} enables us to characterize the conditions under which a point belongs to the topological interior of a set. In the context of preconvergence spaces, we also introduce the concept of inherence.

\begin{definition}[Adapted from~\cite{schechter1996handbook}]
		 Let $\langle X,L\rangle$ be a preconvergence space. The \textbf{$L$-inherence }of a subset $S\subseteq X$ is the set $$\hbox{inh}_L(S)=\{x\in X:\forall \varphi\in\textsc{Nets}(X)(\varphi\to_L x\implies S\in\varphi^{\uparrow})\}$$
		 We say that $S$ is \textbf{$L$-open} if
		$S\subseteq\hbox{inh}_L(S)$. We denote denote by $\mathcal{O}(L)$ the family of $L$-opens of $X$.
\end{definition}

\begin{remark}
Inherence generalizes the concept of interior. Indeed, given a topological space $\langle X,\tau\rangle$  and a subset $S\subseteq X$. Considering $L=\lim_\tau$, by Proposition \ref{open} we have $\hbox{inh}_L(S)=\hbox{int}_\tau(S)$, where $\hbox{int}_\tau(S)$ denotes the topological interior of $S$.
\end{remark}

\begin{remark}
	In topological spaces, it is true that the interior of a subset is contained in the subset. This does not generally hold for preconvergence spaces. But it does hold if $\langle X, L\rangle $ is a preconvergence space where $L$ is centered. Indeed, given $x\in\hbox{inh}_L(S)$, since $\mathfrak{u}_x\to_L x$, we have $S\in\mathfrak{u}_x$. It follows that $x\in S$. This means that $\hbox{inh}_L(S)\subseteq S$.
	
\end{remark}

But what is the interest in defining inherence? So far, we have generalized the concept of a topological space, where we associate its corresponding limit space with each topological space. However, it is often good to work with a topology, so the question arises: Is there a good way to transition from a preconvergence space to a topological space? The answer is yes, and the term $L$-open is suggestive.
\label{symbol:inherence}

\begin{proposition}
	
	The family $\mathcal{O}(L)$ is a topology on $X$, called the topology induced by $L$. We refer as $\langle X,\mathcal{O}(L)\rangle$ to topological modification of $\langle X,L\rangle$. \label{symbol:inducedTopology}
	
\end{proposition}

\begin{proof}
	\ 
	\begin{enumerate}
		\item It is easy to check that $X$ and $\emptyset$ are $L$-open.
		\item Let $U,V\in\mathcal{O}(L)$ be $L$-open sets, Given a net $\varphi\in\textsc{Nets}(X)$ such that $\varphi\to_L x\in U\cap V$, since $U$ and $V$ are $L$-open, it happens that $U,V\in\varphi^{\uparrow}$. Then $U\cap V\in \varphi^{\uparrow}$. This means that $U\cap V\in\mathcal{O}(L)$.
		\item Let $\mathcal{U}\subseteq\mathcal{O}(L)$ be a family of $L$-open sets. Given a net $\varphi\in\textsc{Nets}(X)$ such that $\varphi\to_L x\in\bigcup\mathcal{U}$, there is $U\in\mathcal{U}$ such that $x\in U$. Since $U$ is $L$-open, it happens that $U\in\varphi^{\uparrow}$. Notice that $\bigcup\mathcal{U}\in\varphi^{\uparrow}$ because $U\subseteq\bigcup \mathcal{U}$. This proves that $\bigcup\mathcal{U}\in\mathcal{O}(L)$.
	\end{enumerate}
\end{proof}

\begin{example}
	\label{sphere}
	Recall that in Example \ref{circle} we define a preconvergence $L$ in the circle $\mathbb{S}^1$. Now, we show that the topological modification $\mathcal{O}(L)$ is the chaotic topology. Let $S\subseteq \mathbb{S}^1$ be a nonempty $L$-open set, suppose that there is $y\in\mathbb{S}^1\setminus S$ such that $||x-y||\leq r$ for some $x\in S$. In this case $\mathfrak{u}_y\to_L x$, but $S\notin \mathfrak{u}_y$. Then for all $y\in \mathbb{S}^1$ such that $||x-y||\leq r$ for some $x\in S$ it happens that $y\in S$. Moreover, notice that for all $y\in \mathbb{S}^1$ there is a sequence $\langle x_n\rangle_n$ of points in $S$ such that $||x_m-y||\leq r$ for some $m\in\mathbb{N}$. Indeed, denoting by $B(x,\frac{r}{2})$ the open ball with center $x$ and radius $\frac{r}{2}$, we see that there is $x_0\in B(x,\frac{r}{2})\cap \mathbb{S}^1$, as well as there is $x_1\in B(x_0,\frac{r}{2})$. This argument allows us to create a sequence $\langle x_n\rangle_n$ such that $x_n\in S$ for all $n\in\mathbb{N}$, $x_n\in B(x_{n-1},\frac{r}{2})\cap \mathbb{S}^1$ and  $||x_m-y||\leq r$ for some $m\in\mathbb{N}$, see Figure \ref{arg}. It follows that $y\in S$. This proves that $\mathbb{S}^1=S$.
	
	\begin{figure}[H]
		\centering
		\includegraphics[width=5cm]{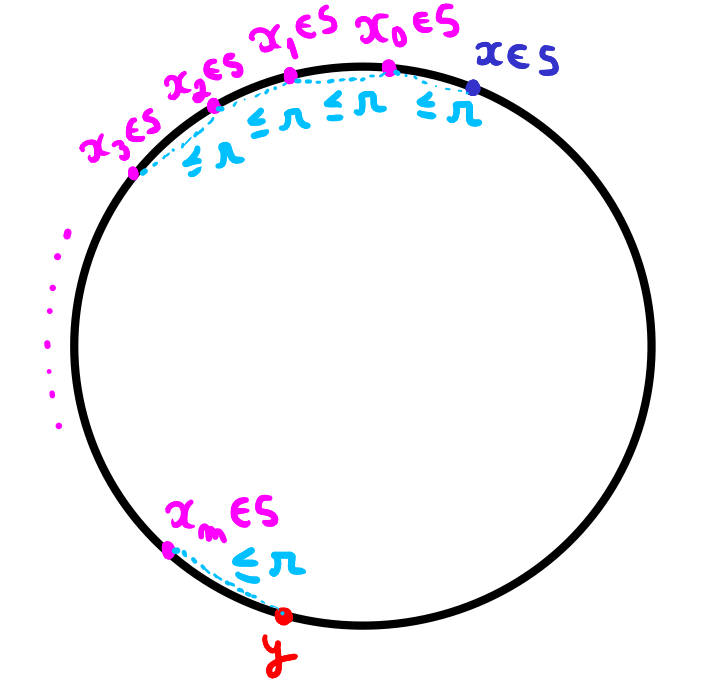}
		\caption{A pictorial description of this argument}
		\label{arg}
	\end{figure}
	
	Since the chaotic topology induces the chaotic preconvergence, in particular, we see that the convergence induced by $\mathcal{O}(L)$ is not $L$ in general. Indeed, this preconvergence is not topological by Proposition \ref{topological}
\end{example}

\begin{example}
	The topological modification of the sequential convergence of Example \ref{215} is the standard topology of the real line. Indeed, let $U\subseteq \mathbb{R}$ be an usual open set and $\varphi\in\textsc{Nets}(\mathbb{R})$ be a net such that $\varphi\to_{\hbox{Seq}} x\in U$. It happens that $\varphi\to_{\mathbb{R}} x$. Since $U$ is open, it follows that $U\in\varphi^{\uparrow}$. This proves that $U$ is $\hbox{Seq}$-open. Conversely, if $U\subseteq\mathbb{R}$ is not open there is a net $\varphi\in\textsc{Nets}(\mathbb{R})$ such that $\varphi\to_{\mathbb{R}}x$ but $U\notin\varphi^{\uparrow}$. Notice that $U\notin\mathcal{N}_x$, then there is a sequence $\langle x_n\rangle_n$ in $\mathbb{R}\setminus U$ such that $x_n\to_{\mathbb{R}}x$ but $U\notin \langle x_n\rangle_n^{\uparrow}$. This means that $U$ is not $\hbox{Seq}$-open.
\end{example}


\begin{remark}
Note that if a function between preconvergence spaces is continuous, then it is continuous as a function between their induced topological spaces. Indeed, let \( f: \langle X, L \rangle \to \langle Y, L' \rangle \) be a continuous function and \( V \subseteq Y \) be an \( L' \)-open set. Given a net \( \varphi \in \textsc{Nets}(X) \) such that \( \varphi \to_{L} x \in f^{-1}[V] \), by continuity of \( f \), it happens that \( f \circ \varphi \to_{L'} f(x) \in V \). Since \( V \) is \( L' \)-open, it follows that \( V \in (f \circ \varphi)^{\uparrow} \), which implies that \( f^{-1}[V] \in \varphi^{\uparrow} \). This proves that \( f^{-1}[V] \) is \( L \)-open. Then \( f: \langle X, \mathcal{O}(L) \rangle \to \langle Y, \mathcal{O}(L') \rangle \) is continuous.
 So, this way of topologizing a preconvergence space is so good that it is functorial.  And it is worth noting that in this process we are not losing any continuous functions we had before. So far, we have two functors.

\[
\begin{array}{cc}
	\begin{tikzcd}[ampersand replacement=\&, cramped]
		{\textsc{PrConv}} \& {\textsc{Top}} \\
		{\langle X,L\rangle} \& {\langle X,\mathcal{O}(L)\rangle} \\
		{\langle Y,L'\rangle} \& {\langle Y,\mathcal{O}(L')\rangle}
		\arrow["{\mathcal{O}(\bullet)}", from=1-1, to=1-2]
		\arrow[from=2-1, to=2-2]
		\arrow["f", from=2-1, to=3-1]
		\arrow["f", from=2-2, to=3-2]
		\arrow[from=3-1, to=3-2]
	\end{tikzcd}   & 
	\begin{tikzcd}[ampersand replacement=\&, cramped]
		{\textsc{Top}} \& {\textsc{PrConv}} \\
		{\langle X,\tau\rangle} \& {\langle X,\lim_\tau\rangle} \\
		{\langle Y,\tau'\rangle} \& {\langle Y,\lim_{\tau'}\rangle}
		\arrow["{\lim(\bullet)}", from=1-1, to=1-2]
		\arrow[from=2-1, to=2-2]
		\arrow["f"', from=2-1, to=3-1]
		\arrow["f"', from=2-2, to=3-2]
		\arrow[from=3-1, to=3-2]
	\end{tikzcd}
\end{array}
\]

\end{remark}

Recall that in Example \ref{215} we see that not every preconvergence is induced by a topology. But what is the condition for this to happen? The topology induced by a preconvergence gives us an answer to this question.

\begin{proposition}
	\label{topological}
	A preconvergence $L$ is topological if and only if $L=\lim_{\mathcal{O}(L)}$.
\end{proposition}
\begin{proof}
		Suppose that $L$ is topological. There is a topology $\tau$ such that $\lim_\tau=L$. Since $\lim_{\mathcal{O}(L)}\geq L$, we have $\lim_{\mathcal{O}(L)}\geq \lim_\tau$. Note that for each $x\in X$ the inclusion $\mathcal{N}_{x,\mathcal{O}(L)}\subseteq \mathcal{N}_{x,\tau}$ holds. This proves that $\lim_\tau\geq\lim_{\mathcal{O}(L)}$. The converse is imediate.
\end{proof}

Now we present how to transition between the categories of preconvergence spaces in a functorial way. The idea is to add to the preconvergence what is missing to have a convergence or limit space. For example, a convergence space fails to be a limit space if its preconvergence is not stable, so we add to the preconvergence the convergence of mixings.

\begin{definition}
	Let $\langle X,L\rangle$ be a convergence space. The \textbf{limit modification }of $L$ is the preconvergence $\sqcup(L)$ such that $\varphi\to_{\sqcup(L)} x$ if and only if there are finitely many nets $\varphi_0,\varphi_1,\cdots,\varphi_n\in\textsc{Nets}(X)$ such that $\varphi_i\to_L x$ for all $0\leq i\leq n$ and $\bigcap_{0\leq i\leq n} \varphi_i^{\uparrow}\subseteq\varphi^{\uparrow}$. \label{symbol:limitMod}
\end{definition}

The same idea can be applied to obtain a convergence space from a preconvergence space, see~\cite{schechter1996handbook}. We present only the limit modification since limit spaces are the focus of this work.
This is the best way to create a limit space from a convergence space, the reader can check that $\sqcup(L)$ is the infimum of all limit convergences which are coarser than $L$ . Moreover, the following correspondence defines a functor.

\[\begin{tikzcd}[ampersand replacement=\&,cramped]
	{\textsc{Conv}} \&\& {\textsc{Lim}} \\
	{\langle X,L\rangle} \&\& {\langle X,\sqcup(L)\rangle} \\
	{\langle Y,L'\rangle} \&\& {\langle Y,\sqcup(L')\rangle}
	\arrow["{\sqcup(\bullet)}", from=1-1, to=1-3]
	\arrow[maps to, from=2-1, to=2-3]
	\arrow["f"', from=2-1, to=3-1]
	\arrow["f"', from=2-3, to=3-3]
	\arrow[maps to, from=3-1, to=3-3]
\end{tikzcd}\]
Let us show that indeed $\sqcup(\bullet)$ is a functor.

\begin{enumerate}[i)]
\item $\langle X,\sqcup(L)\rangle$ is a limit space: Centerness and isotonicity follow easily. For stability, let  $\varphi,\psi\in\textsc{Nets}(X)$ be nets  such that $\varphi,\psi\to_{L}x$ and $\rho\in\textsc{Nets}(X)$ be a mixing of $\varphi$ and $\psi$. There are nets $\varphi_0,\varphi_1,\cdots,\varphi_n,\psi_0,\psi_1,\cdots,\psi_m\in\textsc{Nets}(X)$ such that $\bigcap_{0\leq i\leq m} \varphi_i^{\uparrow}\subseteq\psi^{\uparrow}$, $\bigcap_{0\leq i\leq n} \psi_i^{\uparrow}\subseteq\psi^{\uparrow}$ and $\varphi_i,\psi_j\to_L x$ for $0\leq i\leq n$ and $0\leq j\leq m$. Notice that $$\bigcap_{0\leq i\leq n} \varphi_i^{\uparrow}\cap \bigcap_{0\leq j\leq m} \psi_j^{\uparrow}\subseteq \varphi^{\uparrow}\cap\psi^{\uparrow}\subseteq\rho^{\uparrow}$$ This means that $\rho\to_{\sqcup(L)} x$.

\item If $f:\langle X,L\rangle\to\langle Y,L'\rangle$ is continuous and $\varphi\in\textsc{Nets}(X)$ is a net such that $\varphi\to_{\sqcup(L)} x$. There are nets finitely many nets $\varphi_0,\varphi_1,\cdots,\varphi_n\in\textsc{Nets}(X)$ such that $\varphi\to_{L}x$ and $\bigcap_{0\leq i\leq n} \varphi_i^{\uparrow}\subseteq \varphi^{\uparrow}$. Since $f$ is continuous, it happens that $f\circ \varphi_i\to_{L'}f(x)$ for $0\leq i\leq n$. Moreover, $$\bigcap_{0\leq i\leq n}(f\circ \varphi_i)^{\uparrow}\subseteq (f\circ \varphi)^{\uparrow}$$ It follows that $f\circ\varphi\to_{\sqcup(L')}f(x)$. This proves that $f:\langle X,\sqcup(L)\rangle \to\langle Y,\sqcup(L')\rangle$ is continuous. 

\end{enumerate}

\section{The continuous convergence}

\label{sec2.4}
\begin{definition}
	Let \( \mathcal{C} \) be a category. For two objects \( A \) and \( B \) in \( \mathcal{C} \), an \textbf{exponential object} \( B^A \) satisfies the following universal property:
	For any object \( C \) and any morphism \( f: C \times A \to B \), there is a unique morphism \( \tilde{f}: C \to B^A\) such that the following diagram commutes:
	
	\	
	\[\begin{tikzcd}[ampersand replacement=\&,cramped]
		{B^A} \& {B^A\times A} \& B \\
		C \& {C\times A}
		\arrow["\hbox{ev}", from=1-2, to=1-3]
		\arrow["{\tilde{f}}", dashed, from=2-1, to=1-1]
		\arrow["{\tilde{f}\times 1_A}", dashed, from=2-2, to=1-2]
		\arrow["f"', from=2-2, to=1-3]
	\end{tikzcd}\]
	where \( \hbox{ev}: B^A \times A \to B \) is the evaluation morphism.
	
\end{definition}
In this section we show that the category of preconvergence spaces has exponential objects, provided by the so-called continuous convergence, and consequently, it is suitable for homotopy theory. Exponential objects play a fundamental role in category theory. They are a generalization of familiar concepts of exponentiation, such as powers in algebra and function spaces in analysis, to the categorical context. Categories with exponential objects are known as cartesian closed categories. In the preface of~\cite{schechter1996handbook}, Dolecki compares convergence spaces to complex numbers in the sense that their category completes the exponential objects missing in the category of topological spaces, just as the complex numbers complete the missing roots of polynomials. Here, we don't go too deeply into the discussion of continuous convergence, presenting only the essential results. However, this convergence is of great theoretical interest, and readers seeking more details can check~\cite{ref4} and~\cite{binz}.
\begin{definition}

Let $X$ and $Y$ be preconvergence spaces. We say that a filter $\mathcal{F}$ on $\mathcal{C}(X,Y)$ \textbf{converges continuously}, or \textbf{$\mathcal{C}$-converges}, to a continuous function $f \in \mathcal{C}(X,Y)$ if for every $x \in X$ and $\mathcal{H} \in \textsc{Fil}^*(X)$ it happens that $\mathcal{F}(\mathcal{H}) \to f(x)$ whenever $\mathcal{H} \to x$, where $\mathcal{F}(\mathcal{H})$ denotes the filter generated by sets of the form $FH = \{ f(h) \in Y : f \in F, h \in H \}$ for $F \in \mathcal{F}$ and $H \in \mathcal{H}$.

\end{definition}

We initially defined continuous convergence using filters, but we will now translate this to nets. In some propositions, two proofs will be provided—one using filters and the other using nets. The goal is to illustrate how nets offer a more intuitive understanding of this convergence, while filters tend to complicate many arguments. The reader who prefers to look only at the proof with nets will not be missing anything.
\label{symbol:function}
\begin{proposition}
	A net $\langle f_a\rangle_{a\in\mathbb{A}}$ in $\mathcal{C}(X,Y)$ is such that $f_a\to_{\mathcal{C}} f$ if and only if  the net $\langle f_a(x_b)\rangle_{\langle a,b\rangle \in\mathbb{A}\times \mathbb{B}}$ converges to $f(x)$ whenever $\langle x_b\rangle _{b\in\mathbb{B}}$ is a net in $X$  such that $x_b\to x$.
\end{proposition}
\begin{proof}
	Suppose that $\langle f_a\rangle_a\in\textsc{Nets}(\mathcal{C}(X,Y))$ is a net such that $f_a\to_{\mathcal{C}}f$. Let $\langle x_b\rangle _{b\in\mathbb{B}}$ be a net in $X$  such that $x_b\to x$. Consider $\mathcal{H}=\langle x_b\rangle _b^{\uparrow}$, then $\mathcal{H}\to x$. Since $f_a\to_{\mathcal{C}} f$, taking $\mathcal{F}=\langle f_a\rangle _a^{\uparrow}$ we have $\mathcal{F}(\mathcal{H})\to f(x)$. Note that $\mathcal{F}(\mathcal{H})=\langle f_a(x_b)\rangle _{\langle a,b\rangle } ^{\uparrow}$. This means that $f_a(x_b)\to f(x)$.  The proof of the converse is similar.
\end{proof}

\begin{figure}[H]
	\centering
	\includegraphics[width=9cm]{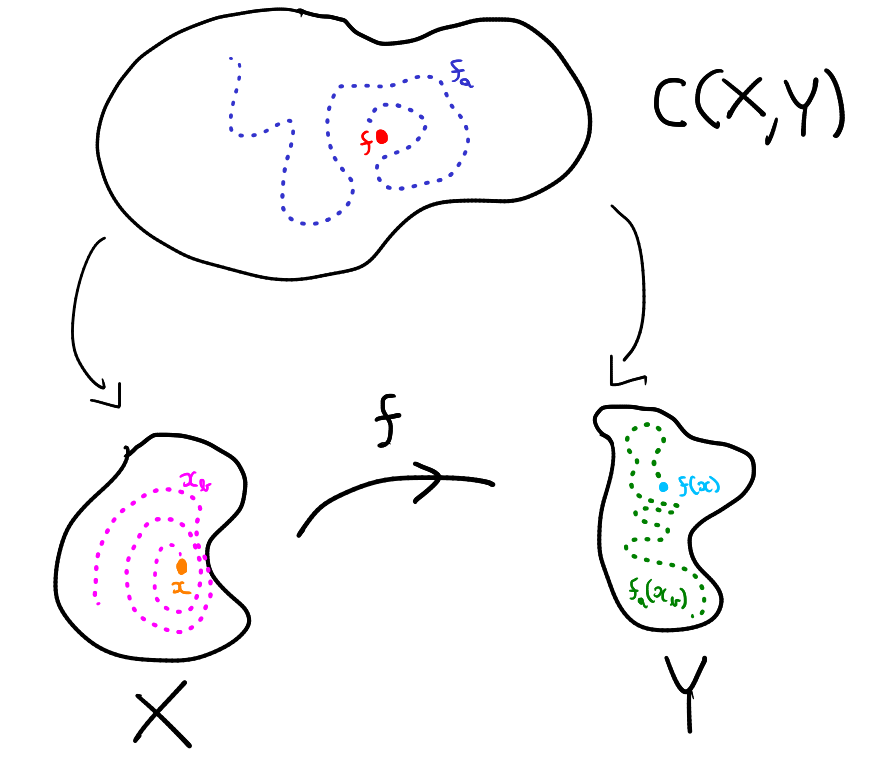}
	\caption{An illustration of how continuous convergence happens.}
\end{figure}

\begin{proposition}
If $X$ and $Y$ are limit spaces., then $\langle \mathcal{C}(X,Y),\mathcal{C}\rangle $ is a limit space.
\end{proposition}

\begin{proof}[Proof 1 (Filters)]
	\
	
	\begin{enumerate}
		\item (Centerness )Let  $g \in \mathcal{C}(X, Y)$ be a continuous function and $\mathcal{H}$ be a proper filter on $X$ such that $\mathcal{H} \to x$. Note that $g(\mathcal{H}) \subseteq \mathfrak{u}_g(\mathcal{H})$. Since $g$ is continuous, it happens that $g(\mathcal{H}) \to g(x)$. By isotonicity of the preconvergence on $Y$, we have $\mathfrak{u}_g(\mathcal{H}) \to g(x)$. Then $\mathfrak{u}_g \to_{\mathcal{C}} g$.
		 
		\item (Isotonicity) Let $\mathcal{A}$ and $\mathcal{B}$ be propers filters on $\mathcal{C}(X,Y)$ such that $\mathcal{A}\subseteq\mathcal{B}$ and $\mathcal{A}\to _{\mathcal{C}} f$. Given a proper filter $\mathcal{H}$ on $X$ such that $\mathcal{H}\to x$, it hapens that $\mathcal{A}(\mathcal{H})\to f(x)$. Since $\mathcal{A}\subseteq\mathcal{B}$, it follows that $\mathcal{A}(\mathcal{H})\subseteq\mathcal{B}(\mathcal{H})$. By isotonicity of the preconvergence on $Y$, $\mathcal{B}(\mathcal{H})\to f(x)$. It follows that $\mathcal{B}\to_{\mathcal{C}}f$.
		\item (Stability) Let $\mathcal{A}$ and $\mathcal{B}$ be propers filters on $\mathcal{C}(X,Y)$ such that $\mathcal{A}\to_{\mathcal{C}}f$ and $\mathcal{B}\to _{\mathcal{C}} f$. Given a proper filter $\mathcal{H}$ in $X$ such that $\mathcal{H}\to x$, it happens that $\mathcal{A}(\mathcal{H})\to f(x)$ and $\mathcal{B}(\mathcal{H})\to f(x)$. Note  that $(\mathcal{A}\cap\mathcal{B})(\mathcal{H})=\mathcal{A}(\mathcal{H})\cap \mathcal{B}(\mathcal{H})$. Since the preconvergence on $Y$ is stable, it follows that $\mathcal{A}(\mathcal{H})\cap \mathcal{B}(\mathcal{H})\to f(x)$. 
	\end{enumerate}
\end{proof}
\begin{proof}[Proof 2 (Nets)]
	\
	\begin{enumerate}
		\item (Centerness) Given $\langle f_a\rangle_{a\in\mathbb{A}}\in\textsc{Nets}(\mathcal{C}(X,Y)$ a net such that $f_a=f$ for all $a\in\mathbb{A}$. Let us show that $f_a\to_{\mathcal{C}} f$. If $\langle x_b\rangle_{b\in\mathbb{B}}$ is a net in $X$ such that $x_b\to x$, notice that $\langle f_a(x_b)\rangle_{\langle a,b\rangle\in\mathbb{A}\times\mathbb{B}}$ is subnet of $\langle f(x_b)\rangle_{b\in\mathbb{B}}$. Since $f$ is continuous, it follows that $f(x_b)\to f(x)$. By isotonicity of the preconvergence it happens that $f_a(x_b)\to f(x)$. This means that $f_a\to_{\mathcal{C}} f$.
		\item (Isotonicity) If $\langle g_d\rangle_{d\in\mathbb{D}}$ is a subnet of $\langle f_a\rangle_{a\in\mathbb{A}}$ and $f_a\to_{\mathcal{C}} f$. Let $\langle x_b\rangle_{b\in\mathbb{B}}$ be a net on $X$ such that $x_b\to x$, we have $f_a(x_b)\to f(x)$. Note that $\langle g_d(x_b)\rangle_{\langle d,b\rangle\in\mathbb{D}\times \mathbb{B}}$ is a subnet of $\langle f_a(x_b)\rangle_{\langle a,b\rangle\in\mathbb{A}\times \mathbb{B}}$. It follows that $g_d(x_d)\to f(x)$. This proves that $g_d\to_{\mathcal{C}} f$.
		\item(Stability) If $\langle f_a\rangle_{a\in\mathbb{A}}$ and $\langle g_a\rangle_{a\in\mathbb{A}}$ are nets on $\mathcal{C}(X,Y)$ such that $f_a,g_a\to_{\mathcal{C}} f$ and $\langle \rho_a\rangle_{a\in\mathbb{A}}$ is a mixing of this nets. Given a net $\langle x_b\rangle_{b\in\mathbb{B}}$ on $X$, such that $x_b\to x$, we have $f_a(x_b),g_a(x_b)\to f(x)$. Since $\langle \rho_a(x_b)\rangle_{\langle a,b\rangle \in\mathbb{A}\times\mathbb{B}}$ is a mixing of $\langle f_a(x_b)\rangle_{\langle a,b\rangle \in\mathbb{A}\times\mathbb{B}}$ and $\langle g_a(x_b)\rangle_{\langle a,b\rangle \in\mathbb{A}\times\mathbb{B}}$, it follows that $\rho_a(x_b)\to f(x)$. This proves that $\rho_a\to_{\mathcal{C}} f$.
	\end{enumerate}
\end{proof}
\begin{proposition}
	The continuous convergence is the coarsest preconvergence such that the evaluation function
	$$\begin{array}{rcl}
		\hbox{ev}: \mathcal{C}(X,Y)\times X & \to     & Y \\ 
		\langle f,x\rangle  & \mapsto & 	f(x)\\
	\end{array}$$ is continuous.
\end{proposition}

\begin{proof}[Proof 1 (Filters)]
	Let $\mathcal{G}$ be a filter on $\mathcal{C}(X,Y)\times X$ such that $\mathcal{G}\to \langle f,x\rangle $ in the product preconvergence. Our goal is show that $\hbox{ev}$ is continuous, that is, $\hbox{ev}(\mathcal{G})\to f(x)$ in $Y$. Consider the filters $\mathcal{F}=\pi_{\mathcal{C}(X,Y)}(\mathcal{G})$ and $\mathcal{H}=\pi_X(\mathcal{G})$. By definition of product preconvergence, it happens that $\mathcal{F}\to_{\mathcal{C}} f$ and $\mathcal{H}\to x$. Since $\mathcal{F}\to_{\mathcal{C}} f$, it follows that  $\mathcal{F}(\mathcal{H})\to f(x)$ in $Y$. Let us show that $\mathcal{F}(\mathcal{H})\subseteq \hbox{ev}(\mathcal{G})$. If $A\in\mathcal{F}(\mathcal{H})$, there are $F\in\mathcal{F}$ and $H\in\mathcal{H}$ such that $FH\subseteq A$, where $F=\pi_{\mathcal{C}(X,Y)}[G_0]$ and $H=\pi_{X}[G_1]$ for some $G_0,G_1\in\mathcal{G}$. Then $G=G_0\cap G_1\in \mathcal{G}$ and $\hbox{ev}[G]\subseteq FH$. It follows that $A\in \hbox{ev}(\mathcal{G})$. This means that $\mathcal{F}(\mathcal{H})\subseteq\hbox{ev}(\mathcal{G})$. Since the preconvergence on $Y$ is isotone, this proves that $\hbox{ev}$ is continuous. It remains to show that the continuous convergence is the coarsest preconvergence such that $\hbox{ev}$ is continuous. Let $\lambda$ be a preconvergence such that $\hbox{ev}$ is continuous and $\mathcal{F}$ be a filter on $\mathcal{C}(X,Y)$ such that $\mathcal{F}\to_{\lambda} f$. Given a filter $\mathcal{H}$ in $X$ such that $\mathcal{H}\to x$, consider the filter  $$\mathcal{F}\otimes \mathcal{H}=\{F\times H:F\in\mathcal{F},H\in\mathcal{H}\}^{\uparrow}$$
	
	Note that $\hbox{ev}(\mathcal{F}\otimes \mathcal{H})=\mathcal{F}(\mathcal{H})$, $\mathcal{F}\subseteq\pi_{\mathcal{C}(X,Y)}(\mathcal{F}\otimes \mathcal{H})$ and $\mathcal{H}\subseteq \pi_X(\mathcal{F}\otimes \mathcal{H})$. Since the preconvergences on $X$ and $Y$ are isotone, we have $\mathcal{F}\otimes\mathcal{H}\to \langle f,x\rangle $. Follows from continuity of $\hbox{ev}$ that $\hbox{ev}(\mathcal{F}\otimes \mathcal{H})=\mathcal{F}(\mathcal{H})\to f(x)$. This proves that $\mathcal{F}\to_{\mathcal{C}}f$.
\end{proof}

\begin{proof}[Proof 2 (Nets)]
	If $\langle f_d,x_d\rangle_{d\in\mathbb{D}}$ is a net in $\mathcal{C}(X,Y)\times X$ converging to $\langle f,x\rangle$, it happens that $f_d\to_{\mathcal{C}} f$ and $x_d\to x$. It follows that $\langle f_d(x_{d'})\rangle_{d,d'\in\mathbb{D}}$ converges to $f(x)$ in $Y$. Since $\langle f_d(x_d)\rangle_{d\in\mathbb{D}}$ is subnet of $\langle f_d(x_{d'})\rangle_{d,d'\in\mathbb{D}}$, we have $\hbox{ev}(f_d,x_d)=f_d(x_d)\to \hbox{ev}(f,x)=f(x)$. This proves that $\hbox{ev}$ is continuous.  It remains to show that the continuous convergence is the coarsest preconvergence such that $\hbox{ev}$ is continuous. Let $\lambda$ be a preconvergence which makes the evaluation $\hbox{ev}$ continuous, $\langle f_a\rangle _{a\in\mathbb{A}}\in\textsc{Nets}(\mathcal{C}(X,Y))$ be a net such that $f_a\to_{\lambda} f$ and  $\langle x_b\rangle_{b\in\mathbb{B}}$ be a net in $X$ such that $x_b\to x$. The net  $\langle f_a,x_b\rangle_{\langle a,b\rangle\in\mathbb{A}\times \mathbb{B}}$ converges to $\langle f,x\rangle$. Since $\hbox{ev}$ is continuous, it follows that $\hbox{ev}\circ \langle f_a,x_b\rangle_{\langle a,b\rangle\in\mathbb{A}\times \mathbb{B}}=\langle f_a(x_b)\rangle_{\langle a,b\rangle\in\mathbb{A}\times \mathbb{B}}\to f(x)$. Then$f_a\to_{\mathcal{C}} f$. This means that $\lambda \geq \mathcal{C}$.
\end{proof}

\begin{proposition}
	Let $X,Y$ and $Z$ be limit spaces. For every continuous function $h:Z\times X\to Y$ there is a unique continuous function $\tilde{h}:Z\to \mathcal{C}(X,Y)$ such that the following diagram commutes
	
	\[\begin{tikzcd}[ampersand replacement=\&,cramped]
		{\mathcal{C}(X,Y)} \& {\mathcal{C}(X,Y)\times X} \& Y \\
		Z \& {Z\times X}
		\arrow["\hbox{ev}", from=1-2, to=1-3]
		\arrow["{\tilde{h}}", dashed, from=2-1, to=1-1]
		\arrow["{\tilde{h}\times 1_X}", dashed, from=2-2, to=1-2]
		\arrow["h"', from=2-2, to=1-3]
	\end{tikzcd}\]
	\label{exp}
\end{proposition}

\begin{proof}
The existence and uniqueness of $\tilde{h}$ it is guaranteed by the commutativity of the diagram and the universal property of the cartesian product. Then  it remains to show that $\tilde{h}$ is continuous.  Indeed, the function we are looking for is

$$	\begin{array}{rcl}
	\tilde{h}: Z & \to     &\mathcal{C}(X,Y)\\ 
	z & \mapsto & h(z,\bullet)  \\
\end{array}$$
where 
$$	\begin{array}{rcl}
	h(z,\bullet): X & \to     &Y\\ 
	x & \mapsto & h(z,x)  \\
\end{array}$$
We have to prove that $\tilde{h}$ and $h(z,\bullet)$ is continuous for every $z\in Z$. Let $\langle x_b\rangle _{b\in\mathbb{B}}$ be a net on $X$ such that $x_b\to x$. Consider the constant net $\langle z\rangle _{b\in\mathbb{B}}$ on $Z$. It follows that $\langle z_b,x_b\rangle _{b\in\mathbb{B}}\to \langle z,x\rangle $. Since $h$ is continuous, we have $h(z_b,x_b)\to h(z,x)$. But $h(z,\bullet)(x_b)=h(z_b,x_b)=h(z,x_b)$ and $h(z,\bullet)(x)=h(z,x)$, then $h(z,\bullet)(x_b)\to h(z,\bullet)(x)$. This proves that $h(z,\bullet)$ is continuous. Now, we will show that $\tilde{h}$ is continuous. Let $\langle z_a \rangle _{a\in\mathbb{A}}$ be a net on $Z$ and $\langle x_b\rangle_{b\in\mathbb{B}}$ a net on $X$ such that $z_a\to z$ and $x_b\to x$. Consider the nets
$$	\begin{array}{rcl}
	\hat{z}: \mathbb{A}\times \mathbb{B} & \to     &Z\\ 
	\langle a,b\rangle  & \mapsto & z_a  \\
\end{array}$$ 
and 
$$	\begin{array}{rcl}
	\hat{x}: \mathbb{A}\times \mathbb{B} & \to     &X\\ 
	\langle a,b\rangle  & \mapsto & x_b \\
\end{array}$$
Since $\langle z_a\rangle ^{\uparrow}=\langle \hat{z}_{\langle a,b\rangle }\rangle ^{\uparrow}$ and $\langle x_b\rangle ^{\uparrow}=\langle \hat{x}_{\langle a,b\rangle }\rangle ^{\uparrow}$, we have  $\hat{z}_{\langle a,b\rangle }\to z$ and $\hat{x}_{\langle a,b\rangle }\to x$. It follows that $ \langle z_a,x_b\rangle_{\langle a,b\rangle }\to \langle z,x\rangle $. Since $h$ is continuous, we have $\langle h(z_a,x_b)\rangle _{\langle a,b\rangle }\to h(z,x)$. Note that $h(z_a,x_b)=h(z_a,\bullet)(x_b)=\tilde{h}(z_a)(x_b))$, so $\tilde{h}(z_b)(x_a)\to \tilde{h}(z)(x)$. It follows that $\tilde{h}(z_a)\to_{\mathcal{C}} \tilde{h}(z)$. This means that $\tilde{h}$ is continuous.

\end{proof}

\begin{remark}
	\label{exptop}
In the category of Hausdorff topological spaces, the set of continuous functions $\mathcal{C}(X, Y)$ is an exponential object if and only if $X$ is locally compact, see~\cite{schechter1996handbook}.
\end{remark}
\begin{proposition} The map
	\[\begin{split}\circ : \mathcal{C}(X,Y)\times \mathcal{C}(Y,Z)&\to \mathcal{C}(X,Z)\\
		\apair{g,f}&\mapsto f\circ g\end{split}\]
	is continuous for every limit spaces $X$, $Y$ and $Z$.
\end{proposition}

\begin{proof}
	If $\apair{g_a,f_a}_{a\in\mathbb{A}}$ converges to $\apair{g,f}$ in $\mathcal{C}(X,Y)\times\mathcal{C}(Y,Z)$ and $\langle x_b\rangle_{b\in\mathbb{B}}$ converges to $x$ in $X$, then $g_a(x_b)\to g(x)$ in $Y$. It follows that $f_{a'}(g_a(x_b))\to f(g(x))$ in $Z$. Note that $\apair{f_{a}(g_a(x_b))}_{\langle a,b\rangle\in\mathbb{A}\times\mathbb{B}}$ is a subnet of $\apair{f_{a'}(g_a(x_b))}_{\langle a',a,b\rangle \in\mathbb{A}\times\mathbb{A}\times\mathbb{B}}$. Indeed, for a tail $$\{f_{a'}(g_a(x_b)):\langle a',a,b\rangle\geq \langle a_1,a_2,b_1\rangle\} $$ of  $\apair{f_{a'}(g_a(x_b))}_{\langle a',a,b\rangle \in\mathbb{A}\times\mathbb{A}\times\mathbb{B}}$ it happens that $$\{f_{a}(g_a(x_b)):\langle a,b\rangle\geq \langle a_3,b_1\rangle \}\subseteq \{f_{a'}(g_a(x_b)):\langle a',a,b\rangle\geq \langle a_1,a_2,b_1\rangle\} $$ 
where $a_3\geq a_1,a_2$ . It follows that $f_a\circ g_a(x_b)\to f\circ g(x)$. Then $\circ (g_a,f_a)\to \circ (f,g)$. This means that $\circ$ is continuous.
\end{proof}
\section{Topological notions in convergence spaces}

In this section, we explore some topological notions in convergence spaces, such as the concepts of adherence, open cover, and compactness. The key idea is to use Propositions \ref{compact} and \ref{prop1.2.1} to generalize the notions of adherent point and compactness.

\begin{definition}
Let $\mathcal{F} \in \textsc{Fil}^*(X)$ be a filter and $S \subseteq X$ be a subset of $X$. We say that $\mathcal{F}$ and $S$ \textbf{mesh}, denoted by $\mathcal{F} \# S$, if $F \cap S \neq \emptyset$ for every $F \in \mathcal{F}$.
\end{definition}

\begin{example}
	Let $X$ be a topological space and $S\subseteq X$ be a subset, for each $x\in\overline{S}$ it happens that $\mathcal{N}_x$ and $S$ mesh. Indeed, recall that $x\in\overline{S}$ if and only $V\cap S\neq\emptyset$ for every neighborhood $V\in\mathcal{N}_x$ of $x$.
	\label{mesh}
\end{example}
\begin{definition}
	 Let $\langle X,L\rangle$ be a preconvergence space. The \textbf{$L$-adherence} of a subset $S\subseteq X$  is the set $$\hbox{adh}_L(S)=\bigcup_{\mathcal{F}\# S} L(\mathcal{F})$$ \label{symbol:adherence}
\end{definition}

\begin{remark}
	 The $L$-adherence of a subset generalizes the idea of closure of a subset of a topological space. Indeed, for a topological space $\langle X,\tau\rangle$, considering $L=\lim_\tau$, byb the Example \ref{mesh} and the definition of convergence of filters in topological spaces we have $\overline{S}=\hbox{adh}_L(S)$ 
	 
	 \label{reclosed}
\end{remark}

\begin{remark}
The same comment made for inherence applies to adherence. It is not generally true that a subset is contained in its adherence; for this, centrality is also needed. Indeed, suppose that $L$ is centered. Given a point $x\in S\subseteq X$, since $\mathfrak{u}_x\to_L x$ and $\mathfrak{u}_x\# S$, it happens that $x\in\hbox{adh}_L(S)$. This means that $S\subseteq \hbox{adh}_L(S)$.
\end{remark}
\begin{proposition}
	Let $L$ be a isotone preconvergence on $X$ and $S\subseteq X$ be a subset. A point $x\in\hbox{adh}_L(S)$ is adherent if and only if there is a net $\varphi\in\textsc{Nets}(S)$ such that $\varphi\to_L x$.
\end{proposition}

\begin{proof}
	If $x\in\hbox{adh}_L(S)$ there is a filter $\mathcal{F}\in\textsc{Fil}^*(X)$ such that $\mathcal{F}\to_L x$ and $\mathcal{F}\# S$. Since for every $F\in\mathcal{F}$ there is $z_F\in F\cap S$, we can construct  a net $\varphi:\mathcal{F}\to S$ such that $\varphi (F)=z_F$ where $\mathcal{F}$ is directed by reverse inclusion. Notice that $\{z_D\in S: D\subseteq F\}\subseteq F$ for every $F\in\mathcal{F}$, then $\mathcal{F}\subseteq \varphi^{\uparrow}$. By isotonicity, it happens that $\varphi\to_L x$. The converse it is imediate.
\end{proof}

\begin{proposition}
	Let $\langle X,L\rangle$ be a preconvergence space. For subsets $A,B\subseteq X$ holds
	
	\begin{enumerate}[i)]
		\item If $A\subseteq B$, then $\hbox{inh}_L(A)\subseteq \hbox{inh}_L(B)$ and $\hbox{adh}_L(A)\subseteq \hbox{adh}_L(B)$.
		\item $\hbox{adh}_L(A\cup B)=\hbox{adh}_L(A)\cup\hbox{adh}_L(B)$
		\item $\hbox{inh}_L(A\cap B)=\hbox{inh}_L(A)\cap\hbox{inh}_L(B)$
	\end{enumerate}
\end{proposition}

\begin{proof}
	\
	\begin{enumerate}[i)]
		\item Suppose that $A\subseteq B$. Let $x\in\hbox{inh}_L(A)$ be inherent point $\varphi\in\textsc{Nets}(X)$  be a net such that $\varphi\to_L x$, since $x\in\hbox{inh}_L(A)$, it happens that $A\in\varphi^{\uparrow}$. Then $B\in\varphi^{\uparrow}$, that is, $x\in\hbox{inh}_L(B)$. This means that $\hbox{inh}_L(A)\subseteq \hbox{inh}_L(B)$. Now, if $x\in\hbox{adh}_L(A)$, there is a filter $\mathcal{F}\in\textsc{Fil}^*(X)$ such that $\mathcal{F}\to_L x$ and $\mathcal{F}\# A$. Since $A\subseteq B$, we have $\mathcal{F}\# B$. It follows that $x\in\hbox{adh}_L(B)$. This proves that $\hbox{adh}_L(A)\subseteq\hbox{adh}_L(B)$.
		\item Since $A$ and $B$ are subsets of $ A\cup B$, by (i) we have $\hbox{adh}_L(A)\cup\hbox{adh}_L(B)\subseteq\hbox{adh}_L(A\cup B)$. If $x\notin\hbox{adh}_L(A)\cup\hbox{adh}_L(B)$ for every filter $\mathcal{F}\in\textsc{Fil}^*(X)$ such that $\mathcal{F}\to_L x$ it does not occur that $\mathcal{F}\# A $ or $\mathcal{F}\# B$. There are $F_1,F_2\in\mathcal{F}$ such that $F_1\cap A=\emptyset$ and $F_2\cap B=\emptyset$. It follows that $(F_1\cap F_2)\cap (A\cup B)=\emptyset$. Then $x\notin\hbox{adh}_L(A\cup B)$, that is, $\hbox{adh}_L(A\cup B)\subseteq\hbox{adh}_L(A)\cup \hbox{adh}_L(B)$. 
		
		\item Since $A\cap B$ is subset of $A$ and $B$, by (i) we have $\hbox{inh}_L(A\cap B)\subseteq\hbox{inh}_L(A)\cap\hbox{inh}_L(B)$. If $x\in \hbox{inh}_L(A)\cap\hbox{inh}_L(B)$ and $\varphi  \in\textsc{Nets}(X)$ is a net such that $\varphi\to_L x$, it happens that $A\in\varphi^{\uparrow}$ and $B\in\varphi^{\uparrow}$. Then $A\cap B\in\varphi^{\uparrow}$. It follows that $x\in\hbox{inh}_L(A\cap B)$. This proves that $\hbox{inh}_L(A)\cap\hbox{inh}_L(B)\subseteq\hbox{adh}_L(A\cap B)$.
	\end{enumerate}
\end{proof}

\begin{definition}
	Let $\langle X,L\rangle$ be a preconvergence space. We say that a subset $S\subseteq X$ is \textbf{$L$-closed} if $X\setminus S$ is $L$-open.
\end{definition}

\begin{proposition}
	Let $\langle X,L\rangle$ be a preconvergence space such that $L$ is isotone. A subset $S\subseteq  X$ is $L$-closed if and only if $L(\varphi)\subseteq S$ for every net $\varphi\in\textsc{Nets}(S)$.
\end{proposition}
\begin{proof}
	Let $S$ be an $L$-closed and $\varphi\in\textsc{Nets}(S)$ be a net such that $\varphi\to_L x$. Since $X\setminus S$ is $L$-open, it happens that $X\setminus S\subseteq \hbox{inh}_L(X\setminus S)$. If $x\in X\setminus S$, then $X\setminus S\in\varphi^{\uparrow}$. But $\varphi\in\textsc{Nets}(S)$ is a net in $S$. This contradiction tells us that $x\in S$. This means that $L(\varphi)\subseteq S$. Conversely, suppose that $S\subseteq X$ is not $L$-closed, that is, $X\setminus S$ is not $L$-open. There is $x\in X\setminus S$ and a net $\varphi\in\textsc{Nets}(X)$ such that $\varphi\to_L x$, but $X\setminus S\notin\varphi^{\uparrow}$. For all $d\in\hbox{dom}(\varphi)$ there is $\varphi_{d'}\in S$ for some $d'\geq d$. The net $\psi:\hbox{dom}(\varphi)\to X$ such that $\psi_d=\varphi_{d'}$ is a net in $S$. Moreover, $\psi$ is subnet of $\varphi$. Since $L$ is isotone, it happens that $\psi\to_{L} x$. This means that $L(\psi)\not\subseteq S$.
	
\end{proof}

\begin{example}

Let us calculate the adherence of some subsets of the lollipop from Example \ref{lollipop}, with the goal of showing that its limit convergence is not topological.	Notice that $\hbox{adh}_\lambda(S\setminus D)=S$. Indeed, let $x\in S$, if $x\notin D$ consider the constant net $\langle x\rangle_d$ in $S\setminus D$ which converges to $x$. Otherwise, if $x\in D$, since $D$ is countable, for every neighborhood $V\subseteq X$ of $x$ it happens that $V\cap(S\setminus D)\neq\emptyset$. Then $x\in\overline{S\setminus D}$, that is there is a net $\varphi\in\textsc{Nets} (S\setminus D)$ which converges to $x$. But $x\neq p$, by definition of $\lambda$, it happens that $\varphi\to_{\lambda} x$. This means that $x\in\hbox{adh}_\lambda(S\setminus D)$. On other hand, if $x\in\hbox{adh}_\lambda(S\setminus D)$, there is a net $\varphi\in\textsc{Nets}(S\setminus D)$ such that $\varphi\to_{\lambda} x$. Suppose that $x\in C$, since $\varphi\in\textsc{Nets}(S\setminus D)$ it cannot happen that $C\in\varphi^{\uparrow}$ or $\varphi^{\uparrow}\# D$. This contradiction tells us that $x\in S$. Moreover, $\hbox{adh}_\lambda(S)=L$. Indeed,if  $x\in\hbox{adh}_\lambda(S)$, there is a net $\varphi\in\textsc{Nets}(S)$ such that $\varphi\to_\lambda x$. If $x=p$, then $x\in L$. Otherwise, if $x\neq p$, it happens that $\varphi\to_{\mathbb{R}^2} x$, that is, $\mathcal{N}_x\subseteq \varphi^{\uparrow}$. Since $\varphi$ is a net in $S$ it cannot happen that $x\in C$. Then $x\in L$. Now, if $x\in L$ and $x\neq p$ consider the constant net $\langle x\rangle_d$ is $S$. Otherwise, if $x=p$, since $D$ is dense there is a  net $\varphi\in\textsc{Nets}(D)$ such that $\varphi\to_{\mathbb{R}^2} x$  and $\varphi^\uparrow\# D$. It follows that $\varphi\to_{\lambda} x$. This proves that the adherence operator is not idempotent. The reader can verify that if a preconvergence is topological, then the closure operator is idempotent. This means that $\langle X,\lambda\rangle $ it is not topological.

\end{example}

Recall that a topological space is compact if and only if every open cover has a finite subcover. In the context of preconvergence spaces, what would compactness be? And what would an open cover be? Proposition \ref{compact} tells us how to characterize compactness it in terms of filter and nets.

\begin{definition}
	A preconvergence space $X$ is said to be \textbf{compact} if every net has a convergent subnet or, equivalently, if every proper filter is contained in a convergent proper filter.
\end{definition}

\begin{definition}
	Let $\langle X,L\rangle$ be a preconvergence space. A family $\mathcal{C}$ of subsets of $X$ is a \textbf{local convergence system} at $x\in X$ if for every net $\varphi\in\textsc{Nets}(X)$ such that $\varphi\to_L x$ there is $C\in\mathcal{C}$ such that $C\in\varphi^{\uparrow}$. We say that $\mathcal{C}$ is a \textbf{convergence system} if is a local convergence system for each $x\in X$.
\end{definition}

\begin{remark}
	The usual terminology is covering system, as in~\cite{ref4}. The choice of the term convergence system is due to the fact that, in some sense, the family is witnessing the convergence of nets.
\end{remark}

\begin{figure}[H]
	\centering
	\includegraphics[width=6cm]{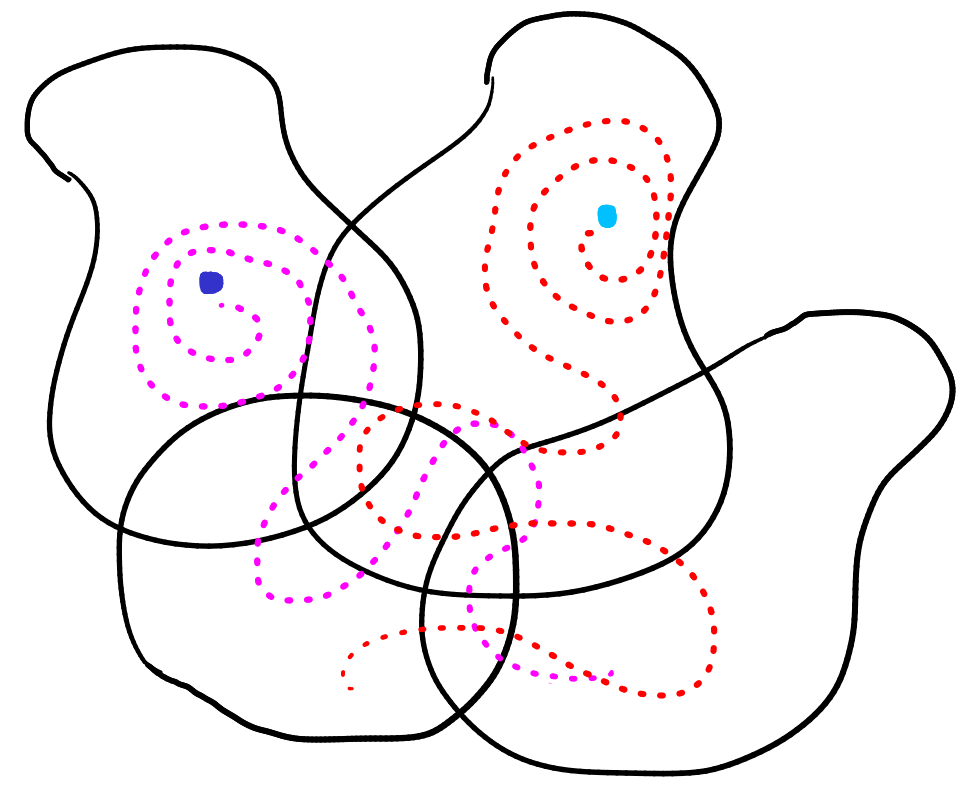}
	\caption{There is some element of the convergence system that contains a tail of the convergent nets.}
\end{figure}

\begin{example}
Every convergence system should be a cover, right? The problem is that this is not always true; for this, it is sufficient to have a centered preconvergence. Let $\langle X,L\rangle$  be a preconvergence space such that $L$ is centered and $\mathcal{C}$ be a convergence  system. For a point $x\in X$, since $L$ is centered, it happens that $\mathfrak{u}_x\to x$. Since $\mathcal{C}$ is a convergence system, it follows that there is $C\in\mathcal{C}$ such that $C\in\mathfrak{u}_x$, that is, $x\in C$. This proves that $\mathcal{C}$ is a cover.
\end{example}

\begin{remark}

Every convergence system should be a cover, right? The problem is that this is not always true; for this, we need a centered preconvergence. Let $\langle X,L\rangle$  be a preconvergence space such that $L$ is centered and $\mathcal{C}$ be a convergence  system. For a point $x\in X$, since $L$ is centered, it happens that $\mathfrak{u}_x\to x$. Since $\mathcal{C}$ is a convergence system, it follows that there is $C\in\mathcal{C}$ such that $C\in\mathfrak{u}_x$, that is, $x\in C$. This proves that $\mathcal{C}$ is a cover.
\end{remark}
\begin{proposition}
Let $f:X\to Y$ be a continuous function and $\mathcal{C}$ be a convergence system for $Y$. The family $f^{-1}(\mathcal{C})=\{f^{-1}[C]:C\in\mathcal{C}\}$ is a convergence system for $X$.
\end{proposition}

\begin{proof}
	Given a net $\varphi\in\textsc{Nets}(X)$ such that $\varphi\to x$, since $f$ is continuous, it happens that $f\circ \varphi\to f(x)$. Since $\mathcal{C}$ is a convergence system, there is $C\in\mathcal{C}$ such that $C\in(f\circ \varphi)^{\uparrow}$. Note that $f^{-1}[C]\in f^{-1}(\mathcal{C})\cap f^{-1}((f\circ \varphi)^{\uparrow})$ and $f^{-1}((f\circ \varphi)^{\uparrow})\subseteq \varphi^{\uparrow}$. It follows that $f^{-1}[C]\in\varphi^{\uparrow}$. This proves that $f^{-1}(\mathcal{C})$ is a convergence system.
\end{proof}

\begin{theorem}
	A convergence space $\langle X,L\rangle$ is compact if and only if every convergence system has a finite subcover.
\end{theorem}

\begin{proof}
	The implications will be proved by contraposition. Let $\mathcal{C}$ be a convergence system such that has no finite subcover. Notice that the family $$\hat{\mathcal{C}}=\{X\setminus C:C\in\mathcal{C}\}$$ has the finite intersection property. Indeed, if there is $\mathcal{S}\subseteq\hat{\mathcal{C}}$ finite such that $\bigcap \mathcal{S}=\emptyset$, we have $X=X\setminus \emptyset=X\setminus \bigcap \mathcal{S}=\bigcup_{S\in\mathcal{S}} X\setminus S$ and $X\setminus S\in\mathcal{C}$, contradicting $\mathcal{C}$ has no finite subcover. By Lemma \ref{ultrafilter}, there is an ultrafilter $\mathfrak{u}$ containing $\hat{\mathcal{C}}$. Notice that $\mathfrak{u}$ cannot converge. Otherwise, since $\mathcal{C}$ is a convergence system, there is $C\in\mathcal{C}\cap\mathfrak{u}$, contradicting  $X\setminus C\in\mathfrak{u}$. This proves that $X$ is not compact. Conversely, if $X$ is not compact there is an ultrafilter $\mathfrak{u}\in\textsc{Fil}^*(X)$ that does not converges. Let us show that $$\hat{\mathfrak{u}}=\{X\setminus U:U\in\mathfrak{u}\}$$ is a convergence system such that has no finite subcover.  If $\mathcal{F}\in\textsc{Fil}^*(X)$ is a filter such that $\mathcal{F}\to_L x$, then $\mathcal{F}\not\subseteq\mathfrak{u}$. Otherwise, since $L$ is isotone, $\mathfrak{u}\to_L x$. It follows that there is $F\in\mathcal{F}\setminus \mathfrak{u}$. Since $X\setminus F\in\mathfrak{u}$, it happens that $F=X\setminus(X\setminus F)\in\hat{\mathfrak{u}}$. Hence $F\in\mathcal{F}\cap\hat{\mathfrak{u}}$. This proves that $\hat{\mathfrak{u}}$ is a convergence system. Moreover, $\hat{\mathfrak{u}}$ has no finite subcover. Otherwise, there is $U_1,\cdots, U_n\in\mathfrak{u}$ such that $X=\bigcup_{1\leq i\leq n} (X\setminus U_i)=X\setminus (\bigcap_{1\leq i\leq n} U_i)$, then $\bigcap_{1\leq i\leq n} U_i=\emptyset \in\mathfrak{u}$, a contradiction.
\end{proof}

\begin{corollary}
	A topological space $\langle X, \tau\rangle $ is compact if and only if the limit space $\langle X,\lim_\tau\rangle$ is compact.

\end{corollary}

\begin{proposition}
	If $X$ is compact and $f: X\to Y$ is a continuous onto function, then $Y$ is compact.
\end{proposition}

\begin{proof}
	Let $\varphi\in\textsc{Nets}(Y)$ be a net. Since $f$ is onto, for each $a\in\hbox{dom}(\varphi)$ there is $x_a\in X$ such that $f(x_a)=\varphi_a$. By construction, $\langle x_a\rangle_{a\in\hbox{dom}(\varphi)}$ is a net in $X$. Since $X$ is compact, there is a subnet $\psi\in\textsc{Nets}(X)$ of $\langle x_a\rangle_{a\in\hbox{dom}(\varphi)}$ converging to some $x\in X$. The continuity of $f$ tells us that $f\circ \psi\to f(x)$. Notice that $f\circ\psi$ is a subnet of the net $\varphi=f\circ\langle x_a\rangle_{a\in\hbox{dom}(\varphi)}$. This proves that $Y$ is compact.
\end{proof}

\newpage

\

\newpage

\chapter{An introduction to ``Algebraic Convergence": Homotopy theory in limit spaces}
\label{chap3}
The goal of this chapter is to present the elementary concepts of homotopy theory in limit spaces. In Section \ref{sec3.1} we present the Pasting Lemma for limit spaces and construct the fundamental groupoid of a limit space in Section \ref{sec3.3}, adapting what is done in~\cite{BrownBook} and~\cite{Kammeyer} for topological spaces.

\section{Restriction and gluing of continuous functions}

Recall that we can restrict continuous functions between topological spaces and still obtain a continuous function, just as we can extend continuous functions. Pasting continuous functions is a bit more delicate, but it can still be done under certain conditions. In this section, we show that these results (all of them?) are valid for preconvergence spaces. We denote the restriction of a function $f:X\to Y$ to a subset $A\subseteq X$ by $f|_A:A\to Y$ and by $f|^B:X\to B$ the function such that $f|^B(x)=f(x)$ for every $x\in X$ and  $f[X]\subseteq B$.

\label{sec3.1}
\begin{proposition}
	Let  $f:\langle X,L\rangle\to\langle Y,L'\rangle$  be a continuous function. For each $S\subseteq X$ the function $f|_S:\langle S,L|_S\rangle \to \langle Y\,L'\rangle$ and $f|^{B}:\langle X,L\rangle\to \langle B,L'|_B\rangle$ whenever $f[X]\subseteq B\subseteq Y$ are also continuous. 
\end{proposition}

\begin{proof}
	Let $\varphi\in\textsc{Nets}(S)$ be a net such that $\varphi\to_{L|_S} x\in S$. It follows that $\varphi\to_L x$ and, since $f$ is continuous, $f\circ\varphi=f|_S\circ\varphi\to_{L'}f(x)$. This proves that $f|_S$ is continuous. Moreover, given a net $\psi\in\textsc{Nets}(X)$ such that $\psi\to_L x$, it happens that $f\circ\psi\to_{L'}f(x)$. Since $f\circ\psi[\hbox{dom}(\psi)]\subseteq f[X]\subseteq B$, it follows that $f\circ\psi=f|^B\circ\psi\to_{L'|_B}f(x)=f|^B(x)$. This proves that $f|^B$ is continuous.
\end{proof}

\begin{proposition}
	\label{315}
	Let $\langle X,L\rangle$ and $\langle Y,L'\rangle$ be preconvergence spaces and $f:X \to Y$ be a function.
	
	\begin{enumerate}
		\item If $f$ is continuous at every point of $S\subseteq X$, then $f|_S:\langle S,L|_S\rangle \to \langle Y, L'\rangle $ is continuous.
		
		\item If $S$ is $L$-open and $f|_S$ is continuous, then $f$ is continuous at every point of $S$.
	\end{enumerate}
\end{proposition}

\begin{proof}
	\
	\begin{enumerate}
		\label{symbol:greaterSet}
		
		\item If $f$ is continuous at $x\in S$ and $\varphi\in\textsc{Nets}(S)$ is a net such that $\varphi\to_{L|_S}x$, then $\varphi\to_{L}x$. It follows that $f\circ\varphi\to_{L'}f(x)$. Since $f\circ\varphi=f|_S\circ\varphi$ and $f(x)=f|_S(x)$ for every $x\in S$, this proves that $f|_S$ is continuous.
		\item Suppose that $S\subseteq X$ is $L$-open. Let $\varphi\in\textsc{Nets}(X)$ be a net such that $\varphi\to_L x\in S$. Since $S$ is $L$-open,  there is $a\in\hbox{dom}(\varphi)$ such that $\varphi[a^{\uparrow}]\subseteq S$, where $$a^{\uparrow}=\{d\in\hbox{dom}(\varphi):d\geq a\}$$ Notice that $a^{\uparrow}$ is a directed set, then $\varphi|_{a^{\uparrow}}\in\textsc{Nets}(S)$ is a subnet of $\varphi$. This means that $\varphi|_{a^{\uparrow}}\to_{L|_S}x$. It follows that $f|_S\circ\varphi|_{a^{\uparrow}}\to_{L'}f|_S(x)$. Since $f|_S\circ\varphi|_{a^{\uparrow}}=f\circ \varphi|_{a^{\uparrow}}$, $f|_S(x)=f(x)$ and $\varphi^{\uparrow}=\varphi|_{a^{\uparrow}}^{\uparrow}$, it happens that $f\circ\varphi\to_{L'}f(x)$. This proves that $f$ is continuous at every point of $S$.
	\end{enumerate}
	\end{proof}

\begin{definition}
	Let $X$ be a preconvergence space. A\textbf{ path} in $X$ is a continuous function $\gamma:[0,1]\to X$. If $\gamma(0)=x$ and $\gamma(1)=y$, we say that $\gamma$ is a path from $x$ to $y$. In the case that $x=y$ we say that $\gamma$ is a loop.
\end{definition}

Suppose that there is a path $\gamma$ from $x$ to $y$ and another path $\gamma'$ from $y$ to $z$ , there should be a path from $x $ to $z$, right? For topological spaces it is true. More precisely, this path would be given by the concatenation
\label{symbol:concatenation}
	$$	\begin{array}{rcl}
	\gamma *\gamma': [0,1] & \to     & X\\ 
	t & \mapsto & \begin{cases}
		\gamma(2t) & \hbox{ if } 0\leq t\leq \frac{1}{2} \\
		\gamma'(2t-1) & \hbox{ if } \frac{1}{2} \leq t\leq 1 \\
	\end{cases}  \\
\end{array}$$

\begin{figure}[H]	\centering
	\includegraphics[width=5cm]{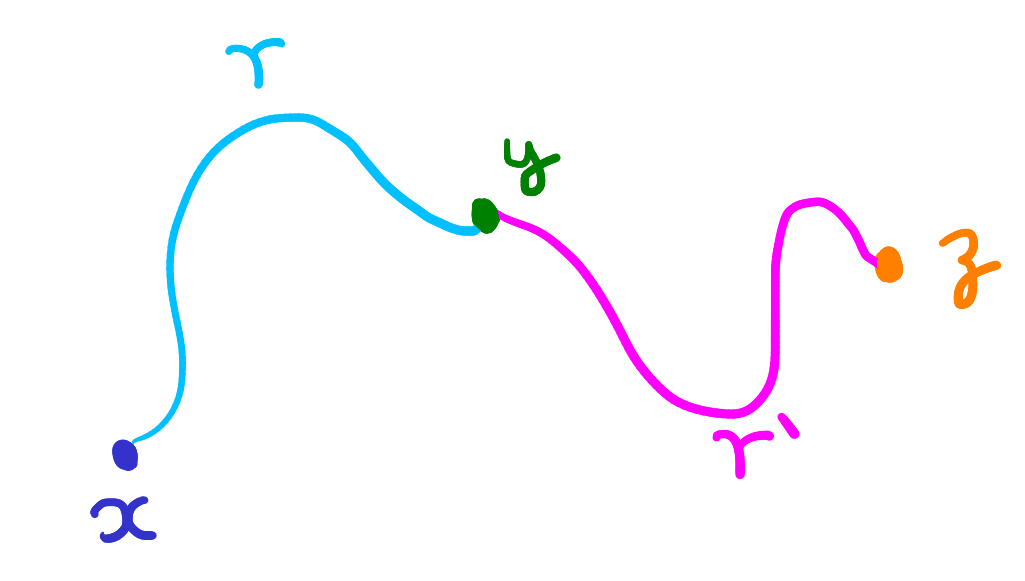}
	\caption{Intuitively, the idea of path gluing involves finding a way to traverse two paths within a certain time interval. For this to be done, one path must be completely traversed by the halfway point of the time interval.}
\end{figure}

The problem is that in the context of preconvergence spaces, this concatenation of paths is not always continuous. For topological spaces, what allows us to conclude continuity is the Pasting lemma, which states that the pasting of continuous functions defined on closed sets is continuous. Recall that in the Example \ref{215} we define a non-stable preconvergence on $[0,1]$. Consider the paths 

$$\begin{array}{rcl}
	\gamma : \langle [0,1],\to_{[0,1]}\rangle  & \to     & \langle [0,1], L\rangle  \\ 
	t & \mapsto & 	\frac{t}{2}\\
\end{array} \ \ \hbox{ and } \ \ \begin{array}{rcl}
	\gamma' : \langle [0,1],\to_{[0,1]}\rangle & \to     & \langle [0,1], L\rangle \\ 
	t & \mapsto & 	\frac{t}{2}+\frac{1}{2}\\
\end{array}$$

Notice that $\gamma*\gamma'=1_{[0,1]}:\langle [0,1],\to_{[0,1]}\rangle \to\langle [0,1],L\rangle $ is not continuous. So, in this case the Pasting lemma fails. Stability seems to be quite important for making the gluing of continuous functions work, and this is indeed true.

\begin{lemma}[Pasting lemma] Let $\langle X,L\rangle $ and $\langle Y,L'\rangle$ be limit spaces and $f\colon X\to Y$ be a function. If $A$ and $B$ are $L$-closed subsets of $X$ such that $ X=A\cup B$ and  the restrictions $f|_A:\langle A,L|_A\rangle\to\langle Y,L'\rangle$ and $f|_B:\langle B,L|_B\rangle\to \langle Y,L'\rangle$ are continuous, then $f$ is continuous.

\end{lemma}

\begin{proof}
Let $\varphi\in\textsc{Nets}(X)$ be a net such that $\varphi\to_L x\in X$. If $x\in X\setminus B$, since $X\setminus B$ is $L$-open, there is $a\in\hbox{dom}(\varphi) $ such that $\varphi [a^{\uparrow}]\subseteq X\setminus B$. Then $\varphi|_{a^{\uparrow}}$ is a net in $A$ and $\varphi|_{a^{\uparrow}}\to_L x$ because it is a subnet of $\varphi$. It follows that $\varphi|_{a^{\uparrow}}\to_{L|_A} x$. Since $f|_A$ is continuous, it happens that $f|_A\circ \varphi|_{a^{\uparrow}}\to_{L'}f|_A(x)=f(x)$. Notice that $f\circ\varphi$ is a subnet of $f|_A\circ\varphi|_{a^{\uparrow}}$, then $f\circ \varphi\to_{L'} f(x)$. This proves that $f$ is continuous at $x$. The case in which $x\in X\setminus A$ is analogous.  On the other hand, if $x\in A\cap B$, consider the nets

$$\begin{array}{rcl}
	\psi: \hbox{dom}(\varphi)  & \to     &  X  \\ 
	d & \mapsto &  \begin{cases}
		\varphi_d  & \hbox{ if }\varphi_d\in A \\
		x  & \hbox{ if } \varphi_d\notin A
	\end{cases} \\
\end{array} \ \ \hbox{ and } \ \ \begin{array}{rcl}
	\gamma : \hbox{dom}(\varphi)& \to     & X \\
	d & \mapsto & \begin{cases}
		\varphi_d  & \hbox{ if }\varphi_d\in B \\
		x  & \hbox{ if } \varphi_d\notin B
	\end{cases} 
\end{array}$$

By construction $\psi$ is a net in $A$ and $\gamma$ is a net in $B$. Notice that $\psi$ and $\gamma$ are mixings of $\varphi$ and a constant net which converges to $x$, then $\psi,\gamma\to_L x$.  This means that $\psi\to_{L|_A} x$ and $\gamma\to_{L|_B} x$. Since $f|_A$ and $f|_B$ are continuous, it happens that $f|_A\circ\psi \to_{L'} f|_A(x)=f(x)$ and $f|_B\circ\gamma\to_{L'} f|_B(x)=f(x)$. Notice that $f\circ\varphi$ is a mixing of $f|_A\circ \psi$ and $f|_B\circ\gamma$. Since $Y$ is a limit space, it follows that $f\circ\varphi \to_{L'} f(x)$. This proves that $f$ is continuous.
\end{proof}

\begin{figure}[H]
	\centering
	\includegraphics[width=8cm]{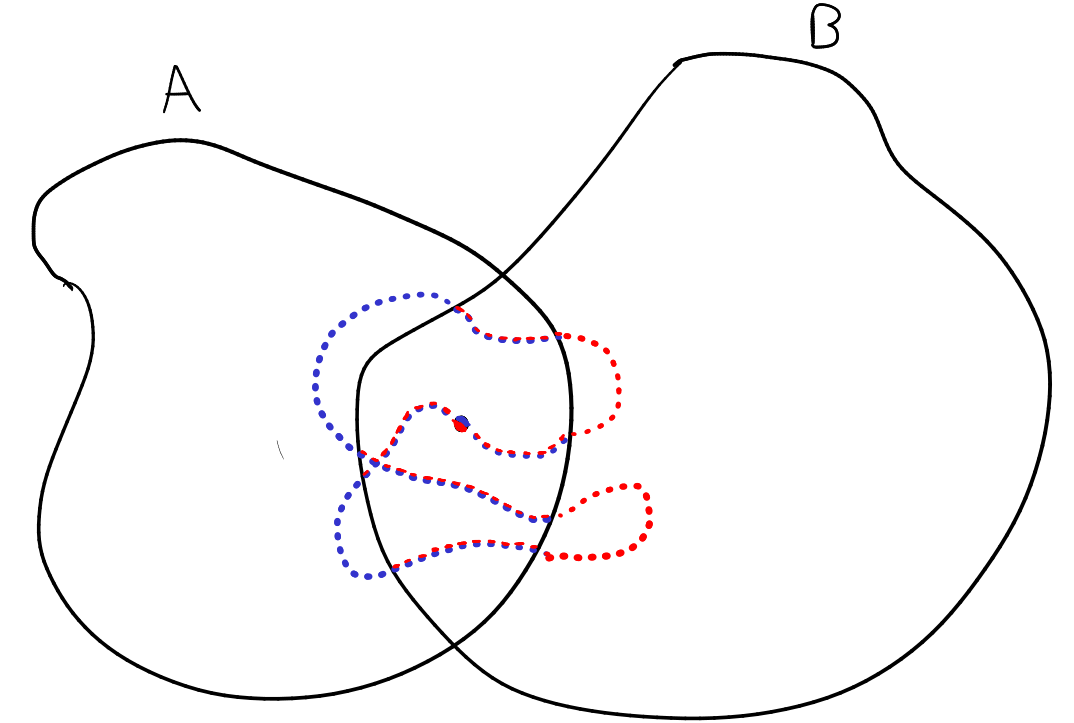}
\caption{The idea of the last proof was to collapse to the limit point all the point in $A \cap B$ in order to obtain two nets, one in $A$ (blue) and the other in $B$ (red).}
\end{figure}

\begin{remark}
In~\cite{absolutepaper}, we present a more general version of the Pasting Lemma, requiring that $X$ be covered by a locally finite family. The version presented by Preuss~\cite{ref5} assumes that the cover is finite and that $X$ is pretopological. On the other hand, the version by Dolecki and Mynard~\cite{schechter1996handbook} assumes that $Y$ is pseudotopological. There is yet another version, where Dossena~\cite{dossena} assumes that both $X$ and $Y$ are pseudotopological. 
\end{remark}

\begin{corollary}
	Let $X$ be a limit space. If $\gamma:[0,1]\to X$ is a path from $x$ to $y$ and $\gamma':[0,1]\to X$ is a path from $y$ to $z$, then $\gamma*\gamma':[0,1]\to X$ is path from $x$ to $z$.
\end{corollary}

\begin{proof}
	Note that $[0,\frac{1}{2}]$ and $[\frac{1}{2},1]$ are closed sets of $[0,1]$ and $[0,1]=[0,\frac{1}{2}]\cup[\frac{1}{2},1]$. Since $\gamma* \gamma'|_{[0,\frac{1}{2}]}$ and $\gamma* \gamma'|_{[\frac{1}{2},1]}$ are continuous and $X$ is a limit space, from the Pasting Lemma the gluing $\gamma*\gamma'$ is continuous.
\end{proof}
\section{Homotopy}

The traditional definition of homotopy in the context of topological spaces is the following: given continuous functions $f,g\colon X\to Y$ a homotopy between $f$ and $g$ is a continuous function $H: [0,1]\times X\to Y$, with $[0,1]$ carrying its standard topology, such that $H(0,x)=f(x)$ and $H(1,x)=g(x)$ for every $x\in X$. It is common to think of a homotopy as a path in the space of continuous functions, which is not entirely legitimate due to the lack of an exponential object in the category of topological spaces. However, we saw in Proposition \ref{exp} that the category of preconvergence spaces has an exponential object. Therefore, in this context, there is no theoretical gap in considering a homotopy as a path in the space of continuous functions. Thus, the definition of homotopy for preconvergence spaces is as follows.

\begin{definition}
	A \textbf{homotopy} between continuous functions $f,g:X\to Y$ is a continuous function $H:[0,1]\to\mathcal{C}(X,Y)$, with $[0,1]$ carrying its standard topological convergence, such that $H(0)=f$ and $H(1)=g$. If there is an homotopy between f and g we denote it by $f\simeq g$.
\end{definition}

\begin{remark}

Suppose that the definition of homotopy were the traditional one. Considering $Z=[0,1]$ in Proposition \ref{exp} we have a unique continuous function $\tilde{H}:[0,1]\to\mathcal{C}(X,Y)$ such that $\hbox{ev}\circ(\tilde{H}\times 1_X)=H$. Notice that $$H(t,x)=\hbox{ev}(\tilde{H}\times 1_X(t,x))=\hbox{ev}(\tilde{H}(t),x)=\tilde{H}(t)(x)$$ for every $(t,x)\in [0,1]\times X$. In particular, $\tilde{H}(0)(x)=f(x)$ and $\tilde{H}(1)(x)=g(x)$ for every $x\in X$. This tells us that $\tilde{H}(0)=f$ and $\tilde{H}(1)=g$. Then exponentiality allows us to work with both definitions.

\end{remark}

\begin{proposition}
	Homotopy is an equivalence relation.
\end{proposition}

\begin{proof}
	\
	\begin{enumerate}[i)]
		
	\item (Reflexivity) Let $f:X\to Y$ be a continuous function, since $\mathcal{C}(X,Y)$ is centered, the constant function $H:[0,1]\to\mathcal{C}(X,Y)$ such that $H(t)=f$ for all $t\in[0,1]$ is continuous. This proves that $f\simeq f$.
	
\item (Symmetry)	If $f\simeq  g$, there is a homotopy $H:[0,1]\to\mathcal{C}(X,Y)$ between $f$ and $g$ . The function $h:[0,1]\to [0,1]$ such that $h(t)=1-t$ is continuous. It follows that the composition $H\circ h: [0,1]\to\mathcal{C}(X,Y)$ is a continuous function such that $H\circ h(0)=g$ and $H\circ h(1)=f$. Then $H\circ h$ is a homotopy between $g$ and $f$. This proves that $g\cong f$.
	
\item (Transitivity)	If $f\simeq g$ and $g\simeq  h$, there are homotopies $H,G:[0,1]\to\mathcal{C}(X,Y)$ such that $H(0)=f$, $H(1)=G(0)=g$ and $G(1)=h$. By the Pasting lemma, $H*G$ is a homotopy betweeen $f$ and $h$. This proves that $f\cong h$.

\end{enumerate}

\end{proof}

\begin{proposition}
	Let $f,g:X\to Y$ be homotopic functions and $h:Y\to Z$ be a continuous function. The compositions $h\circ f$ and $ h\circ g$ are homotopic. In others words, the composition by continuous function preserves homotopy.
	\label{comp}
\end{proposition}

\begin{proof}
	Let $H:[0,1]\to\mathcal{C}(X,Y)$ be a homotopy between $f$ and $g$. Let us show that
	
		$$	\begin{array}{rcl}
		\Psi:[0,1] & \to     &\mathcal{C}(X,Z)\\ 
		 t & \mapsto & h\circ H(t) \\
	\end{array}$$ is a homotopy between $h\circ f$ and $h\circ g$. Notice that $\Psi$ is a composition of continuous functions. Indeed, the following diagram commutes
	
\[\begin{tikzcd}[ampersand replacement=\&,cramped]
	[0,1] \& {\mathcal{C}(X,Y)\times\mathcal{C}(Y,Z)} \\
	\& {\mathcal{C}(X,Z)}
	\arrow["{H\times \backvec{h}}", from=1-1, to=1-2]
	\arrow["\Psi"', from=1-1, to=2-2]
	\arrow["\circ", from=1-2, to=2-2]
\end{tikzcd}\] Where $\backvec{h}:\mathcal{C}(X,Y)\to\mathcal{C}(X,Z)$ is the function such that $\backvec{h}(f)=h\circ f$ for all continuous function $f\in\mathcal{C}(X,Y)$. Since $\Psi(0)=h\circ H(0)=h\circ f$ and $\Psi(1)=h\circ H(1)=h\circ g$, it follows that $\Psi$ is a homotopy between $h\circ f$ and $h\circ g$.
\end{proof}

\begin{example}
 The homotopy category $\textsc{HoLim}$ is the category in which the objects are limit spaces and an arrow is the equivalence class $[f] $ of all continuous functions which are homotopic to $f:X\to Y$. The composition is given by
	
	$$\begin{array}{rcl}
		\circ:\hbox{Hom}(X,Y)\times \hbox{Hom}(Y,Z) & \to     & \hbox{Hom}(X,Z)\\ 
		\langle [f],[g]\rangle &\mapsto & [g\circ f]\\
	\end{array}$$ where $\hbox{Hom}(X,Y)$ denotes the set of arrows from $X$ to $Y$.	Notice that this composition is well-defined. If $f\simeq f'$ and $g\simeq g'$, there are homotopies $G:[0,1]\to\mathcal{C}(X,Y)$ and $H:[0,1]\to\mathcal{C}(Y,Z)$ between $f$ and $f'$ and $g$ and $g'$, respectively. The function

	\[\begin{split} K:[0,1]&\to \mathcal{C}(X,Z)\\
		t&\mapsto H(t)\circ G(t)\end{split}\] is a composition  of continuous functions
		
		\[\begin{tikzcd}[ampersand replacement=\&,cramped]
			{[0,1]} \& {\mathcal{C}(X,Y)\times \mathcal{C}(Y,Z)} \\
			\& {\mathcal{C}(X,Z)}
			\arrow["{G\times H}", from=1-1, to=1-2]
			\arrow["K"', from=1-1, to=2-2]
			\arrow["\circ", from=1-2, to=2-2]
		\end{tikzcd}\] Then $K$ is continuous. Moreover, it happens that $K(0)=H(0)\circ G(0)=g\circ f$ and $K(1)=H(1)\circ G(1)=g'\circ f'$. This means that $K$ is a homotopy between $g\circ f$ and $g'\circ f'$. The reader can easily verify that this composition is associative and that the identity functions are the identities in this category. For topological spaces, this category is important because it allows the study of topological spaces in a more abstract and flexible way. Intuitively, it identifies spaces that have the same ``shape" or homotopy structure, even if they are not exactly the same. For more details see~\cite{hatcher}.
		 
\end{example}

\section{The Fundamental Groupoid}

\label{sec3.3}

Our work in Algebraic Topology, or in this context ``Algebraic Convergence", begins now. The Fundamental Group of a topological space classifies all paths that start and end at the same point, considering equivalent those loops that are homotopics. In many situations, a space can be decomposed into smaller subspaces. The goal is to find methods to compute the fundamental group of the space based on the fundamental groups of its subspaces. However, since the fundamental group depends on the path component of the base point, any decomposition theorem must consider path connectedness. To simplify this, we generalize the concept of fundamental group to the notion of fundamental groupoid. This will allow us to prove the Seifert van-Kampen theorem in the groupoid version more clearly in the Chapter \ref{chap4}. The following construction is an adaptation for the context of limit spaces of what is done in~\cite{BrownBook} and~\cite{Kammeyer} for topological spaces.

\begin{definition}
	Let $\gamma,\gamma': [0,1]\to X$  be paths in a preconvergence space $X$ with the same end points, say $x$ and $y$. We say that $\gamma$ and $\gamma'$ are \textbf{homotopic relative} to its end points, abbreviated as rel-homotopic, if there is a homotopy $H\colon [0,1]\to \mathcal{C}([0,1],X)$ from $\gamma$ to $\gamma'$ such that $H(t)$ is a path from $x$ to $y$ for every $t\in [0,1]$. We write $\gamma\simeq_{x,y} \gamma'$ to indicate that $\gamma$ and $\gamma'$ are rel-homotopic.
\end{definition}

\begin{figure}[H]
	\centering
	\includegraphics[width=10cm]{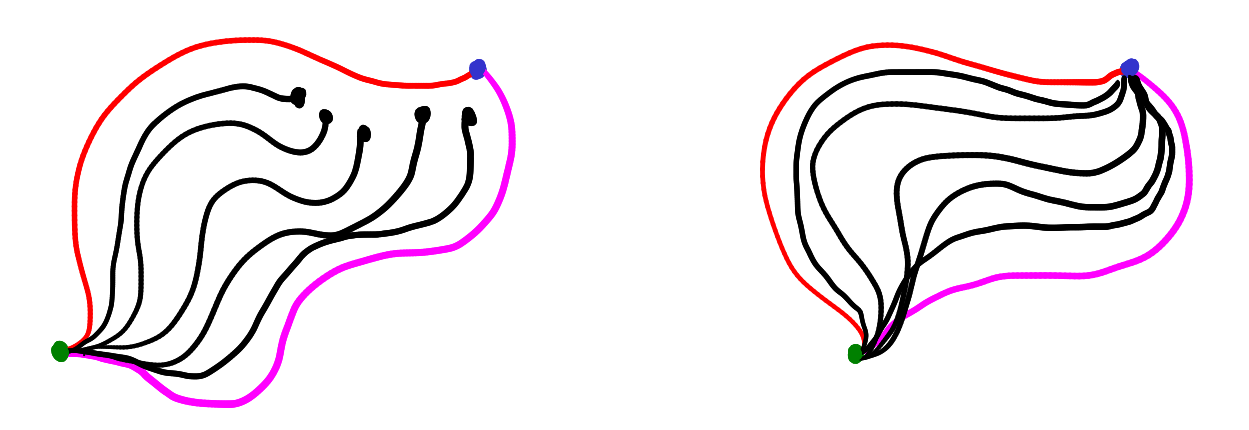}
	\caption{The diference between rel-homotopy and homotopy: On the left, we have a homotopy between two paths, red and pink. Note that not all intermediate paths have the same end points. On the right, a relative homotopy between the red and pink paths, where all intermediate paths have the same end points.}
\end{figure}
\label{symbol:groupoid}

Recall that a groupoid is a category in which all arrows are isomorphisms. We denote by $\textsc{Groupoid}$ the category of groupoids. Groupoids are used in theory of orbifolds, which are topological spaces that locally resemble the quotient of a Euclidean space by a finite group action. They provide a categorical structure that helps to understand the global properties of these spaces. Moreover, they also appear in the theory of covering spaces. In~\cite{BrownBook}, Brown discuss Algebraic Topology using groupoids. Our goal now is to construct the Fundamental Groupoid of a limit space $X$, which we denote by $\Pi(X)$. Of course, we must declare the objects and arrows.  Objects are the points of $X$. An arrow from $x$ to $y$ is a $\simeq_{x,y}$-equivalence class $[\gamma]$ of a path $\gamma\in\mathcal{C}([0,1],X)$ from $x$ to $y$. In this category the composition is given by

	$$\begin{array}{rcl}
	\circ:\Pi(X)\times \Pi(X) & \to     & \Pi(X)\\ 
	\langle [\rho],[\gamma]\rangle &\mapsto & [\gamma*\rho]\\
\end{array}$$ For the fundamental groupoid $\Pi(X)$ to be truly a category, we must prove that this operation is well-defined, is associative and has identities.  Given paths $\gamma,\gamma',\rho$ and $\rho'$ such that $\gamma\simeq_{x,y}\gamma'$ and $\rho\simeq_{y,z}\rho'$. Let us show that  $\gamma*\rho\simeq_{x,z}\gamma'*\rho'$. Let $G\colon [0,1]\to\mathcal{C}([0,1],X)$ be a rel-homotopy between $\gamma$ and $\gamma'$. Notice that
\[\begin{split} K:[0,1]&\to \mathcal{C}([0,1],X)\\
	t&\mapsto G(t)* H(t)\end{split}\]
 is a rel-homotopy between $\gamma*\rho$ and $\gamma'*\rho'$. Indeed, $K$ is a function such that $K(t)$ is path from $x$ to $z$ for each $t\in [0,1]$ and  $K(0)=\gamma*\rho$ and $K(1)=\gamma'*\rho'$. It remains to show that $K$ is continuous.
For a net $\langle x_a\rangle_a$ in $[0,1]$ such that $x_a\to x$, we must show that $\langle G_a*H_a\rangle_a$ is a net in $\mathcal{C}([0,1],X)$ such that $G(x_a)*H(x_a)\to G(x)*H(x)$. If $\langle y_b\rangle_b$ is a net in $[0,1]$ such that $y_b\to y$. Notice that the function

 $$	\begin{array}{rcl}
 	 [0,1]\times [0,1]& \to     & [0,1]\\ 
 	\langle t,s\rangle  & \mapsto & G(t)*H(t)(s)\\
 \end{array} $$
is continuous. It follows that  $G(x_a)*H(x_a)(y_b)\to G(x)*H(x)(y)$. This proves that $K$ is continuous. For  associativity, let $[\gamma]: w\to x$, $[\gamma']: x\to y$ and $[\gamma'']\colon y\to z$ be arrows, notice that
 $$	\begin{array}{rcl}
 	\gamma*(\gamma'*\gamma''): [0,1] & \to     & X\\ 
 	t & \mapsto & \begin{cases}
 		\gamma(2t)&\text{ if }t\in [0,\frac{1}{2}]\\
 		\gamma'(4t-2)&\text{ if }t\in [\frac{1}{2},\frac{3}{4}]\\
 		\gamma''(4t-3)&\text{ if }t\in[\frac{3}{4},1]\end{cases}  \\
 \end{array} $$ and  $$	\begin{array}{rcl}
 (\gamma*\gamma')*\gamma'': [0,1] & \to     & X\\ 
 t & \mapsto & \begin{cases}
 	\gamma(4t)&\text{ if }t\in [0,\frac{1}{4}]\\
 	\gamma'(4t-1)&\text{ if }t\in [\frac{1}{4},\frac{1}{2}]\\
 	\gamma''(2t-1)&\text{ if }t\in[\frac{1}{2},1]\end{cases}  \\
 	\end{array} $$
Consider a continuous pasting of the representatives of the classes without reparametrizations

$$	\begin{array}{rcl}
	\Gamma: [0,3] & \to     & X\\ 
	t & \mapsto & \begin{cases}
		\gamma(t)&\text{ if }t\in [0,1]\\
		\gamma'(t-1)&\text{ if }t\in [1,2]\\
		\gamma''(t-2)&\text{ if }t\in[2,3]\end{cases}  \\
\end{array} $$
There are homeomorphisms $p_0,p_1:[0,1]\to [0,3]$ such that $\Gamma\circ p_0=\gamma*(\gamma'*\gamma'')$ and $\Gamma\circ p_1=(\gamma*\gamma')*\gamma''$. Explicitly, this homeomorphisms are given by

$$	\begin{array}{rcl}
	p_0: [0,1] & \to     & [0,3]\\ 
	t & \mapsto & \begin{cases}
		4t &\text{ if }t\in [0,\frac{1}{2}]\\
		2t+1&\text{ if }t\in [\frac{1}{2},1] \end{cases} \end{array} \\  \\\ \hbox{ and } \\\ \\		
 \begin{array}{rcl}
		p_1: [0,1] & \to     & [0,3]\\ 
		t & \mapsto & \begin{cases}
			2t &\text{ if }t\in [0,\frac{1}{2}]\\
			4t-1&\text{ if }t\in [\frac{1}{2},1]  \end{cases}  \\
\end{array} $$
Note that $p_0$ and $p_1$ are paths from $0$ to $3$. There is a rel-homotopy $H: [0,1]\to\mathcal{C}([0,1],[0,3])$ between $p_0$ and $p_1$, then
\[\begin{split}\Phi\colon [0,1]&\to\mathcal{C}([0,1],X)\\
	t&\mapsto \Gamma\circ H(t)\end{split}\]
is a rel-homotopy between  $\gamma*(\gamma'*\gamma'')$ and $(\gamma*\gamma')*\gamma''$. Then $[(\gamma*\gamma')*\gamma'']=[\gamma*(\gamma'*\gamma'')]$. The identities arrows are given by the classes of constant paths. Indeed, let $\gamma:[0,1]\to X$ be a path from $x$ to $y$. We denote by $\hat{x}$ and $\hat{y}$ the constant path with image $x$ and $y$ respectively. For each $t\in[0,1]$ consider the continuous function 

\[\begin{split}H_t\colon [0,1]&\to X\\
	s&\mapsto \begin{cases} 
		x & \text{if } 0 \leq s \leq \frac{1-t}{2}, \\
		\gamma\left( \frac{2s+t-1}{t+1} \right) & \text{if } \frac{1-t}{2} \leq s\leq 1 \\
		\end{cases}\end{split}\]

Notice that $H_0=\gamma*\hat{x}$ and $H_1=\gamma$, then the function $H:[0,1]\to \mathcal{C}([0,1],X)$ such that $H(t)=H_t$ is a rel-homotopy between $\gamma*\hat{x}$ and $\gamma$.  To see that $\hat{y}*\gamma$ is rel-homotopic to $\gamma$ for each $t\in [0,1]$, consider the continuous function 

\[\begin{split}H'_t\colon [0,1]&\to X\\
	s&\mapsto \begin{cases} 
		\gamma(\frac{2s}{t+1}) & \text{if } 0 \leq s \leq \frac{t+1}{2}, \\
		 y & \text{if } \frac{t+1}{2} \leq s\leq 1 \\
\end{cases}\end{split}\]
Note that $H'_0=\hat{y}*\gamma$ and $H'_1=\gamma$. Moreover $H'_t$ is a path from $x$ to $y$ for each $t\in [0,1]$. This means that the function $H':[0,1]\to\mathcal{C}([0,1],X)$ such that $H'(t)=H'_t$ is a rel-homotopy between $\hat{y}*\gamma$ and $\gamma$. Now we show that every arrow in $\Pi(X) $ is an isomorphism. Composing every arrow $[\gamma]:x\to y \in \Pi(X)$ with $[\gamma\circ \rho=\backvec{\gamma}]:y\to x\in \Pi (X)$, where 

\[\begin{split}\rho\colon [0,1]&\to [0,1]\\
 t&\mapsto 1-t\end{split}\]
 We see that $[\gamma*\backvec{\gamma}]=[\hat{y}]$ and $ [\backvec{\gamma}*\gamma]=[\hat{x}]$. Indeed, for each $t\in[0,1]$ consider the path

 \[\begin{split}H_t\colon [0,1]&\to X \\
 	s&\mapsto \begin{cases} 
 		\gamma(2ts) & \text{ if } 0 \leq s \leq \frac{1}{2}, \\
 		\gamma(2t(1 - s)) & \text{ if } \frac{1}{2} \leq s \leq 1.
 		\end{cases} \end{split}\]
Notice that the function $H: [0,1]\to \mathcal{C}([0,1], X)$ such that $H(t)=H_t$ for each $t\in [0,1]$ is a rel-homotopy between $\hat{x}$ and $\backvec{\gamma}*\gamma$. Similarly, we show that there is rel-homotopy between $\gamma*\backvec{\gamma}$ and $\hat{y}$.
Just as happens with the topological fundamental groupoid, the construction above induces a functor $\Pi\colon\Lim\to\Groupoid$ which sends a continuous function $f\colon X\to Y$ to the functor
\[\begin{split}\Pi(f)\colon\Pi(X)&\to \Pi(Y)\\
	[\gamma]&\mapsto [f\circ \gamma]\end{split}\]
which is well-defined by Proposition \ref{comp}. Notice that the $\Lim$-$\textsc{Groupoid}$ functor above extends the usual $\Top$-$\textsc{Groupoid}$ functor
\[\begin{tikzcd}[cramped]
	\Lim && {\Groupoid} \\
	\\
	\Top
	\arrow["\Pi", from=1-1, to=1-3]
	\arrow["i", from=3-1, to=1-1]
	\arrow["\Pi"', from=3-1, to=1-3]
\end{tikzcd}\]

\begin{remark}
	The reader may be thinking that to compute the fundamental groupoid of a limit space, it suffices to compute the fundamental groupoid of the topological modification. In Example \ref{discret}, we show that this is not the case. But first, we need a lemma.
\end{remark}
\begin{proposition}
	Let $X$ and $Y$ be topological spaces and $f:X\to Y$ be a function. If $X$ is compact and the image $f[X]$ is uncountable, there is $x\in X$ such that $f[V]$ is uncountable for every neighborhood $V\subseteq X$ of $x$.

\end{proposition}
\begin{proof}
	Suppose that for every $x\in X$ there is a neighborhood $V_x\subseteq X$ of $x$ such that $f[V_x]$ is countable. Let $U_x\subseteq V_x\subseteq X$ be an open set such that $x\in X$, we have $X=\bigcup_{x\in X} U_x$. Since $X$ is compact, it follows that there are $x_0,\cdots,x_n\in X$ such that $X=\bigcup_{0\leq i\leq n} U_{x_i}$. Notice that $f[X]=\bigcup_{0\leq i\leq n} f[U_{x_i}]$ is countable. This contradiction tells us that there is $x\in X$ such that $f[V]$ is uncountable for every neighborhood $V\subseteq X$ of $x$.
\end{proof}

\begin{lemma}
	\label{path}
	 If $\gamma\colon [0,1]\to\mathbb{R}$ is a non-constant path, then there is $t\in [0,1]$ such that $\gamma[V]$ is uncountable for every neighborhood $V\subseteq [0,1]$ of $t$.
	 \end{lemma}

\begin{proof}

Since $[0,1]$ is compact and connected and $\gamma$ is non-constant, the image $\gamma[0,1]$ is a non-degenerated closed interval. In particular, is uncountable. The result follows from the previous proposition. 
\end{proof}

\begin{example}
	\label{discret}
	Recall that in Example \ref{seq} we define a sequential convergence $\hbox{Seq}$ on the real line. Now we show that $\Pi(\langle \mathbb{R},\hbox{Seq}\rangle)$ is discrete in the categorical sense, that is, a category whose only morphisms are the identities. Assume that there is a non-constant path $\gamma:[0,1]\to \langle \mathbb{R},\hbox{Seq}\rangle$. Notice that $\gamma:[0,1]\to\langle \mathbb{R},\to_{\mathbb{R}}\rangle$ is a non-constant path. By Lemma \ref{path} there is $t\in [0,1]$ such that $\gamma [V]$ is uncountable for every neigborhood $V\subseteq[0,1]$ of $t$. Let  $\varphi\in\textsc{Nets}([0,1])$ be a net such that $\varphi^{\uparrow}=\mathcal{N}_t$, it happens that $\varphi\to_{\mathbb{R}} t$ but $\gamma\circ \varphi\not\to_{\hbox{Seq}} \gamma (t)$. Indeed, the tails of $\gamma\circ\varphi$ are uncountable, then $\gamma\circ\varphi$ cannot be subnet of a sequence. This contradiction tells us that every path in $\langle \mathbb{R},\hbox{Seq}\rangle $ is constant. In particular, $\Pi(\langle \mathbb{R},\hbox{Seq}\rangle)\neq\Pi(\langle \mathbb{R},\to_{\mathbb{R}}\rangle)$.

\end{example}

A pointed limit space is a pair $\langle X,x_0\rangle$, where $X$ is a limit space and $x_0\in X$ is a point of $X$. To obtain the fundamental group of a pointed limit space $\langle X,x_0\rangle$, we simply put $\pi_1(X,x_0)=\Pi(X)[x_0,x_0]$, where $\Pi(X)[x_0,x_0]$ is the set of equivalence classes of loops around $x_0\in X$. This coincides precisely with the classical definition presented in~\cite{Kammeyer}.  \label{symbol:fundGroup}

\begin{example}
	 The fundamental group $\pi_1( X, p)$ of the lollipop of Example \ref{lollipop} is the group of integers $\mathbb{Z}$. A loop $\gamma:[0,1]\to \langle X,\lambda\rangle$ at $p\in X$ cannot intercept $S$. Otherwise, there is $t\in [0,1]$ such that $\gamma(t)\in S$. Notice that $\mathcal{O}(\lambda)$ it is a topology finer than the usual subspace topology of the plane. Since $\gamma:[0,1]\to \langle X,\mathcal{O}(\lambda)\rangle$ is continuous, it follows that $\gamma:[0,1] \to X$ in the usual topological sense. Moreover $[0,t]$ is connected, then the image $\gamma[0,t]$ is connected. It follows that $\gamma[0,t]\cap S\cap V\neq\emptyset$ for every neighborhood $V\in\mathcal{N}_p$ of $p$, otherwise $\gamma[0,t] $  would be disconnected. Since $D$ is countable, for all $n\in\mathbb{N}$ there is $\gamma(x_n)\in (S\setminus D)\cap B(p,\frac{1}{2^n})$. Since $\langle \gamma(x_n)\rangle_n$ is a sequence in $S\setminus D$, then $\gamma(x_n)\not\to_{\lambda} \gamma(0)=p$. Notice that $x_n\to 0$  and $\langle \gamma(x_n)\rangle_n$ is a sequence in $S\setminus D$, then $\gamma(x_n)\not\to_{\lambda} \gamma(0)=p$. Contradicting $\gamma$ being continuous. Hence $\gamma:[0,1]\to C$ is a loop at $p$ in $C$. This means that $\pi_1(X,p)=\pi_1(C,p)$. Notice that the preconvergence in $C$ is the usual in the plane, then $\pi_1(C,p)$ coincides with the topological fundamental group and the result follows.
	 
	 \begin{figure}[H]
		\centering
	 	\includegraphics[width=7cm]{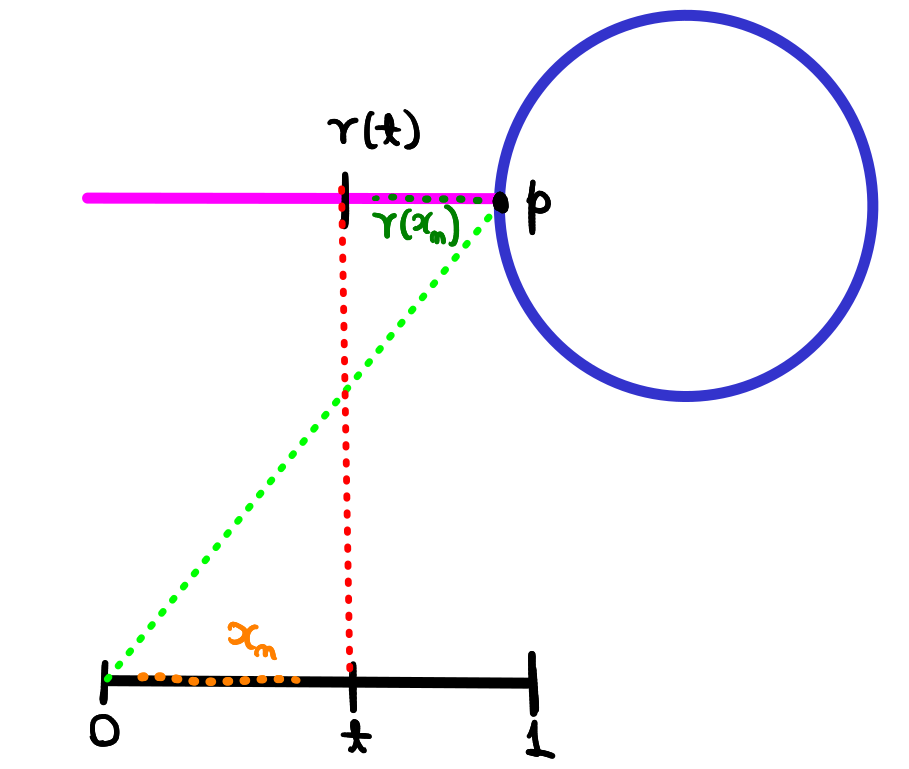}
	 	\caption{An illustration of the presented argument.}
	 \end{figure}
\end{example}

\begin{proposition}
	Let $X$ and $Y$ be limit spaces. The fundamental group $\Pi(X\times Y)$ is isomorphic to the product $ \Pi(X)\times \Pi(Y)$. In particular, $\pi_1(X\times Y,\langle x_0,y_0\rangle)$ is isomorphic $\pi_1(X,x_0)\times \pi_1(Y,y_0)$ for every points $x_0\in X$ and $y_0\in Y$.
	\label{gprod}
\end{proposition}

\begin{proof}
	Notice that 
	
\[\begin{tikzcd}[ampersand replacement=\&,cramped]
	{\Pi(X\times Y)} \&\& {\Pi(X)\times \Pi(Y)} \\
	{\langle x_0,y_0\rangle} \&\& {\langle x_0,y_0\rangle} \\
	{\langle x_1,y_1\rangle} \&\& {\langle x_1,y_1\rangle}
	\arrow[from=1-1, to=1-3]
	\arrow[maps to, from=2-1, to=2-3]
	\arrow["{[\gamma]}"', from=2-1, to=3-1]
	\arrow["{[\pi_X\circ\gamma]}"', curve={height=12pt}, from=2-3, to=3-3]
	\arrow["{[\pi_Y\circ \gamma]}", curve={height=-12pt}, from=2-3, to=3-3]
	\arrow[maps to, from=3-1, to=3-3]
\end{tikzcd}\]
it is an isomorphism of categories. Indeed, this is a functor because $$(\pi_X\circ\gamma)*(\pi_X\circ\gamma')=\pi_X\circ(\gamma*\gamma') \ \hbox{ and } \ (\pi_Y\circ\gamma)*(\pi_Y\circ\gamma')=\pi_Y\circ(\gamma*\gamma')$$ for paths $\gamma,\gamma':[0,1]\to X\times Y$. The proof that this functor is an isomorphism of categories is the same as in the classical topological case, see~\cite{vick}.
	
\end{proof}

\begin{example}
	Let us apply the previous proposition to find $\pi_1 (X\times X,\langle p,p\rangle )$, where $X$ is the lollipop of Example \ref{lollipop}. Since $\pi_1(X,p)=\mathbb{Z}$, it follows that $\pi_1( X\times X,\langle p,p\rangle)$ is isomorphic to $ \pi_1(X,p)\times \pi_1(X,p)=\mathbb{Z}\times \mathbb{Z}$. This means that $\pi_1 (X\times X,\langle p,p\rangle )$ is $\mathbb{Z}\times \mathbb{Z}$.
	This illustrates the usefulness of Proposition \ref{gprod}, because this cartesian product is geometrically strange (see  Figure \ref{prod}) and it could be challenging to calculate the fundamental group directly from some intuition.
	
	\begin{figure}[H]
		\centering
		\includegraphics[width=7cm]{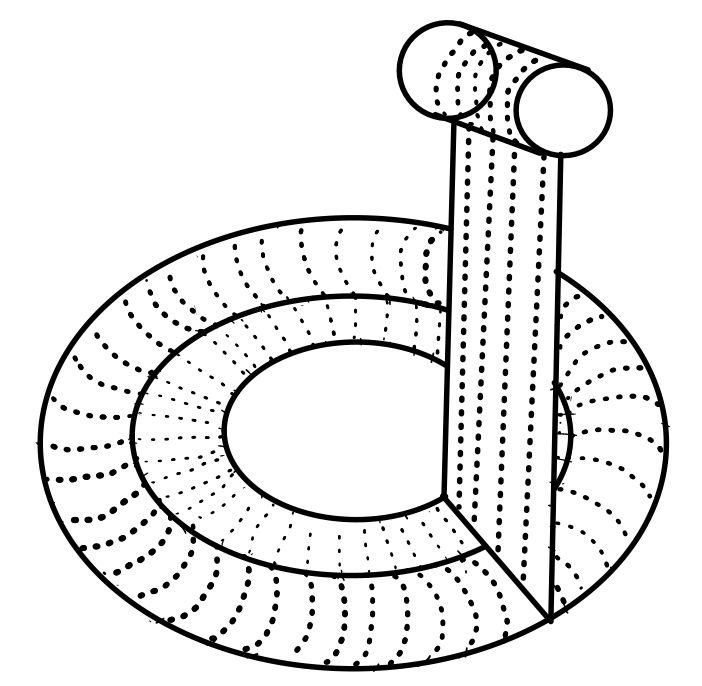}
		\caption{To obtain this space, the idea is to attach a lollipop at each point of the lollipop.}
		\label{prod}
	\end{figure}
\end{example}

\newpage

\ 
 
\newpage

\chapter{The Seifert-Van Kampen Theorem for fundamental groupoids of limit spaces}

\label{chap4}

The proof of the Seifert-van Kampen Theorem for fundamental groupoids of limit spaces presented in this chapter is an adaptation of the one found in~\cite{Kammeyer}. The most significant modification being the use of a convergence system rather than an open cover. First, we need to recall the concept of a colimit. 	Let \(\mathcal{C}\) be a category and \(F: \mathcal{J} \to \mathcal{C}\) a functor. A cocone with vertex \(C\) for \(F\) is a collection of morphisms \(\{ f_j: F(j) \to C \}_{j \in \mathcal{J}}\) such that, for every morphism \(f: j \to k\) in \(\mathcal{J}\), we have \(f_k \circ F(f) = f_j\). A colimit of \(F\) is a universal cocone, i.e., a cocone \(\{ f_j: F(j) \to C \}_{j \in \mathcal{J}}\) such that, for any other cocone \(\{ g_j: F(j) \to N \}_{j \in \mathcal{J}}\), there is a unique morphism \(v: C \to N\) such that \(v \circ f_j = g_j\) for all \(j \in \mathcal{J}\). Colimits are essential because they provide a powerful language to describe and study universal constructions.

\begin{lemma}[Lebesgue]
	Let $H\colon K\to Z$ be a continuous function from a compact metric space $K$ to a convergence space $Z$. If $\mathcal{C}$ is a convergence system for $Z$, then there is a real number $\delta>0$ such that every subset $A\subseteq K$ of diameter less than $\delta$ is contained in some $C\in\mathcal{C}$.
\end{lemma}

\begin{proof}
	The proof is the same as in the topological case; see~\cite{Kammeyer}.
\end{proof}

\begin{theorem}[van-Kampen--Groupoid Version for limit spaces]
	Let $X$ be a limit space and $\mathcal{O}$ be a convergence system of $X$, which is closed under finite intersections. Consider $\mathcal{O}$ as a small category with morphisms given by inclusions. The colimit of the functor 
	\[\begin{tikzcd}[ampersand replacement=\&,cramped]
		{\mathcal{O}} \& {\textsc{Groupoid}} \\
		U \& {\Pi(U)} \\
		V \& {\Pi(V)}
		\arrow["{\Pi|_{\mathcal{O}}}", from=1-1, to=1-2]
		\arrow[maps to, from=2-1, to=2-2]
		\arrow[ from=2-1, to=3-1]
		\arrow[ from=2-2, to=3-2]
		\arrow[maps to, from=3-1, to=3-2]
	\end{tikzcd}\] is the fundamental groupoid $\Pi(X)$.
\end{theorem}
 
\begin{proof}
	Since $\mathcal{O}$ is closed under finite intersections, the diagram commutes 
	\[\begin{tikzcd}[ampersand replacement=\&,cramped]
		{\Pi(U)} \& {\Pi(X)} \\
		{\Pi(U\cap V)} \& {\Pi(V)}
		\arrow[from=1-1, to=1-2]
		\arrow[from=2-1, to=1-1]
		\arrow[from=2-1, to=1-2]
		\arrow[from=2-1, to=2-2]
		\arrow[from=2-2, to=1-2]
	\end{tikzcd}\]
	for each $U,V\in\mathcal{O}$. This means that $\langle \Pi(U)\to\Pi(X)\rangle_{U\in\mathcal{O}}$ is a cocone in $\Pi_{|\mathcal{O}}$. Let us show that it is universal. Let $\langle F_U:\Pi(U)\to\mathcal{G}\rangle _{U\in\mathcal{O}}$ be a cocone, where $\mathcal{G}$ is a groupoid, our goal is to show that there is a unique functor $F:\Pi(X)\to\mathcal{G}$ such that the diagram commutes 
	
	\[\begin{tikzcd}[ampersand replacement=\&,cramped]
		{\Pi(U)} \& {\mathcal{G}} \\
		{\Pi(X)}
		\arrow["{F_U}", from=1-1, to=1-2]
		\arrow[from=1-1, to=2-1]
		\arrow["F"', from=2-1, to=1-2]
	\end{tikzcd}\]
	On objects we define $F$ by $F(x)=F_U(x)$ if $x\in U$ and $U\in\mathcal{O}$. We must show that it is well-defined. Since $\langle F_U:\Pi(U)\to\mathcal{G}\rangle _{U\in\mathcal{O}}$ is a cocone, the diagram commutes
	

	\[\begin{tikzcd}[ampersand replacement=\&,cramped]
		{\Pi(U)} \\
		{\Pi(U\cap V)} \& {\mathcal{G}} \\
		{\Pi(V)}
		\arrow["{F_U}", curve={height=-6pt}, from=1-1, to=2-2]
		\arrow[from=2-1, to=1-1]
		\arrow["{F_{U\cap V}}"', from=2-1, to=2-2]
		\arrow[from=2-1, to=3-1]
		\arrow["{F_V}"', curve={height=6pt}, from=3-1, to=2-2]
	\end{tikzcd}\]
	It follows that $F_U(x)=F_{U\cap V}(x)=F_V(x)$, where $V,U\in\mathcal{O}$ is such that $x\in U\cap V$. Now we define $F$ in the arrows. Let $[\gamma]:x\to y$ be an arrow between $x$ and $y$ in $\Pi(X)$. Choose a Lebesgue number $\delta$ for the convergence system $\{\gamma^{-1}[U]:U\in\mathcal{O}\}$ of the compact metric space $[0,1]$. Subdividing $[0,1]$ into $n$ intervals of diameter less than $\delta$, we see that $\gamma=\gamma_1*\gamma_2*\cdots*\gamma_n$ is a concatenation of $n$ paths $\gamma_1,\cdots,\gamma_n$, where the image of $\gamma_i$ is contained in some $U_i\in\mathcal{O}$ for $i=1,\cdots, n$, see Figure \ref{conc}. Then we define $F([\gamma])=F_{U_n}([\gamma_n])\circ\cdots \circ F_{U_1}([\gamma_1])$. It is remains show that this construction is well-defined.

	\begin{figure}[H]
		\centering
		\includegraphics[width=6cm]{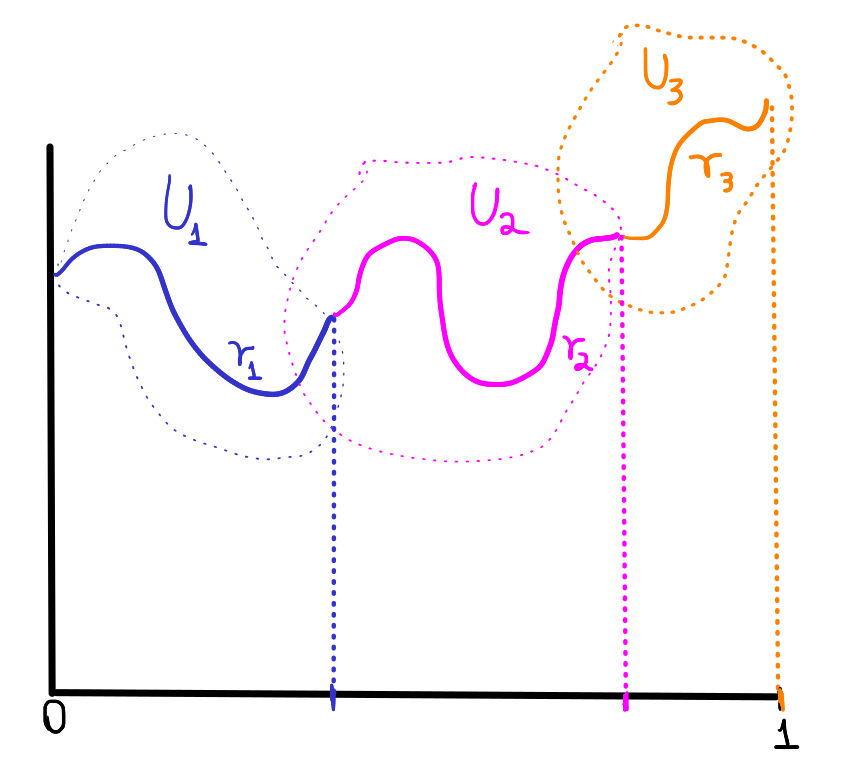}
		\caption{An illustration of how we can see 
			$\gamma$ as a concatenation.}
		\label{conc}
	\end{figure}

	\begin{enumerate}
		\item Let us show that $F([\gamma])$ does not depend on the subdivision and the choice of the sets $U_i$. For this, it is enough to prove the case $\gamma=\gamma_1*\gamma_2$, where the images of $\gamma,\gamma_1$ and $\gamma_2$ are contained is $U_0,U_1$ and $U_2$, respectively. The proof of the general case follows by induction. By commutativity of the cocone diagram $\langle F_U:\Pi(U)\to\mathcal{G}\rangle_{U\in\mathcal{O}}$ and the functoriality of $F_{U_0}$, it happens that

		$$F_{U_2}([\gamma_2])\circ F_{U_1}([\gamma_1])=F_{U_2\cap U_0}([\gamma_2])\circ F_{U_1\cap U_0}([\gamma_1])=F_{U_0}([\gamma_2])\circ F_{U_0}(\gamma_1)=$$
		
		$$=F_{U_0}([\gamma_1*\gamma_2])=F_{U_0}([\gamma])$$
		
		\item Let $H:[0,1]\times [0,1]\to X$ be a rel-homotopy from $\gamma$ to $\gamma'$.  Since $\{H^{-1}(U)\}_{U\in\mathcal{O}}$ is a convergence system for the compact metric space $[0,1]\times [0,1]$ we can choose a Lebesgue number $\delta$ for this convergence system. Consider $[0,1]\times [0,1]$ with the maximum metric, then we subdivide the square \([0,1] \times [0,1]\) into $n$ squares of diameter $\frac{1}{n}$ where $n\in\mathbb{N}^*$ is such that $\frac{1}{n}<\delta$. The argument that follows is guided by the intuition of the following figure.
		
		\begin{figure}[H]
			\centering
			\includegraphics[width=13cm]{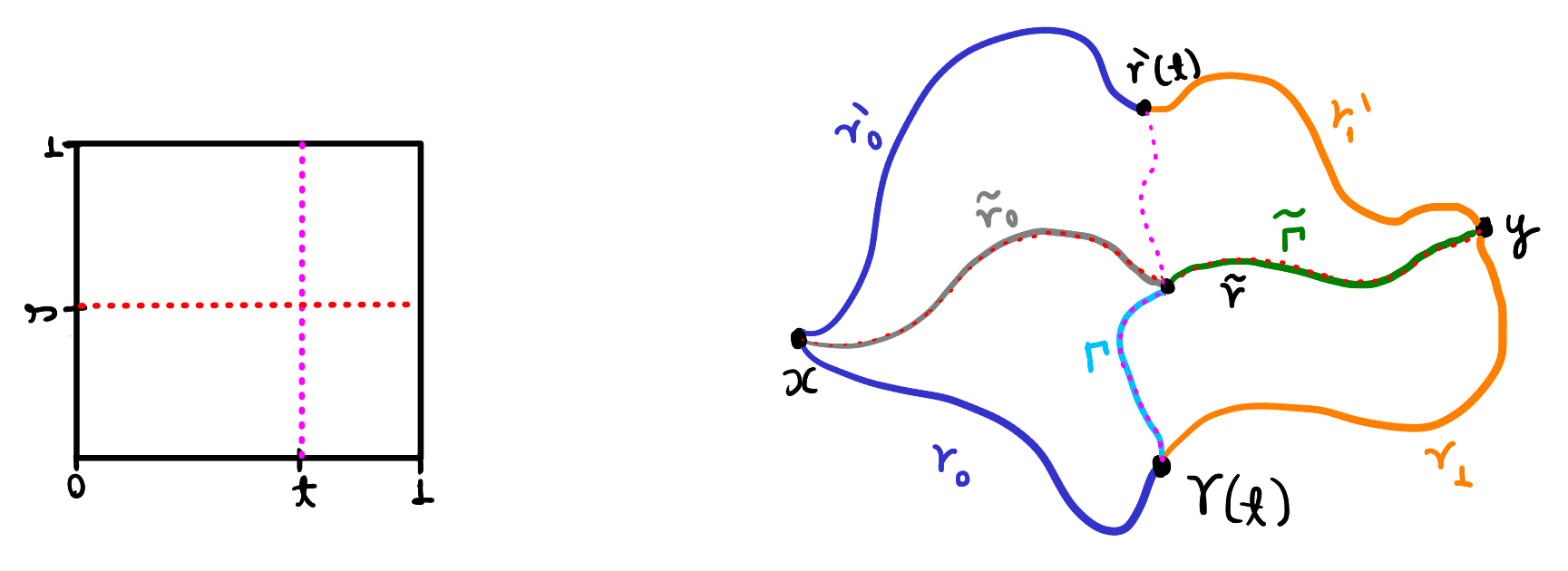}
			\caption{The idea is move the path \(\gamma\) to \(\gamma'\) by homotopies relative to the endpoints through the squares.
			}
			\label{migue}
		\end{figure}

		Consider $\gamma=\gamma_0*\gamma_1$  and $\gamma'=\gamma'_0*\gamma'_1$ where the image of $\gamma_0$ is contained in $U$, the image of $\gamma_1$ is contained in $V$, the image of $\gamma'_0$ is contained in $U'$ and the image of $\gamma'_1$ is contained in $V'$. Let $t\in [0,1]$ such that $\gamma(t)$ is the endpoint of $\gamma_0$ and $s\in [0,1] $ such that $t\leq s$  and the diameter of $[0,t]\times [0,s]$, $[t,1]\times [0,s]$, $[0,t]\times [s,1]$ and $[t,1]\times [s,1]$ are less than $\delta$, it happens that $H[[0,t]\times [0,s]]\subseteq U$,   $H[[t,1]\times [0,s]]\subseteq U'$,  $H[[0,t]\times [s,1]]\subseteq V$ and  $H[[t,1]\times [s,1]]\subseteq V'$. Consider the path

		\[\begin{split}\tilde{\gamma}\colon [0,1]&\to X \\
			r&\mapsto \begin{cases} 
				H(r,t) & \text{ if } 0 \leq r \leq s \\
				H(s,r) & \text{ if } s \leq r \leq 1
		\end{cases} \end{split}\]
		
		Notice that $\tilde{\gamma}$ is a path in $V$ and $\tilde{\gamma}\cong_{\gamma (t),y} \gamma_1$. It follows that $F_V([\tilde{\gamma}])=F_V([\gamma_1])$. Write $\tilde{\gamma}$ as a concatenation $\tilde{\gamma}=\Gamma*\tilde{\Gamma}$. Consider a path $\tilde{\gamma}_0$  in $U$ that starts at $x$ and ends at the endpoint of $\Gamma$ constructed in a way similar to $\tilde{\gamma}$ such that $\gamma_0*\Gamma$ is rel-homotopic to $\tilde{\gamma}_0$. Since $\tilde{\Gamma}$ is rel-homotopic to $\backvec{\Gamma}*\gamma_1$, it follows that $$F([\tilde{\gamma}_0*\tilde{\Gamma}])=F_V([\tilde{\Gamma}])\circ F_U([\tilde{\gamma}_0])=F_V([\backvec{\Gamma}*\gamma_1])\circ F_U([\gamma_0*\Gamma])=F_V([\gamma_0*\Gamma*\backvec{\Gamma}*\gamma_1])=F([\gamma])$$Proceeding with the same argument, we will arrive at $$F([\gamma])=F_V([\gamma_1])\circ F_U([\gamma_0])=F_{V'}([\gamma'_1])\circ F_{U'}([\gamma'_0])$$ This means that $F$ is well-defined on arrows.

	\end{enumerate}
	Moreover, by construction, $F$ is a functor that factorizes the cocone $\langle F_U:\Pi(U)\to \mathcal{G}\rangle_{U\in\mathcal{O}}$ over the
	cocone $\langle \Pi(U)\to\Pi(X) \rangle_{U\in\mathcal{O}}$ and it is unique with this property.

\end{proof}

\begin{example}
	In Example \ref{discret}, we directly showed that the fundamental groupoid of $\langle \mathbb{R},\hbox{Seq}\rangle$ is discrete. Now we will use the van-Kampen theorem for this. Notice that for a real number $x_0\in \mathbb{R}$ the family $$\mathcal{N}=\{N\cup \{x_0\}:N\subseteq \mathbb{R}\hbox{ is countable} \}$$ is a convergence system closed under finite intersections. Indeed, 
	
	\begin{enumerate}[i)]
		\item If $\varphi\in\textsc{Nets}(\mathbb{R})$ is a net such that $\varphi\to_{\hbox{Seq}} x$. There is a sequence $s:\mathbb{N}\to\mathbb{R}$ such that $s\to_{\mathbb{R}} x$ and $s^{\uparrow}\subseteq\varphi^{\uparrow}$. The set $N=s[\mathbb{N}]$ is countable, then $N\cup\{x_0\}\in\mathcal{N}$. Moreover, there is $d'\in\hbox{dom}(\varphi)$ such that $\varphi[d'^{\uparrow}]\subseteq N\subseteq N\cup\{x_0\}$. It follows that $N\cup \{x_0\}\in\varphi^{\uparrow}$. This means that $\mathcal{N}$ is a convergence system. 
		
		\item For $N\cup\{x_0\},M\cup\{x_0\}\in\mathcal{N}$, notice that $$(N\cup \{x_0\})\cap( M\cup\{x_0\})=(N\cup M)\cup \{x_0\}\in\mathcal{N}$$Then $\mathcal{N}$ is closed under finite intersections.
	\end{enumerate}
	Since each $N'\in\mathcal{N}$ is countable, a path $\gamma:[0,1]\to N'$ is constant. Otherwise $\gamma:[0,1] \to \langle N',\mathcal{O}(\hbox{Seq}|_{N'})$ is a non-constant path, where $\mathcal{O}(\hbox{Seq}|_{N'}) $ in the usual subspace topology. Your image should be an non-degenerated interval, but that cannot happen because $N'$ does not contain non-degenerate open intervals. This means that $\Pi (N')$ is discrete of each $N'\in\mathcal{N}$. Follows from Seifert-van Kampen Theorem that $\Pi(X)=\hbox{colim}_{N'\in\mathcal{N}}\Pi(N')$. This proves that $\Pi(X)$ is discrete.
\end{example}

\newpage

\ 

\newpage

\chapter{Some remarks and future works}

\label{chap5}

In summary, this work illustrates how to ``extend" Algebraic Topology to something that might be called ``Algebraic Convergence". The simplicity and clarity of the proofs presented here demonstrate how the use of nets is beneficial to convergence theory. Additionally, we also observed that the category of limit spaces prevails over the category of topological spaces. However, it is worth mentioning that the category of limit spaces may be too general for future works. Compactness is fundamental in Topology, and for this purpose, the category of pseudotopological spaces is more convenient, see~\cite{schechter1996handbook, ref3}. In~\cite{riesernew}, Rieser shows that the categories of pseudotopological and limit spaces admit cofibration category structures, and that the category of pseutopological spaces admits a model
category structure. In particular, he shows that pseudotopological suspensions coincide with their topological counterparts when applied to spheres. This ensures that many classical results in the homotopy theory of topological spaces can be transferred to pseudotopological spaces. Our choice of limit spaces was motivated by the minimal hypotheses needed for our current purposes. We may pursue research in the following directions:
	\begin{enumerate}[i)]
		\item A universal cover of a connected topological space $X$ is a simply connected space $Y$ with a continuous function $f:Y\to X$ that is a covering map. A natural step would be to extend or adapt results concerning universal coverings to convergence spaces. This was recently proposed by Marroquín~\cite{marroquin}, using the language of filters. This work addresses connectedness in convergence spaces, but it is worth noting that the definition presented is the same as in~\cite{ref3}, which means that the topological modification is a connected topological space.
		
		\item   More recently, Milićević and Scoville~\cite{MS} discussed how to further develop singular homology and higher homotopy groups in the category of pseudotopological spaces.

\item  Recall that a sheaf \( \mathcal{F} \) on a topological space \( \langle X,\tau\rangle\) is a pair $\langle X,\mathcal{F}(U)_{U\in \tau}\rangle$ where \( \mathcal{F}(U) \) is a set, called the set of sections over \( U \), and to each inclusion of open sets \( V \subseteq U \) there is a restriction \( \rho_{U,V} : \mathcal{F}(U) \rightarrow \mathcal{F}(V) \), satisfying the following properties
	
	\begin{enumerate}
		\item  For any open set \( U \), the restriction \( \rho_{U,U} \) is the identity on \( \mathcal{F}(U) \), and if \( W \subseteq V \subseteq U \), then \( \rho_{U,W} = \rho_{V,W} \circ \rho_{U,V} \).
		
		\item  If \( \{ U_i \}_{i\in\mathcal{I}}\) is an open cover for an open set \( U \) and \( s_i \in \mathcal{F}(U_i) \) is a section for each \( i \in\mathcal{I}\) such that \( \rho_{U_i, U_i \cap U_j}(s_i) = \rho_{U_j, U_i \cap U_j}(s_j) \) for all \( i, j\in\mathcal{I} \), then there is a unique section \( s \in \mathcal{F}(U) \) such that \( \rho_{U, U_i}(s) = s_i \) for each \( i\in\mathcal{I} \).
	\end{enumerate}

We could argue that the generality of limit spaces demands a less restrictive notion of a fundamental group or groupoid. Considering the alternatives presented by Kennison~\cite{JK}, we can address the following problem: What would be the definition of a sheaf for limit spaces? Could this be used to describe fundamental groups of limit spaces?

\end{enumerate}

\newpage

\

\newpage

\nocite{*} 
\bibliographystyle{plain} 
\bibliography{refTCC.bib} 

\mbox{}

\end{document}